\documentclass[12pt]{amsart}
\usepackage{amssymb}
\usepackage{amsfonts}
\usepackage{latexsym}
\usepackage{amscd}
\usepackage[mathscr]{euscript}
\usepackage{xy} \xyoption{all}
\usepackage[active]{srcltx}

\usepackage{tikz}

\usepackage{color,graphics}

\newcommand{\N}{{\mathbb{N}}}
\newcommand{\Z}{{\mathbb{Z}}}

\newcommand{\C}{{\mathbb{C}}}

\newcommand{\ol}{\overline}
\newcommand{\uloopr}[1]{\ar@'{@+{[0,0]+(-4,5)}@+{[0,0]+(0,10)}@+{[0,0] +(4,5)}}^{#1}}
\newcommand{\uloopd}[1]{\ar@'{@+{[0,0]+(5,4)}@+{[0,0]+(10,0)}@+{[0,0]+ (5,-4)}}^{#1}}
\newcommand{\dloopr}[1]{\ar@'{@+{[0,0]+(-4,-5)}@+{[0,0]+(0,-10)}@+{[0, 0]+(4,-5)}}_{#1}}
\newcommand{\dloopd}[1]{\ar@'{@+{[0,0]+(-5,4)}@+{[0,0]+(-10,0)}@+{[0,0 ]+(-5,-4)}}_{#1}}

\newcommand{\calV}{{\mathcal V}}

\newcommand{\gotr}{{\mathfrak r}}

\newcommand{\alphaxi}{{\alpha
^{x_i}(x_1,\dots , \widehat{x_i},\dots ,x_{k_u})}}

\newcommand{\Lab}{L^{{\rm ab}}}

\newcommand{\luloop}[1]{\ar@'{@+{[0,0]+(-8,2)}@+{[0,0]+(-10,10)}@+{[0, 0]+(2,2)}}^{#1}}

\newcommand{\dotedge}{\ar@{.}}
\newcommand{\eqedge}{\ar@{=}}

\newcommand{\FSGr}{\mathbf{FSGr}}

\newcommand{\Mon}{\mathbf{Mon}}

\newcommand{\Ehat}{\hat{E}}

\newcommand{\mon}[1]{\calV(#1)} 

\newcommand{\So}{{\rm Source}}

\newcommand{\bgast}{\mbox{\Large$*$}}

\numberwithin{equation}{section}

\theoremstyle{plain}
\newtheorem{theorem}{Theorem}[section]
\newtheorem{lemma}[theorem]{Lemma}
\newtheorem{proposition}[theorem]{Proposition}
\newtheorem{corollary}[theorem]{Corollary}

\theoremstyle{definition}
\newtheorem{definition}[theorem]{Definition}
\newtheorem{example}[theorem]{Example}

\newtheorem{remark}[theorem]{Remark}
\newtheorem{remarks}[theorem]{Remarks}
\newtheorem*{remark*}{Remark}

\newtheorem*{assumption*}{Assumption}

\newtheorem{construction}[theorem]{Construction}
\newtheorem{hypothesis}[theorem]{Hypothesis}
\newtheorem{notation}[theorem]{Notation}

\begin{document}

\null\vskip-1cm

\title[Dynamical systems]{Dynamical systems associated to separated graphs, graph algebras, and paradoxical decompositions}%
\author{Pere Ara}
\address{Departament de Matem\`atiques, Universitat Aut\`onoma de Barcelona,
08193 Bellaterra (Barcelona), Spain.} \email{para@mat.uab.cat}
\author{ Ruy Exel}
\address{Departamento de Matem\'atica, Universidade Federal de Santa Catarina,
88010-970 Florian\'opolis SC, Brazil.}\email{exel@mtm.ufsc.br}
\thanks{The first named author was partially supported by DGI MICIIN-FEDER
MTM2011-28992-C02-01, and by the Comissionat per Universitats i
Recerca de la Generalitat de Catalunya. The second named author was
partially supported by CNPq.} 
\date{\today}

\maketitle

\begin{abstract}
We attach to each finite bipartite separated graph $(E,C)$ a partial
dynamical system $(\Omega (E,C), \mathbb  F, \theta)$, where $\Omega
(E,C)$ is a zero-dimensional metrizable compact space, $\mathbb F$
is a finitely generated free group, and $\theta $ is a continuous
partial action of $\mathbb F$ on $\Omega (E,C)$. The full crossed
product C*-algebra $\mathcal O (E,C)= C(\Omega (E,C)) \rtimes
_{\theta^*} \mathbb F$ is shown to be a canonical quotient of the
graph C*-algebra $C^*(E,C)$ of the separated graph $(E,C)$.
Similarly, we prove that, for any $*$-field $K$, the algebraic
crossed product $\Lab _K(E,C) = C_K(\Omega (E,C))\rtimes
_{\theta^*}^{\text{alg}} \mathbb F$ is a canonical quotient of the
Leavitt path algebra $L_K(E,C)$ of $(E,C)$. The monoid $\mon{\Lab
_K(E,C)}$ of isomorphism classes of finitely generated projective
modules over $\Lab _K (E,C)$ is explicitly computed in terms of
monoids associated to a canonical sequence of separated graphs.
Using this, we are able to construct an action of a finitely
generated free group $\mathbb F$ on a zero-dimensional metrizable
compact space $Z$ such that the type semigroup $S(Z,\mathbb F,
\mathbb K)$ is not almost unperforated, where $\mathbb K$ denotes
the algebra of clopen subsets of $Z$. Finally we obtain a
characterization of the separated graphs $(E,C)$ such that the
canonical partial action of $\mathbb F$ on $\Omega (E,C)$ is
topologically free.
\end{abstract}

\tableofcontents

\section{Introduction}
\label{sect:intro}

In their seminal paper \cite{CK}, J. Cuntz and W. Krieger gave a
dynamical interpretation of the Cuntz-Krieger algebras associated to
square $\{ 0,1\}$-matrices in terms of the corresponding subshifts
of finite type. Since then, there has been a long and fruitful
interaction between the theory of combinatorially defined
C*-algebras and dynamics, see
  \cite{WatatanibyRuy},
  \cite{KPPRByRuy},
  \cite{KPR},
  \cite{ExelLaca},
  \cite{KPActionsByRuy},
  \cite{BPRSByRuy},
  \cite{FLRByRuy},
  \cite{BHRSByRuy},
  \cite{PatGraphByRuy},
  \cite{TomfordeByRuy},
  \cite{RaeSzyByRuy},
  \cite{PTWByRuy},
  \cite{RSYByRuy},
  \cite{MuhlyByRuy},
  \cite{CombinByRuy}.

\medskip

In  \cite[Section 5]{ExelAmena}, the second named author describes
any Cuntz-Krieger C*-algebra as a crossed product of a commutative
C*-algebra by a \emph{partial action} of a non-abelian free group. The same
applies to graph C*-algebras, which are generalizations of
Cuntz-Krieger algebras.

\medskip

A wider class of graph algebras has been introduced in \cite{AG} and
\cite{AG2}. These algebras are attached to {\it separated graphs},
where a separated graph is a pair $(E,C)$ consisting of a directed
graph $E$ and a set $C=\bigsqcup _{v\in E^0} C_v$, where each $C_v$
is a partition of the set of edges whose terminal vertex is $v$.
Their associated graph C*-algebras $C^*(E,C)$ and Leavitt path
algebras $L_K(E,C)$ (for an arbitrary field $K$), provide
generalizations of the usual graph C*-algebras (\cite{Raeburn}) and
Leavitt path algebras (\cite{AA1}, \cite{AMP}) associated to
directed graphs, although these algebras behave quite differently
from the usual graph algebras because the range projections
associated to different edges need not commute. One motivation for
their introduction was to provide graph-algebraic models for the
Leavitt algebras $L_K(m,n)$ of \cite{lea}, and the C*-algebras
$U^{\text{nc}}_{m,n}$ studied L. Brown \cite{Brown} and  McClanahan
 \cite{McCla1}, \cite{McCla2}, \cite{McCla3}. Another motivation
was to obtain graph algebras whose structure of projections is as
general as possible. Thanks to deep  results due to G. Bergman
\cite{Berg1} \cite{Berg2}, it is possible  to show that the natural
map $M(E,C) \to \mon{L_K(E,C)}$ from a monoid $M(E,C)$ naturally
associated to $(E,C)$ to the monoid $\mon{L_K(E,C)}$ of isomorphism
classes of finitely generated projective right modules over
$L_K(E,C)$ is an isomorphism, see \cite[Theorem 4.3]{AG}.  Since
each conical abelian monoid $M$ is isomorphic to a monoid of the
form $M(E,C)$ for a suitable separated graph $(E,C)$
(\cite[Proposition 4.4]{AG}), we conclude that Leavitt path algebras
of separated graphs realize all the conical abelian monoids. (By the
results of \cite{AMP}, this is not true for the class of Leavitt
path algebras of non-separated graphs.) The first author and Ken
Goodearl raised in \cite{AG2} the problem of whether the natural map
$M(E,C)\to \mon{C^*(E,C)}$ is an isomorphism (see also \cite[Section
6]{Aone-rel} for a brief discussion on this problem).

\medskip

Recall that a set $S$ of partial isometries in a $*$-algebra
$\mathcal A$ is said to be {\it tame} \cite[Proposition
5.4]{ExelPJM} if every element of $U=\langle S\cup S^*\rangle$, the
multiplicative semigroup generated by $S\cup S^*$, is a partial
isometry, and the set $\{ e(u)\mid u\in U \}$ of final projections
of the elements of $U$ is a commuting set of projections  of
$\mathcal A$. A main difficulty in working with $C^*(E,C)$ and
$L_K(E,C)$ is that, in general, the generating set of partial
isometries of these algebras is not tame. This is not the case for
the usual graph algebras, where it can be easily shown that the
generating set of partial isometries is tame.

\medskip

The main objective of the present paper is to associate a partial
dynamical system to any finite bipartite separated graph $(E,C)$,
and to uncover some of the fundamental properties of the crossed
products corresponding to it. To a large extent, this task can be
made simultaneously in the purely algebraic category and in the
C*-algebraic category. However, several analytic aspects, such as
the study of exactness, nuclearity and reduced crossed products will
need more specific tools. This dynamical system is described by the
action $\theta $ of a finitely generated free non-abelian group
$\mathbb F$, whose rank coincides with the number of edges of $E$,
on a  zero-dimensional metrizable compact space $\Omega (E,C)$. The
corresponding C*-algebra full crossed product $\mathcal O (E,C) :=
C(\Omega (E,C)) \rtimes _{\theta^*} \mathbb F$ is shown  to coincide
with a canonical quotient of the graph C*-algebra $C^*(E,C)$ of the
separated graph $(E,C)$ (Corollary
\ref{cor:Lab+OECarecrossedprods}). Namely we obtain an isomorphism
$$ C(\Omega (E,C)) \rtimes _{\theta^*}
\mathbb F\cong C^*(E,C)/J $$ where $J$ is the closed ideal of
$C^*(E,C)$ generated by all the commutators $[e(\gamma ), e(\mu)]$,
where $\gamma, \mu $ belong to the multiplicative semigroup $U$ of
$C^*(E,C)$ generated by the canonical partial isometries,
corresponding to $E^1\cup (E^1)^*$. Note that $E^1\cup (E^1)^*$
becomes a tame set of partial isometries in $\mathcal O (E,C)$. The
ideal $J$ is the closure of an increasing union of ideals $J_n$,
with $J_n$ the ideal generated by all the commutators $[e(\gamma) ,
e(\mu)]$, where $\gamma, \mu$ can be expressed as products of $\le
n$ canonical generators. We show in Section
\ref{sect:main-construct} that there is a canonical sequence $\{
(E_n, C^n)\}$ of finite bipartite separated graphs such that
$C^*(E,C)/J_n\cong C^*(E_n, C^n)$ for all $n$.
 Consequently the C*-algebra $\mathcal O (E;C)=
C^*(E,C)/J$ is isomorphic to an inductive limit $\varinjlim_n
C^*(E_n,C^n)$ with surjective transition maps. This gives an
approximation result of the dynamical C*-algebra $\mathcal O (E,C)$
by graph C*-algebras of separated graphs. We also consider in
Section \ref{sect:topfree} a reduced C*-algebra $\mathcal O ^r
(E,C)$, which is defined in terms of the {\it reduced crossed
product} of $C(\Omega (E,C))$ by $\mathbb F$. These C*-algebras have
been studied in detail in \cite{AEK} in the case where
$(E,C)=(E(m,n), C(m,n))$ is the separated graph whose graph
C*-algebra is Morita-equivalent to MacClanahan's algebras
$U^{\text{nc}}_{m,n}$ (see Example \ref{exam:m,ndyn-system} for the
definition),  under the notations $\mathcal O _{m,n}$ and $\mathcal
O _{m,n}^r$. They are attached to the so-called universal
$(m,n)$-dynamical system, see \cite[Theorem 3.8]{AEK}. It is shown
in \cite{AEK} that the natural map $\mathcal O _{m,n}\to \mathcal O
^r _{m,n}$ is not injective.

\medskip

A similar result holds when we work in the category of $*$-algebras,
obtaining a corresponding isomorphism
$$C_K(\Omega (E,C))\rtimes ^{\text{alg}}_{\theta^*} \mathbb F \cong L_K(E,C)/J \cong \varinjlim
L_K(E_n,C^n),$$ where $C_K(\Omega(E,C))$ is the $*$-algebra of
continuous functions from $\Omega (E,C)$ to the $*$-field $K$
endowed with the discrete topology, and the crossed product is taken
in the algebraic category.
 We denote this quotient $*$-algebra $L_K(E,C)/J$ by
$\Lab _K(E,C)$. (Here, of course, $J$  means the algebraic ideal
generated by the commutators $[e(\gamma), e(\mu)]$, for $\gamma,
\mu\in U$.) This leads to a monoid isomorphism $\mon{\Lab
_K(E,C)}\cong \varinjlim_n \mon{L_K (E_n,C^n)}\cong \varinjlim _n
M(E_n, C^n)$. Conjecturally, the same monoid $\varinjlim _n M(E_n,
C^n)$ would compute the nonstable K-theory of $\mathcal O (E,C)$.
The monoid $\mathfrak R (E,C) : =\varinjlim M(E_n, C^n)$ is a
refinement monoid (see Section 2 for the definition) and the natural
map $M(E,C)\to \mathfrak R (E,C)$ gives a canonical refinement of
$M(E,C)$ (Lemma \ref{lem:refinement}). This condition implies in
particular that the map $M(E,C)\to \mathfrak R (E,C)$ is an
order-embedding.

\medskip

We remark that there is no loss of generality in assuming that our
separated graphs are bipartite, as we show in Proposition
\ref{prop:reductiontobipartitegraphs} that every graph algebra of a
separated graph is Morita-equivalent to a graph algebra of a
bipartite separated graph.

\medskip

As a biproduct of our results, we are able to answer a question
raised in recent papers by Kerr \cite{Kerr}, Kerr and Nowak
\cite{KN}, and R\o rdam and Sierakowski \cite{RS}. The statement of
this question is unrelated to graph algebras but, using our new
theory, we can give a complete answer to it. The classical Tarski's
Theorem (\cite[Corollary 9.2]{Wagon}) asserts that, for an action of
a group $G$ on a set $X$, a subset $E$ of $X$ is not $G$-paradoxical
if and only if there is a finitely additive $G$-invariant measure
$\mu \colon \mathcal P (X)\to [0,\infty]$ such that $\mu (E)= 1$.
This result follows from special properties of the {\it type
semigroup} $S(X,G)$ of the action. Observe that the set of finitely
additive $G$-invariant measures $\mu$ as above is precisely the
state space of $(S(X, G), [E])$.

\medskip

If $G$ is a discrete group acting by continuous transformations on a
topological space $X$ and $\mathbb D$ is a $G$-invariant subalgebra
of subsets of $X$, one may restrict the equidecomposability relation
to subsets in $\mathbb D$, obtaining a ``relative" type semigroup
$S(X, G, \mathbb D) $, as in \cite{Kerr}, \cite{KN}, \cite{RS}. Of
particular interest is the case where $X$ is a zero-dimensional
metrizable compact space and $\mathbb D$ is the subalgebra $\mathbb
K$ of clopen subsets of $X$. It has been asked in the
above-mentioned papers (cf. \cite[page 285]{RS}, \cite[Question
3.10]{Kerr}) whether the analogue of Tarski's Theorem holds in this
more general context, that is, whether for $E\in \mathbb K$ we have
that $2[E]\nleq [E]$ in $S(X,G,\mathbb K)$ if and only if the state
space of $(S(X,G, \mathbb K), [E])$ is non-empty. We show in
Corollary \ref{cor:non-paradoxical2} that the answer to this
question is negative. Indeed, our Theorem
\ref{thm:main-typesemigroup2} implies that $S(X,G, \mathbb K)$ does
not satisfy in general any cancellation or order-cancellation law,
since, given any finitely generated conical abelian monoid $M$ we
can find a triple $(X,G, \mathbb K)$, with $X$ a 0-dimensional
metrizable compact space, such that there is an order-embedding of
$M$ into $S(X,G, \mathbb K)$. Our methods give in a natural way {\it
partial actions} of discrete groups on totally disconnected spaces.
We use globalization results from \cite{Abadie} to reach the desired
goal with {\it global actions}.

\medskip

We also obtain a characterization of the finite bipartite separated
graphs $(E,C)$ such that the canonical action of the free group
$\mathbb F$ on the space $\Omega (E,C)$ is topologically free. It
turns out that this is equivalent to a version, adapted to separated
graphs, of the well-known condition (L) for graphs, which is the
condition that every cycle has an entry (\cite{KPR},
\cite[Definition 12.1]{ExelLaca}), see Theorem \ref{thm:top-free}.
Using this result and known facts on crossed products by partial
actions from \cite{ELQ} we show that, if $(E,C)$ satisfies condition
(L), then the  {\it reduced crossed product} C*-algebra $\mathcal O
^r (E,C)=C(\Omega (E,C))\rtimes ^r \mathbb F$ satisfies condition
(SP), that is, every nonzero hereditary subalgebra contains a
nonzero projection, see Theorem \ref{thm:SPproperty}. A
corresponding result is also shown for the algebra $\Lab _K (E,C)$
(Theorem \ref{thm:AlgebraicSP} and Remark
\ref{rem:one-sidedideals}).

\medskip

Many properties of these new classes of algebras remain to be
investigated. We mention the computation of K-theoretic invariants
in terms of the separated graph $(E,C)$, the study of the structure
of ideals, and the study of exactness and nuclearity. Topologically
free minimal orbits of the topological space $\Omega (E,C)$ with
respect to the canonical action of the free group will surely
provide examples of simple C*-algebras with exotic properties. In a
different direction, the algebras constructed in the present paper
represent a further step in the strategy initiated in \cite{AG} to
construct von Neumann regular rings and exchange rings with
prescribed $\mathcal V$-monoids (see \cite{Areal} for a survey on
this problem). This will be discussed in more detail in
\cite{AGprep}, where the refinement monoid $\mathfrak R
(E,C)=\varinjlim M(E_n, C^n)$ associated to the separated graph
$(E,C)$ described in Example \ref{exam:alggenepi} will be realized
as the $\mathcal V$-monoid of an exchange $K$-algebra obtained as a
universal localization of the corresponding algebra $\Lab _K(E,C)$.
The latter algebra is Morita equivalent to the algebra of the {\it
free monogenic inverse semigroup} (see Example
\ref{exam:alggenepi}).

\medskip

\noindent{\bf Contents.} We now explain in more detail the contents
of this paper. In Section 2 we recall the basic definitions needed
for our work, coming from the papers \cite{AG2} and \cite{AG}.
Section 3 contains some preparatory material concerning graph theory
and semigroup theory. In particular we introduce the crucial concept
of a multiresolution of a separated graph $(E,C)$ at a set of
vertices $V$ of $E$. This is a modification of the concept of
resolution of a separated graph at a set of vertices, which has been
introduced in \cite[Section 8]{AG}. As in \cite{AG}, the aim to
introduce this concept is to give embeddings of the monoid $M(E,C)$
associated to a separated graph $(E,C)$ into a refinement monoid
corresponding  to another separated graph $(F,D)$, given in general
as the limit of an infinite sequence of separated graphs obtained by
successive applications of the resolution process. The
multiresolution process differs from the resolution process in that
it involves at once all the sets in the corresponding partitions
$C_v$, $v\in V$, and so it does not depend on the particular choice
of pairs $X,Y\in C_v$, for $v\in V$. In Section 4, we define, for a
given finite bipartite separated graph $(E,C)$, a canonical sequence
of finite separated graphs $(F_n,D^n)$, obtained by successive
applications of the multiresolution process, in such a way that the
union $(F_{\infty}, D^{\infty})=\bigcup_{n=0}^{\infty}(F_n,D^n)$
satisfies that $M(F_{\infty}, D^{\infty})$ is a refinement of
$M(E,C)$. The canonical sequence $\{(E_n,C^n)\}$ of finite bipartite
separated graphs obtained as the {\it sections} of $(F_{\infty},
D^{\infty})$ is of special importance. On one hand, it is shown in
Lemma \ref{lem:refinement} that there are natural connecting
homomorphisms $\iota_n\colon M(E_n,C^n)\to M(E_{n+1}, C^{n+1})$ such
that $M(F_{\infty}, D^{\infty})\cong \varinjlim _n
(M(E_n,C^n),\iota_n)$, so that the map $M(E,C)\to \varinjlim _n
(M(E_n,C^n),\iota_n)$ is a refinement of $M(E,C)$. On the other
hand, it is shown in Theorem \ref{thm:commutatorsstepn} that there
is a natural isomorphism $L_K(E,C)/J_n\cong L_K(E_n,C^n)$, where
$J_n$ is the ideal of $L_K(E,C)$ generated by all the commutators
$[e(\gamma), e(\mu)]$, where $\gamma, \mu$ are products of $\le n$
canonical generators from $E^1\cup (E^1)^*$. A similar result holds
for the corresponding graph C*-algebras. The notations $\Lab (E,C)$
and $\mathcal O (E,C)$ are introduced in \ref{notation:Lab}. The
topological space $\Omega (E,C)$ is also introduced in Section 5 as
the spectrum of the canonical commutative AF-subalgebra $\mathcal
B_0$ of $\mathcal O (E,C)$.

\medskip

We develop in Section 6 the crossed product structure of the
algebras $\Lab _K(E,C)$ and $\mathcal O (E,C)$. This is done in a
soft way, using the universal properties of the involved objects.
The topological space $\Omega (E,C)$ is shown in Corollary
\ref{cor:univECdynsustem} to be the {\it universal $(E,C)$-dynamical
system}. This gives a different proof of \cite[Theorem 3.8]{AEK},
where the particular case of $(m,n)$-dynamical systems is
considered. Moreover the natural partial action of the free group
$\mathbb F= \mathbb F \langle E^1 \rangle $ on $\Omega (E,C)$
induces $*$-isomorphisms $\Lab _K(E,C)\cong C_K(\Omega (E,C))\rtimes
_{\alpha} \mathbb F$ and $\mathcal O (E,C)\cong C(\Omega
(E,C))\rtimes _{\alpha } \mathbb F$ (Corollary
\ref{cor:Lab+OECarecrossedprods}). Section 7 is devoted to the study
of the type semigroups. Our main result is Theorem
\ref{thm:main-typesemigroup2}, where it is shown that, given any
finitely generated abelian conical monoid $M$ there exists a
0-dimensional metrizable compact space $Z$ and an action of a
finitely generated free group $\mathbb F$ on it so that $M$
order-embeds in the type semigroup $S(Z, \mathbb F, \mathbb K)$.
This uses the machinery from Sections 3--6 plus Abadie's
globalization results \cite{Abadie}. In Section 8 we develop
descriptions of the space $\Omega (E,C)$ and the canonical partial
action of the free group $\mathbb F$ on it. There are two different
descriptions, one in terms of configurations and another in terms of
certain ``choice functions", called here $E$-functions. We show in
Theorem \ref{thm:Efunctions} that the basic clopen sets
corresponding to the vertices in the odd layers of the graph
$F_{\infty}$ (introduced in Construction
\ref{cons:complete-multiresolution}) correspond precisely to the
partial $E$-functions. Several examples are presented in Section 9.
Apart from the motivational example leading to the $(m,n)$-dynamical
system (Example \ref{exam:m,ndyn-system}), we consider various
interesting examples which are connected with constructions already
investigated in disparate contexts. So for example we recover Truss
example from \cite{Truss} of a $G$-space $X$ with failure of the
$2$-cancellation law $2x=2y\implies x=y$ in the type semigroup
(Example \ref{exam:trussexam}), a separated graph $(E,C)$ such that
the algebra $\mathcal O (E,C)$ is Morita-equivalent to the
C*-algebra of the free monogenic inverse semigroup \cite{HR}
(Example \ref{exam:alggenepi}), and another separated graph $(F,D)$
such that $\mathcal O (F,D)$ (respectively $\Lab _K(F,D)$) is
Morita-equivalent to the group C*-algebra of the lamplighter group
$C^*(\Z_2\wr \Z)$ (respectively the group algebra $K[\Z_2\wr \Z]$).
 Finally, Section 10 contains our
characterization of the finite bipartite separated graphs $(E,C)$
such that the canonical action of the free group group $\mathbb F$
on the space $\Omega (E,C)$ is topologically free and the
consequences for the algebraic structure of $\mathcal O (E,C)$ and
$\Lab _K(E,C)$.

\medskip

After the first version of this paper was posted at arXiv, we have been informed by Fred Wehrung that, by using results of H.
Dobbertin, he has obtained a proof of the fact that any countable conical refinement monoid with order-unit 
is isomorphic to a type semigroup of the form $S(Z,\mathbb F,\mathbb K)$, where $\mathbb F$ is a countably generated free group.

\section{Preliminary definitions}
\label{sect:prels}

The concept of separated graph, introduced in \cite{AG}, plays a
vital role in our construction. In this section, we will recall this
concept and we will also recall the definitions of the {\it monoid
associated to a separated graph}, the {\it Leavitt path algebra} and
the {\it graph C*-algebra} of a separated graph.

We will use the reverse notation as in \cite{AG} and \cite{AG2}, but
in agreement with the one used in \cite{AEK}, and in the book
\cite{Raeburn}.

\begin{definition}{\rm (\cite{AG})} \label{defsepgraph}
A \emph{separated graph} is a pair $(E,C)$ where $E$ is a graph,
$C=\bigsqcup _{v\in E^ 0} C_v$, and $C_v$ is a partition of
$r^{-1}(v)$ (into pairwise disjoint nonempty subsets) for every
vertex $v$. (In case $v$ is a source, we take $C_v$ to be the empty
family of subsets of $r^{-1}(v)$.)

If all the sets in $C$ are finite, we say that $(E,C)$ is a
\emph{finitely separated} graph. This necessarily holds if $E$ is
column-finite (that is, if $r^{-1}(v)$ is a finite set for every
$v\in E^0$.)

The set $C$ is a \emph{trivial separation} of $E$ in case $C_v=
\{r^{-1}(v)\}$ for each $v\in E^0\setminus \So (E)$. In that case,
$(E,C)$ is called a \emph{trivially separated graph} or a
\emph{non-separated graph}.
\end{definition}

The following definition gives the Leavitt path algebra $L_K(E,C)$
as a universal object in the category of $*$-algebras over a fixed
field with involution $K$. The underlying algebra is naturally
isomorphic to the algebra defined in \cite[Definition 2.2]{AG} in
case every vertex is either the range vertex or the source vertex of
some edge in $E$.

\begin{definition}
\label{def:LPASG} {\rm Let $(K, *)$ be a field with involution. The
{\it Leavitt path algebra of the separated graph} $(E,C)$ with
coefficients in the field $K$, is the $*$-algebra $L_K(E,C)$ with
generators $\{ v, e\mid v\in E^0, e\in E^1 \}$, subject to the
following relations:}
\begin{enumerate}
\item[] (V)\ \ $vv^{\prime} = \delta_{v,v^{\prime}}v$ \ and \ $v=v^*$ \ for all $v,v^{\prime} \in E^0$ ,
\item[] (E)\ \ $r(e)e=es(e)=e$ \ for all $e\in E^1$ ,
\item[] (SCK1)\ \ $e^*e'=\delta _{e,e'}s(e)$ \ for all $e,e'\in X$, $X\in C$, and
\item[] (SCK2)\ \ $v=\sum _{ e\in X }ee^*$ \ for every finite set $X\in C_v$, $v\in E^0$.
\end{enumerate}
\end{definition}

The Leavitt path algebra $L_K(E)$ is just $L_K(E,C)$ where $C_v= \{
r^{-1}(v)\}$ if $r^{-1}(v)\ne \emptyset $ and $C_v=\emptyset $ if
$r^{-1}(v)=\emptyset$. An arbitrary field can be considered as a
field with involution by taking the identity as the involution.
However, our ``default" involution over the complex numbers $\C$
will be the complex conjugation.

We now recall the definition of the graph C*-algebra $C^*(E,C)$,
introduced in \cite{AG2}.

\begin{definition} The \emph{graph C*-algebra} of a separated graph $(E,C)$ is the
C*-algebra $C^*(E,C)$  with generators $\{ v, e \mid v\in E^0,\ e\in
E^1 \}$, subject to the relations (V), (E), (SCK1), (SCK2). In other
words, $C^*(E,C)$ is the enveloping C*-algebra of $L(E,C)$.
\end{definition}

In case $(E,C)$ is trivially separated, $C^*(E,C)$ is just the
classical graph C*-algebra $C^*(E)$. There is a unique
*-homomorphism  $L_\C(E,C) \rightarrow
C^*(E,C)$ sending the generators of $L(E,C)$ to their canonical
images in $C^*(E,C)$. This map is injective by \cite[Theorem
3.8(1)]{AG2}.

If no confusion can arise, we will suppress the field $K$ from our
notation.

Since both $L_K(E,C)$ and $C^*(E,C)$ are universal objects with
respect to the same sets of generators and relations, many of the
constructions of this paper apply in the same way to either class. A
remarkable difference is that we {\it know} exactly what is the
structure of the monoid $\mon{L_K(E,C)}$ for any separated graph
$(E,C)$, but we still do not know the structure of the monoid
$\mon{C^*(E,C)}$, although it is conjectured in \cite{AG2} that the
natural map $L_\C(E,C)\to C^*(E,C)$ induces an isomorphism
$\mon{L_\C(E,C)}\to \mon{C^*(E,C)}$. See \cite[Section 6]{Aone-rel}
for a short discussion on this problem.

Recall that for a unital ring $R$, the monoid $\mon{R}$ is usually
defined as the set of isomorphism classes $[P]$ of finitely
generated projective (left, say) $R$-modules $P$, with an addition
operation given by $[P]+[Q]= [P\oplus Q]$. For a nonunital version,
see Definition \cite[Definition 10.8]{AG}.

For arbitrary rings, $\mon{R}$ can also be described in terms of
equivalence classes of idempotents from the ring $M_\infty(R)$ of
$\omega\times\omega$ matrices over $R$ with finitely many nonzero
entries. The equivalence relation is \emph{Murray-von Neumann
equivalence}: idempotents $e,f\in M_\infty(R)$ satisfy $e\sim f$ if
and only if there exist $x,y\in M_\infty(R)$ such that $xy=e$ and
$yx=f$. Write $[e]$ for the equivalence class of $e$; then $\mon{R}$
can be identified with the set of these classes. Addition in
$\mon{R}$ is given by the rule $[e]+[f]= [e\oplus f]$, where
$e\oplus f$ denotes the block diagonal matrix $\left(
\begin{smallmatrix} e&0\\ 0&f \end{smallmatrix} \right)$. With this
operation, $\mon{R}$ is an abelian monoid, and it is \emph{conical},
meaning that $a+b=0$ in $\mon{R}$ only when $a=b=0$. Whenever $A$ is
a $C^*$-algebra, the monoid $\mon{A}$ agrees with the monoid of
equivalence classes of projections  in $M_{\infty}(A)$ with respect
to the equivalence relation given by $e\sim f $ if and only if there
is a partial isometry $w$ in $M_{\infty}(A)$ such that $e=ww^*$ and
$f=w^*w$; see \cite[4.6.2 and 4.6.4]{Black} or \cite[Exercise
3.11]{rordam}.

We will need the definition of $M(E,C)$ only for finitely separated
graphs. The reader can consult \cite{AG} for the definition in the
general case. Let $(E,C)$ be a finitely separated graph, and let
$M(E,C)$ be the abelian monoid given by generators $a_v$, $v\in
E^0$, and relations $a_v=\sum _{e\in X} a_{s(e)}$, for $X\in C_v$,
$v\in E^0$. Then there is a canonical monoid homomorphism $M(E,C)\to
\mon{L_K(E,C)}$, which is shown to be an isomorphism in
\cite[Theorem 4.3]{AG}. The map $\mon{L_{\C}(E,C)}\to
\mon{C^*(E,C)}$ induced by the natural $*$-homomorphism $L_K(E,C)\to
C^*(E,C)$ is conjectured to be an isomorphism for all finitely
separated graphs $(E,C)$ (see \cite{AG2} and \cite[Section
6]{Aone-rel}).

\section{Multiresolutions}
\label{sect:multires}

In this section, we will introduce the concept of mutiresolution of
a finitely separated graph $(E,C)$ at a vertex $v\in E^0$, which is
closely related to the notion of resolution, studied in \cite{AG}.
These computations will be heavily used in the forthcoming sections.

An (abelian) monoid $M$ is said to be a {\it refinement monoid} if
whenever $a+b=c+d$ in $M$, there exist $x,y,z,t$ in $M$ such that
$a=x+y$ and $b=z+t$ while $c=x+z$ and $d=y+t$.

We recall the following definitions from \cite{AG}. For $X\in C$, we
write ${\bf s}(X):= \sum_{e\in X} s(e)$.

\begin{definition} \label{defsquigarrow}
Let $F$ be the free abelian monoid on $E^0$. For $\alpha,\beta \in
F$, write $\alpha \rightsquigarrow_1 \beta$ to denote the following
situation:
\begin{itemize}
\item[] $\alpha= \sum_{i=1}^k u_i+\sum_{i=k+1}^m u_i$ and
$\beta= \sum_{i=1}^k  {\bf s}(X_i)+\sum_{i=k+1}^m u_i$ for some
$u_i\in E^0$ and some $X_i\in C_{u_i}$, where $u_1,\dots ,u_k\in
E^0\setminus \So (E)$.
\end{itemize}
We view $\alpha=0$ as an empty sum of the above form (i.e.,
$k=m=0$), so that $0\rightsquigarrow_1 0$. Note also, taking $k=0$,
that $\alpha \rightsquigarrow_1 \alpha$ for all $\alpha \in F$.
\end{definition}

\begin{definition}
\label{defstar} {\bf Assumption $(*)$:} Let $v\in E^0$. Then $(E,C)$
satisfies $(*)$ at $v$  if for all $X,Y\in C_v$, there exists
$\gamma\in F$ such that ${\bf s}(X) \rightsquigarrow_1 \gamma$ and
${\bf s} (Y) \rightsquigarrow_1 \gamma$.

This assumption only needs to be imposed when $X$ and $Y$ are
disjoint, since otherwise $X=Y$ and the conclusion is trivially
satisfied.

The separated graph $(E,C)$ satisfies Assumption $(*)$ in case it
satisfies $(*)$ at every vertex $v\in E^0$.
\end{definition}

Assumption (*) guarantees the refinement property in the monoid
$M(E,C)$:

\begin{proposition}
\label{refinement} \cite[Proposition 5.9]{AG} Let $(E,C)$ be a
finitely separated graph. If Assumption $(*)$ holds, then the monoid
$M(E,C)$ is a refinement monoid.
\end{proposition}

In fact \cite{AG} contains a version of the above result which holds
for arbitrary separated graphs $(E,C)$.

The multiresolution of a separated graph $(E,C)$ at a vertex $v$
(with $|r^{-1}(v)|<\infty$) is the simultaneous resolution (by
refinement) of {\it all} the subsets $X\in C_v$ at once. We describe
this process here. This should be compared with the resolution
process of \cite{AG}, where  only {\it pairs} $X,Y\in C_v$ are
refined.

\begin{definition}
\label{multiresatv} Let $C_v=\{ X_1,\dots ,X_k\}$ with each $X_i$ a
finite subset of $r^{-1}(v)$. Put $M=\prod _{i=1}^k |X_i|$. Then the
multiresolution of $(E,C)$ at $v$ is the separated graph $(E_v,
C^v)$ with
$$E_v^0=E^0\sqcup \{v(x_1,\dots ,x_k)\mid x_i\in X_i, i=1,\dots ,k \},$$
and with $E_v^1=E^1\sqcup \Lambda$, where $\Lambda $ is a new set of
arrows defined as follows. For each $x_i\in X_i$, we put $M/|X_i|$
new arrows $\alpha ^{x_i}(x_1,\dots ,x_{i-1},x_{i+1},\dots ,x_k )$,
$x_j\in X_j$, $j\ne i$, with
$$r(\alpha ^{x_i}(x_1,\dots ,x_{i-1},x_{i+1},\dots ,x_k ))=s(x_i), \text{ and }
s(\alpha ^{x_i}(x_1,\dots ,x_{i-1},x_{i+1},\dots ,x_k ))=v(x_1,\dots
,x_k).$$ For a vertex $w\in E^0$, define the new groups at $w$ as
follows. These groups are indexed by the edges $x_i\in X_i$,
$i=1,\dots ,k$, such that $s(x_i)=w$. For one such $x_i$, set
$$X(x_i) =\{\alpha ^{x_i} (x_1,\dots ,x_{i-1},x_{i+1},\dots , x_k)\mid x_j\in X_j, j\ne i\}.$$
Then $$(C^v)_w=C_w\sqcup \{X(x_i)\mid x_i\in X_i, s(x_i)=w,
i=1,\dots , k\}.$$ The new vertices $v(x_1,\dots , x_k)$ are sources
in $E_v$.
\end{definition}

\begin{definition}
\label{multiresatsetofvs} Let $V\subseteq E^0$ be a set of vertices
such that, for each $u\in V$,  $C_u=\{X_1^u, \dots ,X_{k_u}^u\}$,
with each $X^u_i$ a finite subset of $r^{-1}(u)$. Then the {\it
multiresolution of $E$ at $V$} is the separated graph $(E_V, C^V)$
obtained by applying the above process to all vertices $u$ in $V$.

Hence
$$E_V^0=E^0\sqcup \Big( \bigsqcup _{u\in V} \{v(x^u_1,\dots ,x^u_{k_u})\mid x^u_i\in X^u_i, i=1,\dots ,k_u \} \Big),$$
and $E_V^1=E^1\sqcup \Big( \bigsqcup _{u\in V} \Lambda_u\Big)$,
where $\Lambda _u$ is the corresponding set of arrows $\alpha
^{x^u_i}(x^u_1,\dots ,x^u_{i-1},x^u_{i+1},\dots ,x^u_{k_u} )$
defined as in Definition \ref{multiresatv}, for each $u\in V$. The
sets $(C^V)_w$, for $w\in E_V^0$, are defined just as in Definition
\ref{multiresatv}:
$$(C^V)_w=C_w\sqcup \{X(x^u_i)\mid x^u_i\in X^u_i, s(x^u_i)=w,
i=1,\dots , k_u, u\in V\}.$$ The new vertices $v(x^u_1,\dots ,
x^u_{k_u})$ are sources in $E_V$.

\end{definition}

\begin{lemma}
\label{lem:refiatV} Let $V\subseteq E^0$ be a set of vertices such
that, for each $u\in V$,  $C_u=\{X^u_1, \dots ,X^u_{k_u}\}$, with
each $X^u_i$ a finite subset of $r^{-1}(u)$. Then the separated
graph $(E_V, C^V)$ satisfies (*) at $u$ for all $u\in V$.
\end{lemma}

\begin{proof}
Observe that, for $x_i^u\in X^u_i$, we have $s(x^u_i)\rightarrow _1
\sum _{x^u_j\in X^u_j, j\ne i} v(x^u_1,\dots
,x^u_{i-1},x^u_i,x^u_{i+1},\dots ,x^u_{k_u})$, so that
$${\bf s}(X^u_i)\rightsquigarrow _1 \sum _{(x_1^u,\dots ,
x_{k_u}^u)\in X_1^u\times \cdots \times X_{k_u}^u} v(x^u_1,\dots
,x^u_{k_u})=:\gamma .$$ This clearly shows the result, indeed we
have ${\bf s}(X^u_i)\rightsquigarrow _1 \gamma$, for all $i=1,\dots
,k_u$.
\end{proof}

Let us recall the definition of unitary embedding:

\begin{definition}
\label{def:unitemb} {\rm Following \cite{W}, a monoid homomorphism
$\psi:M\rightarrow F$ is a \emph{unitary embedding} provided
\begin{enumerate}
\item $\psi$ is injective;
\item $\psi(M)$ is \emph{cofinal} in $F$, that is, for each $u\in F$ there is some $v\in M$ with $u\le \psi(v)$;
\item whenever $u,u'\in M$ and $v\in F$ with $\psi(u)+v= \psi(u')$, we have $v\in \psi(M)$.
\end{enumerate}
}\end{definition}

\begin{lemma} \label{lem:unitary3}
Let $(E,C)$ be a separated graph and let $V\subseteq E^0$ be a set
of vertices such that $|r^{-1}(u)|<\infty$ for all $u\in V$. Let
$\iota: (E,C) \rightarrow (E_V,C^V)$ denote the inclusion morphism,
where $(E_V, C^V)$ is the multiresolution of $(E,C)$ at $V$. Then
$M(\iota): M(E,C)\rightarrow M(E_V,C^V)$ is a unitary embedding.
\end{lemma}

\begin{proof} The proof follows the steps of the proof of \cite[Lemma 8.6]{AG}.

We will sketch some of the details.

Set $\mu= M(\iota)$. For $u\in V$, set $C_u=\{ X_1^u,\dots
,X_{k_u}^u\}$. Let $F$ be the free monoid on generators
$a(x_1^u,\dots ,x_{k_u}^u)$, for $x_i^u\in X_i^u$, $i=1,\dots ,k_u$,
$u\in V$. Let $M$ be the monoid given by generators $\{b(x_i^u)\mid
x_i^u\in X_i^u, i=1,\dots ,k_u, u\in V\}$ with the relations
$\sum_{x_i^u\in X_i^u} b(x_i^u)= \sum_{x_j^u\in X_j^u} b(x_j^u)$ for
all $1\le i <j\le k_u$, for all $u\in V$. There is a natural
homomorphism $\psi: M\rightarrow F$ sending $b(x_i^u)$ to
$$ \sum _{j\ne i} \sum_{x_j^u\in X_j^u} a(x_1^u,\dots ,x_{i-1}^u, x_i^u,x_{i+1}^u,\dots ,x_k^u)$$
for $x_i^u\in X_i^u$, $i=1\dots ,k_u$, $u\in V$. Arguments similar
to the ones used in the proof of \cite[Lemma 8.2]{AG} give that
$\psi$ is a unitary embedding.

There is a unique homomorphism $\eta: M\rightarrow M(E,C)$ sending
$b(x_i^u)$ to $[s(x_i^u)]$ for $x_i^u\in X_i^u$, $1\le i\le k_u$,
$u\in V$, and there is a unique homomorphism $\eta': F\rightarrow
M(E_V,C^V)$ sending $a(x_1^u,\dots ,x_{k_u}^u) \mapsto
[v(x_1^u,\dots ,x_{k_u}^u)]$ for $x_i^u\in X_i^u$, $i=1,\dots ,k_u$,
$u\in V$. There is a commutative diagram as follows:
\begin{equation} \label{monres}
\xymatrixrowsep{3pc}\xymatrixcolsep{5pc} \xymatrix{
M \ar[r]^{\psi} \ar[d]_{\eta} & F \ar[d]^{\eta'} \\
M(E,C) \ar[r]^{\mu} &M(E_V,C^V) }
\end{equation}
An easy adaptation
of the proof of \cite[Lemma 8.6]{AG} gives that \eqref{monres} is a
pushout in the category of abelian monoids. It follows from
\cite[Lemma 1.6]{W} that $\mu$ is a unitary embedding, completing
the proof.
\end{proof}

\begin{remark}
\label{remark:univpro} (Universal property) Let $(E,C)$ be a
separated graph and let $V\subseteq E^0$ be a set of vertices such
that $|r^{-1}(u)|<\infty$ for all $u\in V$.

The pushout property appearing in the proof of the above lemma is
equivalent to the following universal property of $M(E_V,C^V)$:
Given a monoid homomorphism $\Phi\colon M(E,C)\to N$ and given a
multiresolution  $\{ c(x^u_1,\dots ,x^u_{k_u})\mid x_i^u\in
X^u_i\}$, $u\in V$, in $N$ of the set of equations
$$ \Phi ({\bf s}(X^u_1))=\Phi ({\bf s}(X^u_2))=\cdots = \Phi ({\bf
s}(X^u_{k_u})), \quad (u\in V)$$ (so that $\Phi (s(x^u_i))= \sum
_{j\ne i, x^u_j\in X^u_j} c(x_1^u,\dots, x_{k_u}^u)$ for all
$x^u_i\in X^u_i$, $u\in V$), there exists a unique monoid
homomorphism $\tilde{\Phi}\colon M(E_V,C^V)\to N$ such that
$\tilde{\Phi}(v)= \Phi(v)$ for all $v\in E^0$ and $\tilde{\Phi}
(v(x^u_1,\dots , x_{k_u}^u))= c(x_1^u,\dots , x^u_{k_u})$ for all
$x^u_i\in X^u_i$, $i=1,\dots ,k_u$, $u\in V$.

Note that the above multiresolutions always exist if $N$ is a
refinement monoid.
\end{remark}

\section{Bipartite separated graphs}
\label{sect:bipsepgraphs}

\begin{definition}
\label{def:bipartitesepgraph} Let $E$ be a directed graph. We say
that $E$ is a {\it bipartite directed graph} if $E^0= E^{0,0} \sqcup
E^{0,1}$, with all arrows in $E^1$ going from a vertex in $E^{0,1}$
to a vertex in $E^{0,0}$. To avoid trivial cases, we will always
assume that $r^{-1}(v)\ne \emptyset$ for all $v\in E^{0,0}$ and
$s^{-1}(v)\ne \emptyset $ for all $v\in E^{0,1}$.

A {\it bipartite separated graph} is a separated graph $(E,C)$ such
that the underlying directed graph $E$ is a bipartite directed
graph.
\end{definition}

\begin{definition}
\label{def:sep-graph-of-cam} Given a finitely presented abelian
conical monoid
$$M=\langle \mathcal X \mid \mathcal R \rangle \, ,$$
where $\mathcal X$ is a finite set of generators $a_1,\dots ,a_n$
and $\mathcal R$ is a finite set of relations $\mathbf r_1,\dots
,\mathbf r_m$ of the form
\begin{equation}
\label{eq:rels-form} \mathbf r_j :  \qquad \sum _{i=1}^n r_{ji}a_i =
\sum _{i=1}^n s_{ji}a_i ,
\end{equation}
where $r_{ji}$ and $s_{ji}$ are non-negative integers satisfying
$\sum _{i=1}^n r_{ji}>0$ and $\sum _{i=1}^n s_{ji}>0$ for all $j$,
and $\sum _{j=1}^m (r_{ji}+s_{ji}) >0$ for all $i$,  we may
associate to it a finite bipartite separated graph $(E,C)$ such that
$s^{-1}(v)\ne \emptyset $ for all $v\in E^{0,0}$ and $r^{-1}(v)\ne
\emptyset$ for all $v\in E^{0,1}$, as follows (cf. \cite[Proposition
4.4]{AG}:)
$$E^0=E^{0,0}\sqcup E^{0,1}, \quad \text{with} \quad E^{0,0}=\mathcal R,
\quad E^{0,1}=\mathcal X.  $$ For $\mathbf r_j\in \mathcal R$ given
by (\ref{eq:rels-form}), we set $C_{\mathbf r_j}= \{X_j,Y_j\}$,
where $X_j$ has exactly $r_{ji}$ arrows from $a_i$ to $\mathbf r_j$,
for $i=1,\dots ,n$, and $Y_j$ has exactly $s_{ji}$ arrows from $a_i$
to $\mathbf r_j$, for $i=1,\dots ,n$, so that $r^{-1} (\mathbf r_j)=
X_j\sqcup Y_j$. Then $E^1= \bigsqcup_{j=1}^m r^{-1}(\mathbf r_j)$.
We have an isomorphism $M(E, C)\cong M$ (\cite[Proposition
4.4]{AG}).
\end{definition}

Note that there is a bijective correspondence between the finite
bipartite separated graphs $(E,C)$ satisfying the conditions in
Definition \ref{def:bipartitesepgraph} such that $|C_v|= 2$ for all
$v\in E^{0,0}$ and the finite presentations of abelian conical
monoids as above. Also, recall that every finitely generated abelian
semigroup is finitely presented, by Redei's Theorem (see
\cite{Freyd} for a very simple proof), so that we can apply the
above process to any finitely generated abelian conical monoid. (Of
course, the finite bipartite separated graph will depend on the
finite presentation, not just on the monoid.)

\medskip

It is useful to introduce the following terminology:

\begin{definition}
\label{def:refinementofM} Let $M$ be an abelian conical monoid. A
{\it refinement} of $M$ is an abelian conical monoid $N$, together
with a monoid homomorphism $\iota \colon M\to N$ such that:
\begin{enumerate}
\item[(a)] $\iota $ is a unitary embedding.
\item[(b)] $N$ is a refinement monoid.
\item[(c)] For each refinement monoid $P$ and each monoid homomorphism
$\psi \colon M\to P$ there is a monoid homomorphism
$\widetilde{\psi}\colon N\to P$ such that $\psi =
\widetilde{\psi}\circ \iota$.
\end{enumerate}
\end{definition}

\begin{construction}
\label{cons:complete-multiresolution} (a) Let $(E, C)$ be a finite
bipartite separated graph. We define a nested sequence of finite
separated graphs $(F_n, D^n)$ as follows. Set $(F_0,D^0)=(E, C)$.
Assume that a nested sequence
$$(F_0,D^0)\subset (F_1,D^1)\subset\dots \subset (F_n,D^n)$$
has been constructed in such a way that for $i=1,\dots ,n$, we have
$F_i^0=\bigsqcup _{j=0}^{i+1} F^{0,j}$ for some finite sets
$F^{0,j}$ and $F_i^1=\bigsqcup _{j=0}^i F^{1,,j}$, with $s(F^{1,j})=
F^{0,j+1}$ and $r(F^{1,j})= F^{0,j}$ for $j=1,\dots ,i$. We can
think of $(F_n, D^n)$ as a union of $n$
 bipartite separated graphs. Assume that condition (*) for $(F_n, D^n)$ holds at all
 vertices in $\bigsqcup_{j=0}^{n-1} F^{0,j}$. Set $V_n= F^{0, n}$,
 and let $(F_{n+1}, D^{n+1})$ be the multiresolution of $(F_n, D^n)$
 at $V_n$. Then $F_{n+1}^0= F_n^0\bigsqcup F^{0,n+2}=\bigsqcup _{j=1}^{n+2} F^{0,j}$ and
 $F_{n+1}^1= F_n^1\bigsqcup F_{n+1}^{1,n+1}= \bigsqcup _{j=0}^{n+1}
 F^{1,j}$, with $s(F^{1,n+1}) = F^{0,n+2}$ and $r(F^{1,n+1})=s(F^{1,n})= F^{0,n+1}$.
Moreover, by Lemma \ref{lem:refiatV}, the separated graph $(F_{n+1},
D^{n+1})$ satisfies condition (*) at all the vertices in
 $\bigsqcup _{j=0}^{n} F^{0,j}$.

\medskip

\noindent (b) Let
 $$(F_{\infty}, D^{\infty})= \bigcup _{n=0}^{\infty} (F_n,
D^n) \, .$$
 Observe that $(F_{\infty}, D^{\infty})$ is the direct limit of the
 sequence $\{ (F_n,D^n) \}$ in the category $\FSGr$ defined in
 \cite[Definition 8.4]{AG}. Since the functor $M\colon \FSGr\to \Mon $
 is continuous (see \cite[8.4, 4.1]{AG}), it follows
that $M(F_{\infty}, D^{\infty})=\varinjlim M(F_n, D^n)$. Since
$(F_{\infty}, D^{\infty})$ clearly satisfies condition (*) at all
its vertices, it follows from \cite[Proposition 5.9]{AG} that
$M(F_{\infty}, D^{\infty})$ is a refinement monoid. Moreover , since
all maps $M(F_n, D^n)\to M(F_{n+1}, D^{n+1})$ are unitary embeddings
by Lemma \ref{lem:refiatV}, it follows that the map $M(E,C)\to
M(F_{\infty}, D^{\infty})$ is a unitary embedding. Finally,
condition (c) in Definition \ref{def:refinementofM} follows from
Remark \ref{remark:univpro}. Hence, we have that $M(E,C)\to
M(F_{\infty}, D^{\infty})$ is a refinement of $M(E,C)$. We call
$(F_{\infty}, D^{\infty})$ the {\it complete multiresolution} of
$(E,C)$.

\medskip

\noindent (c) We define a canonical sequence $ (E_n,C^n)$ of finite
bipartite separated graphs as follows:
\begin{enumerate}
\item Set $(E_0,C^0)= (E,C)$.
\item $E_n^{0,0}= F^{0,n}$,  $E_n^{0,1}=F^{0,n+1}$, and $E_n^1= F^{1,n}$.
Moreover $C^n_v= D^n_v$ for all $v\in E_n^{0,0}$ and
$C^n_v=\emptyset$ for all $v\in E_n^{0,1}$.
\end{enumerate}
We call the sequence $\{ (E_n,C^n)\}_{n\ge 0}$ the {\it canonical
sequence of bipartite separated graphs} associated to $(E,C)$.
\end{construction}

\begin{lemma}
\label{lem:refinement} Let $(E,C)$ be a finite bipartite separated
graph, let $(E_n,C^n)$ be the canonical sequence of bipartite
separated graphs associated to $(E, C)$, and let $(F_{\infty},
D^{\infty})$ be the complete multiresolution of $(E,C)$. Then the
following properties hold:
\begin{enumerate}
\item[(a)]  For each $n\ge 0$, there is a natural isomorphism
$$\varphi _n \colon M(E_{n+1}, C^{n+1}) \longrightarrow M((E_n)_{V_n},
(C^n)^{V_n}), $$ where $V_n= E_n^{0,0}= F^{0,n}$.
\item[(b)] For each $n\ge 0$, there is a canonical unitary embedding
$$\iota _n \colon
M(E_n, C^n)\to M(E_{n+1}, C^{n+1}).$$
\item[(c)] The canonical inclusion $j_n\colon (E_n,C^n)\to (F_n,
D^n)$ induces an isomorphism  $$M(j_n)\colon M(E_n,C^n)\to M(F_n,
D^n).$$
\item[(d)] We have $M(F_{\infty}, D^{\infty}) \cong  \varinjlim (M(E_n, C^n),\iota _n)$.
Consequently, the natural map \linebreak $M(E, C)\to \varinjlim
(M(E_n, C^n),\iota _n)$ is a refinement of $M(E, C)$.
\end{enumerate}
\end{lemma}

\begin{proof}
To simplify the notation, write $E_{V_n}:=(E_n)_{V_n}$ and $C^{V_n}=
(C^n)^{V_n}$.

 (a) Define $\varphi _n \colon M(E_{n+1}, C^{n+1})
\longrightarrow M(E_{V_n}, C^ {V_n}) $ by $\varphi _n (a_v)= a_v $
for $v\in E_{n+1}^0$. Conversely, define $\psi _n\colon M(E_{V_n},
C^{V_n})\to M(E_{n+1}, C^{n+1})$ by $\psi _n(a_v)= a_v$ for $v\in
E_{n+1}^0\subset E_{V_n}^0$ and $\psi_n (a_v)= {\bf s}(X)$ for $X\in
C_v$ if $v\in E_n^{0,0}$. We have to show that $\psi_n$ is a
well-defined monoid homomorphism. There are no relations at the
vertices at $E^{0,1}_{n+1}$, both in $M(E_{n+1}, C^{n+1})$ and in
$M(E_{V_n}, C^{V_n})$, because these vertices are sources in both
graphs. It is clear that the relations at vertices $v$ in
$E^{0,0}_{n+1}$ are preserved by $\psi_n$, because
$C^{n+1}_v=C^{V_n}_v$ for these vertices.   For a vertex $v\in
E^{0,0}_n$, take $X, Y\in C_v$.
 Setting $C_v=\{X_1,\dots ,X_s\}$, there are $1\le
p,q\le s$ such that $X=X_p$ and $Y=X_q$. In $M(E_{n+1}, C^{n+1})$,
we have
\begin{equation}
\label{eq:psi-well-defined}
 {\bf s}(X_p)= \sum _{x_p\in X_p} s(x_p)=
\sum _{(x_1,\dots, x_s)\in \prod_{i=1}^s X_i} v(x_1,\dots ,x_s)=\sum
_{x_q\in X_q} s(x_q)={\bf s}(X_q).
\end{equation}
 Hence ${\bf s}(X)={\bf s}(Y)$ for all $X,Y\in
C_v$, which shows that $\psi_n$ is well-defined. Clearly $\varphi
_n$ and $\psi _n$ are mutually inverse, so we get that $\varphi _n$
is an isomorphism.

(b) By Lemma  \ref{lem:unitary3}, the natural map
$M(\iota_{V_n})\colon M(E_n, C^n)\to M(E_{V_n}, C^{V_n})$ is a
unitary embedding. Hence the map $\iota_n = \varphi _n^{-1}\circ
M(\iota _{V_n})\colon M(E_n,C^n)\to M(E_{n+1}, C^{n+1})$ is a
unitary embedding.

(c) We use induction on $n$. For $n=0$, we have $(E_0,C^0)=(F_0,
D^0)$. Assume that, for some $n\ge 0$, the natural map $j_n\colon
(E_n,C^n)\to (F_n,D^n)$ induces an isomorphism $M(j_n)\colon
M(E_n,C^n)\to M(F_n,D^n)$. Since $(F_{n+1}, D^{n+1})=((F_n)_{V_n},
(D^n)^{V_n})$, we get an isomorphism $M(E_{V_n},C^{V_n})\to
M(F_{n+1}, D^{n+1})$ extending $M(j_n)$. Now we have the following
commutative diagram:

\begin{equation}
\begin{CD}
M(E_n,C^n) @>{M(\iota_{V_n})}>>
  M(E_{V_n}, C^{V_n}) @<{\varphi_n}<< M(E_{n+1}, C^{n+1})\\
@V{M(j_n)}VV  @V{\cong }VV  @VV{M(j_{n+1})}V \\
M(F_n, D^n) @>>> M(F_{n+1},D^{n+1}) @>{\text{id}}>>   M(F_{n+1},
D^{n+1})
\end{CD}
\end{equation}
Since $\varphi _n$ is an isomorphism by (a), we get that
$M(j_{n+1})$ is also an isomorphism.

(d) Write $\mathfrak R = \varinjlim (M(\iota_n),\iota_n)$. We have
from (c) a commutative diagram
$$\xymatrix@!=6.5pc{M(E_0,C^0) \ar[r]^{\iota_0} \ar[d]_{\text{id}} & M(E_{1}, C^{1}) \ar[r]^{\iota_{1}} \ar[d]_{M(j_1)}^{\cong} & M(E_{2},
C^{2}) \ar[d]_{M(j_2)}^{\cong}
\ar @{.>}[rr] & &  \mathfrak R\ar[d]\\
M(F_{0}, D^{0}) \ar[r]  & M(F_{1},D^{1}) \ar[r]  & M(F_{2}, D^{2})
\ar @{.>}[rr] & &  M(F_{\infty},D^{\infty})}
$$
such that all the vertical maps $M(j_n)\colon M(E_{n}, C^{n})\to
M(F_{n}, D^{n})$ are isomorphisms, so we get an isomorphism
$\mathfrak R\to M(F_{\infty},D^{\infty})$. This shows the result.
\end{proof}

\medskip

Given a finite presentation $\langle \mathcal X \mid \mathcal
R\rangle $ of a finitely presented abelian conical monoid $M$ as in
Definition \ref{def:sep-graph-of-cam}, we can associated to it the
refinement monoid
$$\mathfrak R (\mathcal X\mid \mathcal R )= M(F_{\infty}, D^{\infty}) \, ,$$
where $(F_{\infty}, D^{\infty})$ is the complete multiresolution of
the bipartite separated graph associated to the presentation
$\langle \mathcal X \mid \mathcal R\rangle $. Observe that this
construction depends strongly on the presentation, not just the
monoid $M$. For instance, we may consider two different
presentations $\langle a,b \mid  a=a, b=b\rangle $ and $\langle
a,b\mid a+b=a+b\rangle$ of the free abelian monoid on two generators
$F$. In the first presentation the resulting $\mathfrak R$ is just
$F$.  For the second presentation, the corresponding refinement
monoid $\mathfrak R$ is atomless, and is isomorphic to the monoid
$\mon{K[G]}$, where $G$ is the lamplighter group, see Example
\ref{exam:lamplighter}.

If we consider the standard presentation
$$M(E)=\langle a_v\mid a_v= \sum _{e\in r^{-1}(v)} a_{s(e)}
\quad (v\in E^0)\,  \rangle$$ of a  graph monoid $M(E)$ associated
to a non-separated finite graph $E$, then it can be shown that
$\mathfrak R = M(E)$. Since $\mathfrak R$ is a refinement monoid, we
obtain an alternative proof of \cite[Proposition 4.4]{AMP} for
finite graphs.

\section{The main construction}
\label{sect:main-construct}

This section contains our main construction on {\it algebras}. Let
$(E,C)$ be a finite bipartite separated graph, with $r(E^1)=E^{0,0}$
and $s(E^1)=E^{0,1}$. Let $\{ (E_n, C^n)\}_{n\ge 0}$ be the
canonical sequence of bipartite separated graphs associated to it
(see Definition \ref{cons:complete-multiresolution}(c)), and let
$B_n$ be the commutative, finite dimensional subalgebra of
$L(E_n,C^n)$ generated by $E_n^0$. Here $L(E_n,C^n)$ stands for the
Leavitt path *-algebra of the separated graph $(E_n,C^n)$ over a
fixed field with involution $(K,*)$.

\begin{theorem}
\label{thm:algebras} With the above notation, for each $n\ge 0$,
there exists a surjective $*$-algebra homomorphism $$\phi _n \colon
L(E_n, C^n)\twoheadrightarrow L(E_{n+1}, C^{n+1}).$$ Moreover, the
following properties hold:
\begin{enumerate}
\item[(a)]
 $\ker (\phi _n)$
is the ideal $I_n$ of $L(E_n,C^n)$ generated by all the commutators
$[ee^*, ff^*]$, with $e,f\in E_n^1$, so that $L(E_{n+1},
C^{n+1})\cong L(E_n, C^n))/I_n$.
\item[(b)] The restriction of $\phi_n$ to $B_n$ defines an injective homomorphism from $B_n$
into $B_{n+1}$.
\item[(c)] There is a commutative diagram
\begin{equation}
\label{eq:commu-daigram1}
 \begin{CD}
M(E_n, C^n) @>\iota_n>> M(E_{n+1}, C^{n+1})\\
@V{\cong}VV @VV{\cong}V  \\
\mon{L(E_n, C^n)} @>\mon{\phi _n}>> \mon{L(E_{n+1}, C^{n+1})}
\end{CD}\end{equation}
where the vertical maps are the canonical maps, which are
isomorphisms by \cite[Theorem 4.3]{AG}.
\end{enumerate}
\end{theorem}

\begin{proof} Set $A_n=L(E_n, C^n)$. Define $\phi _n\colon A_n\to
A_{n+1}$ on vertices $u\in E_n^{0,0}$ by the formula
$$\phi_n (u)= \sum _{(x_1,\dots ,x_{k_u})\in \prod_{i=1}^{k_u} X^u_i} v(x_1,\dots
,x_{k_u}), $$ where $C_u=\{ X_1^u,\dots , X^u_{k_u} \}$, and by $\phi
_n(w)=w$ for all $w\in E_n^{0,1}$. For an arrow $x_i\in X^u_i$,
define
$$\phi _n(x_i)= \sum _{x_j\in X^u_j, j\ne i} (\alpha
^{x_i}(x_1,\dots , \widehat{x_i},\dots ,x_{k_u}))^* \, ,$$ where
$\alpha ^{x_i}(x_1,\dots , \widehat{x_i},\dots ,x_{k_u})=\alpha
^{x_i}(x_1,\dots , x_{i-1}, x_{i+1}, \dots  ,x_{k_u})$. We have to
show that the defining relations of $L(E_n,C^n)$ are preserved by
$\phi _n$. It is easily checked that (V) and (E) are preserved by
$\phi _n$. To see that (SCK1) is preserved by $\phi _n$, take $u\in
E_n^{0,0}$ and $x_i,x_i'\in X^u_i$ for some $i$. Observe that
$$ \phi _n(x_i)^*\phi _n(x_i') = \sum _{x_j,y_j\in X_j^u, j\ne i}
\alphaxi \alpha ^{x_i'} (y_1,\dots ,\widehat{x_i'},\dots ,
y_{k_u})^* .$$ If $x_i\ne x_i'$, then $v(x_1,\dots, x_i,\dots
,x_{k_u})v(y_1,\dots , x_i',\dots , y_{k_u})=0$ for all $x_j,y_j\in
X_j^u$, $j\ne i$, and so $\phi _n(x_i)^*\phi_n(x_i')=0$. If
$x_i=x_i'$ then $v(x_1,\dots, x_i,\dots ,x_{k_u})v(y_1,\dots ,
x_i,\dots , y_{k_u})=0$ except when $x_j=y_j$ for all $j\ne i$, and,
hence recalling that $X(x_i)\in C_{s(x_i)}^{n+1}$, we get
$$\phi _n (x_i)^* \phi_n (x_i')= \sum _{x_j\in X_j^u, j\ne i}
\alphaxi \alphaxi ^* =\sum _{z\in X(x_i)} zz^*= s(x_i)
 \, ,$$
as desired. Finally we check that (SCK2) is also preserved by $\phi
_n$. Take $u\in E_n^{0,0}$ and $X^u_i\in C^n_u$. Then
\begin{align*}
\sum _{x_i\in X^u_i} \phi _n(x_i)\phi _n(x_i)^* & = \sum _{x_i\in
X_i^u} \sum _{x_j\in
X^u_j, j\ne i} \alphaxi ^*\alphaxi\\
&  =\sum _{(x_1,\dots, x_{k_u})\in \prod_{j=1}^{k_u}X_j^u}
v(x_1,\dots, x_{k_u}) =\phi _n(u)\, ,
\end{align*} as wanted.

To check surjectivity, it is enough to check that all vertices and
edges in $E_{n+1}$ belong to the image of $\phi _n$. Clearly, $w\in
\phi _n(A_n)$ for all $w\in E_{n+1}^{0,0}$. Now take $x_i\in X^u_i$,
for $i=1,2,\dots ,k_u$. We have
\begin{equation}
\label{eq:xixi*-formula} \phi_n(x_i)\phi _n(x_i)^*= \sum _{y_j\in
X_j^u, j\ne i} v(y_1,\dots , x_i, \dots , y_{k_u}) \, .
\end{equation}
Therefore, we get
\begin{equation}
\label{eq:vx1xku-formula} v(x_1,x_2,\dots ,x_{k_u}) = \phi
_n((x_1x_1^*)(x_2x_2^*)\cdots (x_{k_u}x_{k_u}^*)) \, ,
\end{equation}
showing that $v(x_1,x_2,\dots ,x_n)\in \phi _n (A_n)$. Moreover, we
get
\begin{equation}
\label{eq:alsoalphsinim} \alphaxi ^* = v(x_1,\dots ,x_{k_u}) \phi
_n(x_i) \, ,
\end{equation}
which, together with (\ref{eq:vx1xku-formula}) gives that $\alphaxi
\in \phi _n(A_n)$. This concludes the proof that $\phi _n$ is
surjective.

We now proceed to show (a), (b), (c).

(a) It follows from (\ref{eq:xixi*-formula}) that $\{ \phi
_n(ee^*)\mid e\in E_n^1 \}$ is a family of commuting projections in
$L(E_{n+1}, C^{n+1})$. Therefore all commutators $[ee^*, ff^*]$ are
contained in the kernel of $\phi _n$, and we obtain an induced
surjective map $\ol{\phi}_n\colon A_n/I_n\to A_{n+1}$. We define a
map $\gamma _n\colon A_{n+1}\to A_n/I_n$ by $\gamma _n(w) =w$ for
$w\in E_{n+1}^{0,0}$, and
$$\gamma_n (v(x_1,\dots , x_{k_u})) = (x_1x_1^*) (x_2x_2^*) \cdots
(x_{k_u}x_{k_u}^*),$$
$$\gamma _n(\alphaxi )= x_i^*(x_1x_1^*)(x_2x_2^*)\cdots
(x_{k_u}x_{k_u}^*)$$ for $(x_j)\in \prod_{j=1}^{k_u} X_j$. Using the
commutativity in $A_n/I_n$ of the set of projections $\{ ee^* \mid
e\in E_n^1 \}$, and the defining relations of $L(E_n,C^n)$, it is
easy to check that the defining relations of $A_{n+1}= L(E_{n+1},
C^{n+1})$ are preserved by $\gamma _n$. Clearly, $\gamma _n$ is the
inverse of $\ol{\phi}_n$, and so we get that $\ker (\phi _n)= I_n$.

(b) This follows from the definition of $\phi _n$.

(c) Denote by $\tau _n \colon M(E_n,C^n)\to \mon{A_n}$ the natural
map sending $a_v$ to $[v]$, for $v\in E_n^0$. This map is an
isomorphism by \cite[Theorem 4.3]{AG}. Recall that $\iota _n=
\varphi _n^{-1}\circ M(\iota_{V_n}) = \psi_n \circ M(\iota_{V_n})$
has been defined in the proof of Lemma \ref{lem:refinement}(b).

If $w\in E_n^{0,1}$ then $(\tau _{n+1}\circ \iota _n) (a_w) = [w]=
(\mon{\phi _n}\circ \tau _n )(a_w)$. If $u\in E_n^{0,0}$, then it
follows from the definition of $\psi _n$, formula
(\ref{eq:psi-well-defined}) and the definition of $\phi _n$ that
$$(\tau _{n+1}\circ \iota _n) (a_u)= \sum _{(x_j)\in \prod
_{j=1}^{k_u} X_j^u} [v(x_1,\dots ,x_{k_u})] = (\mon{\phi _n}\circ
\tau _n)(a_u)\, ,$$ so that we get the commutativity of the diagram
(\ref{eq:commu-daigram1}).
\end{proof}

We now study in more detail the relationship between the different
layers. To simplify the notation, we will write $D_n=
F^{0,n}=E_n^{0,0}$ for all $n\ge 0$.

Note that, for $n\ge 2$ we have a surjective map $r_n\colon D_{n}\to
D_{n-2}$ given by $r_n (v(x_1,\dots x_{k_u}))= u$, where $u\in
D_{n-2}$ and $x_i\in X^u_i$, and where, as usual, $C_u=\{X^u_1,\dots
, X^u_{k_u} \}$. For $n=2m$, we thus obtain a surjective map $\gotr
_{2m}=r_2\circ r_4\circ \cdots \circ r_{2m}\colon D_{2m}\to D_0$.
Similarly, we have a map $\gotr_{2m+1}= r_3\circ r_5 \circ \cdots
\circ r_{2m+1} \colon D_{2m+1}\to D_1$. We call $\gotr (v)$ the {\it
root} of $v$. Observe that we have
\begin{equation}
\label{eq:Ddisjintunion} D_{2n}=\bigsqcup _{v\in D_0}
\gotr_{2n}^{-1}(v); \qquad D_{2n+1}=\bigsqcup _{v\in D_1}
\gotr_{2n+1}^{-1}(v)
\end{equation}

\begin{lemma}
\label{lem:firststep}
\begin{enumerate}
\item For all $n\ge 1$ and $v\in D_{2n}$ we have $|s^{-1}(v)|=|C_v|=
|C_{\gotr _{2n}(v)}| $.
\item For all $n\ge 0$ and $w\in D_{2n+1}$ we have $|s^{-1} (w) |=
|C_ w| = | C_{\gotr _{2n+1} (w)} |$.
\end{enumerate}
\end{lemma}

\begin{proof}
(1) By induction, it suffices to show that $|s^{-1}(v)| = |C_v| =
|C_{r_{2n} (v)}|$ for every $v\in D_{2n}$, $n\ge 1$. Let $v\in
D_{2n}$, $n\ge 1$, and write $v=v(x_1,\dots ,x_{k_u})$, where
$x_i\in X_i^u$ and $C_u=\{ X^u_1,\dots , X^u_{k_u} \}$. Note that
$s^{-1}(v)$ has exactly $k_u$ elements, namely $$\alpha
^{x_1}(x_2,\dots , x_{k_u}), \dots , \alpha ^{x_{k_u}}(x_1,\dots
,x_{k_u-1}).$$ Thus $|s^{-1}(v)|= k_u= |C_u|= |C_{r_{2n}(v)}|$. On
the other hand, by definition $C_v=\{X(y) \mid y\in
s_{E_{2n-1}}^{-1}(v) \}$, so that $|C_v|= |s^{-1}(v)|=|C_u|$, as
desired.

(2) This is proved exactly as in (1).
\end{proof}

\begin{remark}
\label{rem:enumerations} Fix an enumeration $\{X^u_1, \dots
,X^u_{k_u} \}$ for the elements of $C_u$, where $u\in D_0= E^{0,0}$.
This gives an enumeration of the elements in $s^{-1}(v)$ for $v\in
r_2^{-1} (u)$. Indeed, for such an element $v$ we have
$v=v(x_1,\dots ,x_{k_n})$, with $x_i \in X^u_i$, and we can set $z_i
= \alphaxi $, so that we obtain the enumeration $\{ z_1,\dots
,z_{k_u}\}$ of $s^{-1}(v)$. This in turn gives an enumeration of the
elements of the set $C_v$, namely $C_v =\{X^v_1, \dots ,X^v_{k_u}
\}$, where $X^v_i= X(z_i)$ for all $i$. Obviously this translates to
all the vertices $v\in D_{2n}$, so that we obtain canonical
enumerations of the elements of $s^{-1}(v)$ and of the elements of
$C_v$.
\end{remark}

Define $\Phi _n=\phi_{n-1}\circ \dots \circ \phi _0\colon L(E,C)\to
L(E_n, C^n)$.

\begin{lemma}
\label{lem:form-of-xiinstep2n} Let $u\in D_0$ and set
$C_u=\{X^u_1,\dots X^u_{k_u} \}$. For each $n\ge 1$ and each $i$ with $1\le i\le k_u$, 
there exists a
partition
$$\gotr _{2n}^{-1} (u) =\bigsqcup _{x_i\in X^u_i} Z_{2n}(x_i) \, ,$$
such that, for each $x_i\in X_i^u$, we have
\begin{equation}
\label{eq:Phi2nformula} \Phi _{2n} (x_i)= \sum _{v\in Z_{2n}(x_i)}
\sum _{x\in X^v_i} x \, ,
\end{equation}
where the enumeration of the elements in $C_v$, $v\in D_{2n}$, is
the canonical one described in Remark \ref{rem:enumerations}.
\end{lemma}

\begin{proof}
We first prove the case $n=1$. Observe that $\phi _0(x_i)= \sum
_{y\in X(x_i)} y^*$. The elements $y\in X(x_i)$ are of the form
$\alphaxi $, and we set
$$Z_2(x_i) = \{ s(y)\mid y\in X(x_i) \}= \{v(x_1,\dots ,x_i,\dots
,x_{k_u})\mid x_j\in X_j^u, j\ne i \} .$$ Note that $s(y)\ne s(y')$
for $y,y'\in X(x_i)$ with $y\ne y'$. Observe that we get a partition
$$\gotr _2^{-1} (u)=\bigsqcup _{x_i\in X_i^u} Z_2(x_i) .$$
Let $v=v(x_1,\dots ,x_i, \dots , x_{k_u})\in Z_2(x_i)$. Then in the
canonical enumeration of the elements of $s^{-1}(v)$, say
$\{z_1,\dots ,z_{k_u}\}$, the element $\alphaxi $ corresponds to
$z_i$, and there is no other element of $X(x_i)$ having source equal
to $v$. The corresponding element $X(y)=X(z_i)$ of $C_v$ is the
element labelled as $X_i^v$ (see Remark \ref{rem:enumerations}).
Therefore we get
$$\Phi _2(x_i) = \phi _1( \sum _{y\in X(x_i)} y^*) =
\sum _{y\in X(x_i)} \sum _{x\in X(y)} x =  \sum _{v\in Z_2(x_i)}
\sum _{x\in X_i^v} x \, .$$

Assume now that $n\ge 1$ and that (\ref{eq:Phi2nformula}) holds.
Note that
\begin{equation}
\label{eq:phi2nPhi2nxi}
 \phi _{2n}(\Phi _{2n}(x_i)) = \sum _{v\in Z_{2n}(x_i)}
\sum _{x_i'\in X^v_i} \sum _{y\in X(x_i')} y^* .
\end{equation}
Observe that $\bigsqcup _{x_i'\in X^v_i} \{ s(y) \mid y\in X(x_i')
\} =r_{2n+2}^{-1}(v)$, and that $s(y)\ne s(y')$ for $y\ne y'$ and
$y,y'\in \bigcup _{v\in Z_{2n}(x_i)}\bigcup _{x_i'\in X_i^v}
X(x_i')$. Set
$$Z_{2n+2}(x_i)= \bigsqcup _{v\in Z_{2n}(x_i)}r_{2n+2}^{-1}
(v)=r_{2n+2}^{-1} (Z_{2n}(x_i)). $$ Clearly we get a partition
$\gotr_{2n+2}^{-1} (u)=\bigsqcup _{x_i\in X^u_i} Z_{2n+2}(x_i)$.
Moreover, arguing as in the case $n=1$ we get
\begin{align*}\Phi_{2n+2} & (x_i)  = \phi _{2n+1} (\sum _{v\in Z_{2n}(x_i)}
\sum _{x_i'\in X^v_i} \sum _{y\in X(x_i')} y^*)\\ &  = \sum _{v\in
Z_{2n}(x_i)} \sum _{x_i'\in X^v_i} \sum _{y\in X(x_i')} \sum _{x\in
X(y)} x= \sum _{v'\in Z_{2n+2}(x_i)} \sum _{x\in X_i^{v'}}x \, ,
\end{align*} completing the induction step.
\end{proof}

\begin{lemma}\label{lem:internalBnnplus1}
Let $u\in D_0$ and set $C_u=\{X_1^u,\dots X^u_{k_u} \}$. For $x_i\in
X_i^u$, $n\ge 0$ and $b\in B_n$ we have
\begin{equation*}
\Phi _{n+1}(x_i)\phi_n(b) \Phi_{n+1}(x_i)^*\in B_{n+1},\qquad
\Phi_{n+1}(x_i)^*\phi _n(b) \Phi _{n+1}(x_i)\in B_{n+1}.
\end{equation*}
Moreover, if $b$ is a projection in $B_n$, then both $\Phi
_{n+1}(x_i)\phi_n(b) \Phi_{n+1}(x_i)^*$ and $\Phi_{n+1}(x_i)^*\phi
_n(b) \Phi _{n+1}(x_i)$ are projections in $B_{n+1}$.
\end{lemma}

\begin{proof}
Note that it is enough to check that $\phi _n(\Phi _n(x_i) w\Phi
_n(x_i)^*)\in B_{n+1}$ and $\phi_n (\Phi _n(x_i)^*w\Phi _n(x_i))\in
B_{n+1}$ for all vertices $w\in E_n^0$.

We start with the cases $n=0$ and $n=1$. For $n=0$ the result is
clear by (SCK1) and (\ref{eq:xixi*-formula}). For $n=1$, observe
that $\phi _0(x_i) w\phi_0 (x_i)^*$ is $0$ except when $w=s(x_i)$
and in this case $\phi_0(x_i) w\phi (x_i)^*= \phi _0(x_ix_i^*)\in
B_1$ by (\ref{eq:xixi*-formula}). On the other hand,
$\phi_0(x_i)^*w\phi_0(x_i)$ is $0$ except for $w=v(x_1,\dots ,x_i,
\dots x_{k_u})$ for some $x_j\in X_j^u$, $j\ne i$. In this case we
have
$$\phi _0(x_i)^*v(x_1,\dots ,xŽ_{k_u})\phi_0 (x_i)=
\alphaxi \alphaxi^* \, ,$$ so that $\phi _1(\phi
_0(x_i)^*v(x_1,\dots ,xŽ_{k_u})\phi_0 (x_i))\in B_2$ again by
(\ref{eq:xixi*-formula}).

Now we will show the result for an even index $2n$, with $n\ge 1$.
By (\ref{eq:Phi2nformula}), we have
$$\Phi_{2n} (x_i) w\Phi _{2n}(x_i) ^*= (\sum _{v\in Z_{2n}(x_i)}
\sum _{x\in X^v_i} x) w( \sum _{v\in Z_{2n}(x_i)} \sum _{x\in X^v_i}
x)^*.$$ Since $n\ge 1$, the projections $\{s(x)\mid x\in \bigcup
_{v\in Z_{2n}(x_i)} X_i^v\}$ are pairwise orthogonal, and we
conclude that the above term is $0$ except when $w= s(x)$ for a,
necessarily unique, $ x\in \bigcup _{v\in Z_{2n}(x_i)} X_i^v $. In
this case ,we get $$\Phi_{2n} (x_i) w\Phi _{2n}(x_i) ^*= xx^* ,$$
and so $\phi_{2n} (\Phi_{2n} (x_i) w\Phi _{2n}(x_i) ^*) =
\phi_{2n}(xx^*)\in B_{2n+1}$ by (\ref{eq:xixi*-formula}).

On the other hand, note that $\Phi_{2n} (x_i)^* v\Phi _{2n} (x_i)
=0$ except for $v\in Z_{2n} (x_i)$. If $v\in Z_{2n}(x_i)$, the we
get
$$\Phi_{2n} (x_i)^* v\Phi _{2n} (x_i)
= (\sum _{x\in X_i^v} x^*)( \sum _{y\in X_i^v} y)= \sum _{x\in
X_i^v} x^*x=\sum _{x\in X_i^v} s(x)\in B_{2n} \, .
$$
We conclude that $\Phi_{2n}(x_i)$ is a partial isometry in $A_{2n}$
with initial projection $\Phi_{2n}(x_i)^*\Phi _{2n}(x_i)= \sum
_{v\in Z_{2n}(x_i)} \sum _{x\in X^v_i} s(x)$ and final projection
$\Phi _{2n}(x_i)\Phi_{2n}(x_i)^*= \sum _{v\in Z_{2n}(x_i)}v $. This
shows the result for the index $2n$.

For an odd integer $2n+1$, with $n\ge 1$, use
(\ref{eq:phi2nPhi2nxi}) and an argument similar to the above to
conclude the result. In this case $\Phi _{2n+1}(x_i)$ is a partial
isometry in $A_{2n+1}$ with initial projection $\Phi _{2n+1}
(x_i)^*\Phi _{2n+1}(x_i)= \sum _{v\in Z_{2n}(x_i)} \sum _{x\in
X^v_i} s(x) $ and final projection $\Phi _{2n+1}(x_i)
\Phi_{2n+1}(x_i)^*=\sum _{v\in Z_{2n+2}(x_i)} v $.
\end{proof}

\begin{remark}
 \label{rem:partial-isometry-at-large}
 Recall that, in the context of separated graphs, the elements of the mutiplicative semigroup $U$ 
 of $L(E,C)$ generated by the canonical set of partial isometries $E^1\cup (E^1)^*$ are not partial isometries
 in general. A similar problem arises when we consider the C*-algebra $C^*(E,C)$ of the separated graph $(E,C)$.
 Lemma \ref{lem:internalBnnplus1} says that, for each element $u$ of $U$, there is some
 $m\ge 1$ such that the image of $u$ in $A_m=L(E_m, C^m)$ {\it is} a partial isometry in $A_m$. This is 
 an essential point in our construction.
 \end{remark}

\begin{theorem}
\label{thm:commutatorsstepn} Let $U$ be the subsemigroup of the
multiplicative semigroup of $A_0=L(E,C)$ generated by $\{x,x^*:x\in
E^1 \}$. For $w\in U$ set $e(w)=ww^*$. For $n\ge 1$, let $J_n$ be
the ideal of $A_0$ generated by all the commutators $[e(w), e(w')]$,
with $w,w'$ elements of $U$ which are products of $\le n$ generators
$\{x,x^*:x\in E^1 \}$. Then $\ker (\Phi _n)=J_n$ so that $A_n\cong
A_0/J_n$.
\end{theorem}

\begin{proof}
Observe that, for $n=1$ this follows from Theorem
\ref{thm:algebras}(a).

Write $U_n$ for the set of elements in $U$ which can be written as a
product of $\le n$ generators. Using Lemma
\ref{lem:internalBnnplus1} it is very simple to show inductively
that $\Phi _n (e(w))= \Phi _n(w)\Phi _n(w)^*$ is a projection in
$B_n$ for any $w\in U_n$. Since $B_n$ is a commutative algebra, we
deduce that $[e(w), e(w')]\in \ker (\Phi _n)$ if $w,w'\in U$ are
products of $\le n$ generators. This shows that $J_n\subseteq \ker
(\Phi _n) $ for all $n\ge 1$.

Before we prove the reverse inclusion, we need to establish a claim.

\medskip

\noindent {\it Claim:} (1) If $z\in E_{2n}^1$, then there exist
$x\in E^1$ and $w_1,\dots w_t\in U_{2n}$ such that 
\begin{equation}
\label{eq:zinE2n} z= \Phi _{2n} (xe(w_1)e(w_2) \cdots e(w_t)) .
\end{equation}

(2) If $z\in E^1_{2n+1}$, then there exist $x\in E^1$ and
$w_1,\dots w_t\in U_{2n+1}$ such that 
\begin{equation}
\label{eq:zinE2nplus1} z= \Phi _{2n+1} (x^*e(w_1)e(w_2) \cdots
e(w_t)) .
\end{equation}

\medskip

\noindent {\it Proof of Claim.} The claim is proved by induction. It
trivially holds for $z\in E_0^1$, and it holds for $z\in E^1_1$ by
(\ref{eq:alsoalphsinim}) and (\ref{eq:vx1xku-formula}). Assume that
(2) holds for all $z\in E_{2n-1}^1$, for some $n\ge 1$. Then we will
show that (1) holds for all $z\in E_{2n}^1$. So, let $z\in
E_{2n}^1$. Then there exist $x_j\in X_j^u$, $j=1,\dots ,k_u$, for
some $u\in D_{2n-1}=E^{0,0}_{2n-1}$, such that $z= \alphaxi $, and it
follows from (\ref{eq:alsoalphsinim}) that
$$z= \phi _{2n-1}(x_i^*)v(x_1,\dots ,x_{k_u}) .$$
By induction hypothesis, we have that there are $x\in E^1$ and
$w_1,\dots ,w_t\in U_{2n-1}$ such that  $x_i = \Phi _{2n-1}(x^*e(w_1)\dots e(w_t))$. 
It
follows that
\begin{equation}
\label{eq:tran-tran} z= \Phi _{2n}(x^*e(w_1)\cdots e(w_t)) ^*
v(x_1,\dots ,x_{k_u}).
\end{equation}
Since $x, w_1, \dots , w_t\in U_{2n-1}$ we have that $\Phi_{2n-1}(xx^*)$, $\Phi_{2n-1}(e(w_j))$, $j=1,\dots ,t$,
is a (commuting) set of projections in $B_{2n-1}\subseteq A_{2n-1}$, and so
\begin{align*}
\Phi_{2n-1} & (x^*e(w_1)\cdots e(w_t))=\Phi_{2n-1}(x^*e(w_1)(xx^*)e(w_2)\cdots (xx^*)e(w_t)(xx^*))\\
\notag & =\Phi_{2n-1}(e(x^*w_1)\cdots e(x^*w_t)x^*).
\end{align*}
Therefore,we get 
\begin{equation}
 \label{eq:tran-tran2}
\Phi_{2n}(x^*e(w_1)\cdots e(w_t))^*= \Phi_{2n}(xe(x^*w_t)\cdots e(x^*w_1)) 
\end{equation}
On the other hand, using  the
induction hypothesis, (\ref{eq:vx1xku-formula}) and an argument as before, we see that
there are $w'_1,\dots w_s'$ in $U_{2n}$ such that
\begin{equation}
\label{eq:tran-tran3} v(x_1,\dots ,x_{k_u})=\Phi
_{2n}(e(w'_1)\cdots e(w_s')).
\end{equation}
Substituting (\ref{eq:tran-tran2}) and (\ref{eq:tran-tran3}) in
(\ref{eq:tran-tran}) we get
$$z= \Phi _{2n} (x e(x^*w_t)\cdots e(x^*w_1)e(w_1')\cdots e(w_s'))\,
, $$ showing (\ref{eq:zinE2n}). A similar argument shows that
(1) implies (2) for the index $2n+1$. This concludes the proof of
the Claim.

\medskip

Now we follow with the proof of the theorem. We show the reverse
inclusion $\ker (\Phi _m)\subseteq J_m$ by induction on $m$. 
For $m=1$, the result follows from Theorem \ref{thm:algebras}(a).
Assume that the result holds for some $m\ge 1$ and let us show that
it also holds for $m+1$. Observe that, by Theorem
\ref{thm:algebras}(a),  $\ker{\Phi_{m+1}}/\ker{\Phi _m}$ is the
ideal of $A_m$ generated by $[zz^*, (z')(z')^*]$, where $z,z'\in
E_m^1·$. Assume first that $m=2n$ for some $n\ge 0$ and that
$z,z'\in E_{2n}^1$. By the first part of the Claim, there exist then
$x,x'\in E^1$ and $w_1,\dots w_t,w'_1,\dots ,w_{t'}'\in U_{2n}$ such
that 
\begin{equation*}
 z= \Phi _{2n} (xe(w_1)e(w_2) \cdots e(w_t)), \qquad z'=\Phi _{2n} (x'e(w_1')e(w_2')\cdots  e(w_{t'}')).
 \end{equation*}
Since $\ker{\Phi_{2n}}=J_{2n}$ by induction hypothesis, and
$J_{2n}\subseteq J_{2n+1}$, it is enough to check that
$$\Big[ xe(w_1) \cdots e(w_t)e(w_t)^*\cdots e(w_1)^*x^*\, ,\,
x'e(w_1')\cdots e(w_{t'}')e(w_{t'}')^*\cdots e(w_1')^*(x')^*
\Big]\in J_{2n+1}.$$ We have
\begin{align*}
& \Big[  xe(w_1) \cdots   e(w_t)e(w_t)^*\cdots e(w_1)^*x^*+J_{2n}\,
,\, x'e(w_1')\cdots e(w_{t'}')e(w_{t'}')^*\cdots
e(w_1')^*(x')^*+J_{2n}
\Big]\\
& = \Big[ xe(w_1) \cdots e(w_t)x^*+J_{2n}\, ,\, x'e(w_1')\cdots
e(w_{t'}')(x')^* + J_{2n} \Big]\\
&  = \Big[ xe(w_1)(x^*x)e(w_2)\cdots (x^*x)e(w_t)x^*+J_{2n}\, ,\,
x'e(w_1')(x')^*x'e(w_2')\cdots (x')^*x'e(w_{t'}')(x')^*+J_{2n}
\Big]\\
& = \Big[ e(xw_1)e(xw_2) \cdots e(xw_t)+J_{2n}\, ,\,
e(x'w_1')e(x'w_2')\cdots e(x'w_{t'}')+J_{2n} \Big]\\
& \subseteq \sum _{i,j} A_0[e(xw_i),e(x'w'_j)]A_0 +J_{2n} \subseteq
J_{2n+1} \, ,
\end{align*}
where we have used Leibnitz rule (i.e. $[xy,z]=[x,z]y+x[y,z]$) for the first containment. A
similar argument, using part (2) of the Claim instead of part (1),
gives the proof of the inductive step for $m=2n+1$. This completes
the proof of the theorem.
\end{proof}

Write $A_{\infty}= \varinjlim A_n$ ,and $B_{\infty}= \varinjlim
B_n$. We have a commutative diagram as follows:

$$\xymatrix@!=6.5pc{B_0 \ar[r] \ar[d] & B_1 \ar[r] \ar[d] & B_2 \ar[d]
\ar @{.>}[rr] & &  B_{\infty}\ar[d]\\
A_0 \ar[r]^{\phi _0}  & A_1 \ar[r]^{\phi_1}  & A_2 \ar @{.>}[rr] & &
A_{\infty}}
$$

All the maps $B_n\to B_{n+1}$ are injective and all the maps $A_n\to
A_{n+1}$ are surjective.

\begin{notation}
\label{notation:Lab} We write $\Lab (E,C)= A_{\infty}.$ Observe
that, by Theorem \ref{thm:commutatorsstepn},  $\Lab (E,C)\cong
L(E,C)/J$, where $J=\bigcup _{n=1}^{\infty} J_n$ is the ideal
generated by all the commutators $[e(w),e(w')]$, with $w,w'\in U$.
 If we use in
the above construction the full graph C*-algebras $C^*(E_n,C^n)$
instead of the Leavitt path algebras $L(E_n, C^n)$, we obtain in the
limit a certain C*-algebra, that we denote by $\mathcal O (E,C)$, in
analogy with the notation introduced in \cite{AEK}.
\end{notation}

\begin{corollary}
\label{cor:VmonoidofLab} We have $\mon{\Lab (E,C)}\cong \varinjlim
(M(E_n, C^n),\iota _n )\cong M(F_{\infty}, D^{\infty})$.
Consequently the natural map $M(E,C)\to \mon{\Lab (E,C) }$ is a
refinement of $M(E,C)$.
\end{corollary}

\begin{proof}
This follows immediately from the continuity of the functor
$\mathcal V$, Theorem \ref{thm:algebras}(c) and Lemma
\ref{lem:refinement}(d).
\end{proof}

\begin{remark}
\label{rem:LabsubO} Assume that $K=\C$, the complex field. Then the
commutative $*$-algebra $B_{\infty}$ is an algebraic direct limit of
finite-dimensional commutative C*-algebras $B_n$, with injective
maps $\phi_n|_{B_n}\colon B_n\to B_{n+1}$. Since each $B_n$ sits as
a $C^*$-subalgebra of $C^*(E_n,C^n)$ (cf. \cite[Theorem
3.8(1)]{AG2}), it follows that the $C^*$-completion $\mathcal B _0$
of $B_{\infty}$ sits as a C*-subalgebra of $\mathcal O (E,C)$.
\end{remark}

Recall the following definition from \cite{ExelPJM}. (The definition
has been conveniently modified to work in the purely algebraic
situation as well as in the C*-context.)

\begin{definition}
\label{def:tame}
  A set $S$ of partial isometries in a $*$-algebra $\mathcal A$
  is said to be {\it tame} if
every element of $U=\langle S\cup S^*\rangle$, the multiplicative
semigroup generated by $S\cup S^*$, is a partial isometry, and the
set $\{ e(u)\mid u\in U \}$ of final projections of the elements of
$U$ is a commuting set of projections  of $\mathcal A$.
\end{definition}

We are now ready to show that $\Lab (E,C)$ is the universal algebra
generated by a {\it tame} set of partial isometries $\underline{x}$,
for $x\in E^1$ satisfying the defining relations of $L(E,C)$, which
can be re-stated completely in terms of the elements in $E^1\cup
(E^1)^*$, indeed in terms of the final projections $e(u)$ of these
elements, as follows:

(PI1) $e(x)e(y)=\delta_{x,y} e(x)$, for $x,y\in X$, $X\in C$.

(PI2) $e(x^*)=e(y^*)$ if $x,y\in E^1$ and $s(x)=s(y)$.

(PI3) $\sum _{x\in X} e(x) =\sum _{y\in Y} e(y)$ for all $X,Y\in
C_v$, $v\in E^{0,0}$.

\smallskip

\noindent For $w\in E^{0,1}$ and $x\in E^1$ with $s(x)=w$, denote
the projection $e(x^*)$ by $p_w$. Similarly, for $v\in E^{0,0}$,
denote by $p_v$ the projection corresponding to $\sum _{x\in X}
e(x)$, for $X\in C_v$. (Note that these projections are well-defined
by (PI1)--(PI3).)

\smallskip

(PI4) $p_vp_{v'}= 0$ for every pair $v,v'$ of different vertices in
$E^0$, and $\sum _{v\in E^0} p_v=1$.

\medskip

Note that the condition $\sum _{v\in E^0} p_v=1$ can be deduced from
the other conditions, because $E$ is a finite graph. It is easily
seen that for a finite bipartite separated graph satisfying our
standing assumptions that $s(E^1)=E^{0,1}$ and $r(E^1)=E^{0,0}$, the
algebra $L(E,C)$ is just the universal algebra generated by a set of
partial isometries $x\in E^1$ subject to the  above relations
(PI1)--(PI4). A similar statement holds for the full graph
C*-algebra $C^*(E,C)$.

\begin{theorem}
\label{thm:universal-tame} The algebra $\Lab (E,C)$ (respectively,
the C*-algebra  $\mathcal O (E,C)$) is the universal $*$-algebra
(respectively, universal C*-algebra) generated by a tame set of
partial isometries $\{ \underline{x}\mid x\in E^1\}$ satisfying the
relations (PI1)--(PI4).
\end{theorem}

\begin{proof} Let $U$ be the multiplicative semigroup of $L(E,C)$
generated by $E^1\cup (E^1)^*$, and let $U_n$ be the set of those
elements of $U$ which can be written as a product of $\le n$
generators. Then $\Lab (E,C)= L(E,C)/J$, where $J= \bigcup
_{n=1}^{\infty} J_n$ and $J_n$ is the ideal generated by the
commutators $[e(w), e(w')]$, with $w,w'\in U_n$ (Theorem
\ref{thm:commutatorsstepn}). If $w\in U_n$ then it was observed in
the proof of Theorem \ref{thm:commutatorsstepn} that $e(w)$ is a
projection in $A_n=L(E,C)/J_n$. It is rather easy
 to show, by induction on $n$, that this implies that every element of $U_n$ is a partial
 isometry in $L(E,C)/J_n$. Denoting by $\underline{u}$ the image of $u\in U$ through
 the canonical projection $L(E,C)\to \Lab (E,C)$, we thus have that
all elements $\underline{u}$ are partial isometries and the final
projections of these elements commute with each other. It follows
that $\{ \underline{x}\mid x\in E^1\}$ is a tame set of partial
isometries in $\Lab (E,C)$. Since $J$ is the ideal generated by all
the commutators $[e(w), e(w')]$, with $w,w'\in U$, it is clear that
$\Lab (E,C)$ is the universal algebra generated by a tame set of
partial isometries satisfying conditions (PI1)--(PI4), as claimed.

The same proof works for $\mathcal O (E,C)$.
\end{proof}

By Stone duality (\cite{Stone}), there is a zero-dimensional metrizable compact
space $\Omega(E,C)$ such that the Boolean algebra of projections in
$B_{\infty}$ is isomorphic to the Boolean algebra of clopen subsets
of $\Omega(E,C)$. A convenient way of describing this space is as
follows. Consider the commutative $AF$-algebra $\mathcal B_0$, which
is the $C^*$-completion of the locally matricial algebra
$B_{\infty}$ (taking $K=\C$). Then there is a unique
zero-dimensional metrizable compact topological space $\Omega(E,C)$
such that $\mathcal B _0\cong C(\Omega(E,C))$. Observe that we have
a basis of clopen sets for the topology of $\Omega(E,C)$ which is in
bijective correspondence with the vertices of the graph
$F_{\infty}$. Note that for any field $K$, the algebra $B_{\infty}$
is the algebra $C_K(\Omega(E,C))$ of continuous functions $f\colon
\Omega(E,C)\to K$, where $K$ is endowed with the discrete topology, see
\cite[Corollaire 1]{Keimel}.

\medskip

We record for later use a fundamental property of the algebra
$B_{\infty}$.

\begin{proposition}
\label{prop:ewgeneratesbinfty} The algebra $B_{\infty}$ is generated
by all the projections $e(w)$, for $w\in U$.
\end{proposition}

\begin{proof}
 The algebra $B_{\infty}$ is generated by the images under the limit maps
 of the projections $v\in D_n$,
 for $n\ge 0$. Now if $v\in D_0$ then it follows from (SCK2) that
 $v$ is a sum of projections $e(x)$, for $x\in X$, $X\in
 C_v$. Similarly, (SCK1) gives that the projections in $D_1$ are of
 the form $e(x^*)$.
Suppose now that $v\in D_{n+1}$ with $n\ge 1$. Then  $v=v(x_1,\dots
,x_{k_u})$, where $x_i\in X_i^u$, $C_u=\{X_1^u,\dots ,X_{k_u}^u \}$
and $u\in D_{n-1}$. Then it follows from either (\ref{eq:zinE2n}) or
(\ref{eq:zinE2nplus1}), and  equation (\ref{eq:vx1xku-formula}) that
$v(x_1,\dots ,x_{k_u})$ is a product of projections of the form
$\Phi _{n}(e(w))$, where $w\in U_{n}$.
\end{proof}

\section{Dynamical systems}
\label{sect:dyn-systems}

\begin{definition}
\label{def:partrep} Let $G$ be a group with identity element $1$. A
{\it partial representation} of $G$ on a unital $K$-algebra
$\mathcal A$ is a map $\pi \colon G\to \mathcal A$ such that $\pi
(1)=1_{\mathcal A}$ and
$$\pi (g)\pi (h)\pi (h^{-1})=\pi (gh)\pi
(h^{-1}), \quad \pi(g^{-1})\pi (g)\pi (h)=\pi (g^{-1})\pi (gh)$$ for
all $g,h\in G$. If $G$ is a free group on a set $X$ and $|g|$
denotes the length of $g\in G$ with respect to $X$, we say that a
partial representation $\pi$ is {\it semi-saturated} in case $\pi
(ts)=\pi (t)\pi (s)$ for all $t,s\in G$ with $|ts|=|t|+|s|$.
\end{definition}

\begin{definition}
\label{def:partaction} Let $G$ be a group and let $\mathcal B$ be an
associative $K$-algebra (possibly without unit). A {\it partial
action} of $G$ on $\mathcal B$ consists of a family of two-sided
ideals $\mathcal D _g$ of $\mathcal B$ ($g\in G$) and algebra
isomorphisms $\alpha _g\colon \mathcal D_{g^{-1}}\to \mathcal D_g$
such that

(i)  $\mathcal D_1= \mathcal B$ and $\alpha _1$ is the identity map
on $\mathcal B$;

(ii) $\alpha _h (\mathcal D _{h^{-1}} \cap \mathcal D_g)=\mathcal D
_{h} \cap  \mathcal D _{hg}$;

(iii) $\alpha _g\circ \alpha _h (x)= \alpha _{gh}(x)$ for each $x\in
\alpha _h^{-1}(\mathcal D _h \cap \mathcal D _{g^{-1}})$.
\end{definition}

Let $(E,C)$ be a finite bipartite separated graph, with
$r(E^1)=E^{0,0}$ and $s(E^1)=E^{0,1}$, as in Definition
\ref{def:bipartitesepgraph}. Throughout this section, we will denote
by $\mathbb F$ the free group on the set $E^1$ of edges of $E$.

\begin{proposition}
\label{prop:univ-partialreps} The algebra $\Lab (E,C)$ is the
universal $*$-algebra for semi-saturated partial representations of
$\mathbb F$ satisfying the relations (PI1)-(PI4). That is, there is
a semi-saturated partial representation $\pi \colon \mathbb F\to
\Lab (E,C)$ sending each $x\in E^1$ to the element $\underline{x}$
of $\Lab (E,C)$, and, moreover given any semi-saturated partial
representation $\rho$ of $\mathbb F$ on a unital $*$-algebra $B$
such that the partial isometries $\rho (x)$, $x\in E^1\sqcup
(E^1)^*$, satisfy relations (PI1)--(PI4), there is a unique
$*$-algebra homomorphism $\varphi \colon \Lab (E,C)\to B$ such that
$\rho = \varphi \circ \pi $.
\end{proposition}

\begin{proof}
Since $\{ \underline{x}\mid x\in E^1 \}$ is a tame set of partial
isometries in $\Lab (E,C)$, it follows from \cite[Proposition
5.4]{ExelPJM} that there exists a (unique) semi-saturated partial
representation $\pi \colon \mathbb F\to \Lab (E, C)$ such that $\pi
(x) = \underline{x}$ for all $x\in E^1$.

The universality follows from Theorem \ref{thm:universal-tame}, just
as in \cite[2.4]{AEK}.
\end{proof}

Let $\pi \colon G\to \mathcal A$ be a partial representation
of a group $G$ on a unital $K$-algebra $\mathcal A$. Then the
elements $\varepsilon _g =\pi (g)\pi (g^{-1})$ $(g\in G)$ are
commuting idempotents such that
\begin{equation}
\label{eq:vareps} \pi (g)\varepsilon _h=\varepsilon _{gh}\pi
(g),\qquad  \varepsilon _h\pi (g)= \pi (g)\varepsilon _{g^{-1}h},
\end{equation}
see \cite[page 1944]{DokExel}. Let $\mathcal B$ be the subalgebra of
$\mathcal A$ generated by all the $\varepsilon _g$ $(g\in G)$, and
for $g\in G$ set $\mathcal D _g = \varepsilon _g \mathcal B$. By
\cite[Lemma 6.5]{DokExel}, the maps $\alpha _g^{\pi}\colon \mathcal
D_{g^{-1}}\to \mathcal D _g$ $(g\in G)$ given by
$$\alpha _g^{\pi }(b)= \pi (g)b\pi (g)^*$$
are isomorphisms of $K$-algebras which determine a partial action
$\alpha ^{\pi}$ of $G$ on $\mathcal B$.

In the case where $\mathcal A= \Lab (E,C)$ and $\pi \colon \mathbb
F\to \Lab (E,C)$ is the partial representation described in
Proposition \ref{prop:univ-partialreps}, we set $\varepsilon
_w=e(w)\in B_{\infty }$, and we observe that, by Proposition
\ref{prop:ewgeneratesbinfty}, the algebra $\mathcal B$ generated by
all the projections $e(w)$ $(w\in \mathbb F)$ is precisely the
commutative algebra $B_{\infty}$. It follows that $\mathcal D
_w=B_{\infty}e(w)$, the ideal of $B_{\infty}$ generated by $e(w)$.

As noted above, it follows that the maps
$$\phi_w\colon B_{\infty}e(w^*)\longrightarrow B_{\infty}e(w)$$
defined by $\phi _w (b)= \pi (w)b\pi (w)^*$, give a partial action
of $\mathbb F$ on $B_{\infty}$, namely $\phi =\alpha ^{\pi}$.

We now recall the notion of a crossed product by a partial action,
see e.g. \cite[Definition 1.2]{DokExel}.

\begin{definition}
\label{def:skewgrouprin}
 Given a partial action
$\alpha$ of a group $G$ on a $K$-algebra $\mathcal B$, the {\it
crossed product} $\mathcal B \rtimes_{\alpha} G$ is the set of all
finite formal sums $\{ \sum _{g\in G} a_g \delta _g\mid a_g\in
\mathcal D_g\}$, where $\delta _g$ are symbols. Addition is defined
in the obvious way, and multiplication is determined by $(a_g\delta
_g)\cdot (b_h\delta _h)= \alpha _g (\alpha
_{g^{-1}}(a_g)b_h)\delta_{gh}$.
\end{definition}

Note that, given a partial action of $G$ on a unital $K$-algebra
$\mathcal B$, we have a partial representation $G\to \mathcal
B\rtimes _{\alpha} G$, defined by $g\mapsto 1_g\delta_g$ (see
\cite[Lemma 6.2]{DokExel}).

We need a further definition.

\begin{definition}
\label{def:univpartialaction} A {\it partial $(E,C)$-action} on a
unital commutative $*$-algebra $\mathcal B$ consists of a collection
of projections $\{ e(u)\mid u\in E^1\cup (E^1)^*\}$ satisfying the
conditions (PI1)-(PI4), together with a collection of $*$-algebra
isomorphisms $\alpha _x\colon e(x^*)\mathcal B\to e(x) \mathcal B$
for every $x\in E^1$.

If $\mathcal B_1$ and $\mathcal B_2$ are two unital commutative
$*$-algebras with partial $(E,C)$ actions $\alpha $ and $\beta $
respectively, then an {\it equivariant homomorphism} from $\mathcal
B_1$ to $\mathcal B_2$ is a $*$-algebra homomorphism $\psi \colon
\mathcal B_1\to \mathcal B_2$ such that $\psi
(e_{\alpha}(u))=e_{\beta} (u)$ for all $u\in (E^1)\cup (E^1)^*$ and
such that the following diagram
\begin{equation}
\begin{CD}
e_{\alpha}(x^*)\mathcal B_1 @>{\psi |}>>
 e_{\beta}(x^*)\mathcal B _2\\
@V{\alpha _x}VV  @V{\beta _x}VV  \\
e_{\alpha}(x)\mathcal B_1 @>{\psi |}>> e_{\beta}(x) \mathcal B_2
\end{CD}
\end{equation}
commutes for all $x\in E^1$.

A commutative $*$-algebra $\mathcal B$ with a partial $(E,C)$-action
$\alpha$ is {\it universal for partial $(E,C)$-actions} if given any
commutative $*$-algebra $\mathcal C$ with a partial $(E,C)$-action
$\beta$, there is a unique equivariant homomorphism $\psi \colon
\mathcal B\to \mathcal C$
\end{definition}

\begin{lemma}
\label{lem:extendingECactions} Any partial $(E,C)$-action on a
unital commutative $*$-algebra $\mathcal B$ can be extended to a
partial action $\alpha$ of $\mathbb F$ on $\mathcal B$ such that
$\mathcal D _g=e(g)\mathcal B$, with $e(g)$ a projection in
$\mathcal B$ for all $g\in \mathbb F$. Moreover, the corresponding
partial representation $\pi\colon \mathbb F\to \mathcal B\rtimes
_{\alpha}\mathbb F$, defined by $\pi (g)= e(g)\delta_g$ is
semi-saturated.
\end{lemma}

\begin{proof}
Let $\alpha$ be a partial $(E,C)$-action on a unital commutative
$*$-algebra $\mathcal B$.  Let ${\bf B}$ be the Boolean algebra of
projections of $\mathcal B$, and let $\mathcal I ({\bf B})$ be the
inverse semigroup of partially defined $*$-algebra isomorphisms with
domain and range of the form ${\bf b}\mathcal B$, with ${\bf b}\in
{\bf B}$. Then $\alpha $ can be extended to a map (also denoted by
$\alpha$) from $\mathbb F$ to $\mathcal I ({\bf B})$ by the rule
$$\alpha _{g_1g_2\cdots g_k}= \alpha_{g_1}\cdot \alpha_{g_2}\cdot
\cdots \cdot \alpha_{g_k} \, , $$ where $g_1g_2\cdots g_k$ is the
reduced form of an element of $\mathbb F$ and $\alpha
_{x^{-1}}=\alpha_x^{-1}$ for all $x\in E^1$. It is easily checked
that this extension provides a partial action on $\mathcal B$. See
also \cite[Example 2.4]{McCla4}.

Finally we show that the representation $\pi \colon \mathbb F\to
\mathcal B \rtimes _{\alpha} \mathbb F$ is semi-saturated. Let
$g,h\in \mathbb F$ be such that $|gh|=|g|+|h|$. The identity
$e(gh)\delta_{gh}=(e(g)\delta _g)(e(h)\delta _h)$ is equivalent to
$e(gh)\le e(g)$. Since $|gh|=|g| |h|$, it is evident from the
definition that $e(gh)\le e(g)$, because $\alpha _{gh}=\alpha
_g\cdot \alpha _h$ in $\mathcal I ({\bf B})$.
\end{proof}

Now we can obtain the basic relation between universality for
partial representation obeying (PI1)-(PI4) and universality for
partial $(E,C)$ actions.

\begin{theorem}
\label{thm:equivalenceofuniversalprops} A partial representation
$\pi$ of $\mathbb F$ on $\mathcal A$ is universal for semi-saturated
partial representations satisfying (PI1)-(PI4) if and only if
$\mathcal A \cong \mathcal B\rtimes_{\alpha} \mathbb F$, where
$(\mathcal B,\alpha )$ is a universal partial $(E,C)$-action on the
commutative $*$-algebra $\mathcal B$.
\end{theorem}

\begin{proof}
Let $\pi$ be a semi-saturated partial representation of $\mathbb F$
on the unital $*$-algebra $\mathcal A$ satisfying (PI1)-(PI4),  and
assume that $(\mathcal A,\pi)$ is universal for semi-saturated
partial representations of $\mathbb F$ satisfying (PI1)-(PI4). Let
$\mathcal B$ be the commutative subalgebra of $\mathcal A$ generated
by the projections $e(g)=\pi (g)\pi (g)^*$. Then $(\mathcal
B,\alpha^{\pi})$ is a partial $(E,C)$-action on $\mathcal B$, and
$\mathcal A\cong \mathcal B\rtimes _{\alpha^{\pi}} \mathbb F$ by
\cite[Proposition 6.8]{DokExel} and the universal property of $\pi$.
Assume that $(\mathcal C, \beta)$ is a partial $(E,C)$-action on a
unital commutative $*$-algebra $\mathcal C$. By Lemma
\ref{lem:extendingECactions}, $\beta$ can be extended to a partial
action, also denoted by $\beta$, of $\mathbb F$ on $\mathcal C$ such
that $\mathcal D _g=e_{\beta}(g)\mathcal C$ for some projection
$e_{\beta}(g)$,  for all $g\in \mathbb F$, and such that the induced
partial representation $\rho\colon \mathbb F\to \mathcal
C\rtimes_{\beta} \mathbb F$ is semi-saturated. Since $\beta $ is a
partial $(E,C)$-action, the partial isometries $\rho (x)$, $x\in
E^1$, satisfy the relations (PI1)-(PI4), so that there is a unique
$*$-homomorphism $\varphi \colon \mathcal B\rtimes _{\alpha} \mathbb
F\to \mathcal C\rtimes_{\beta} \mathbb F$ such that $\varphi (\pi
(g))= e_{\beta}(g)\delta _g$ for all $g\in \mathbb  F$. In
particular, we see that $$\varphi (e(g))= \varphi (\pi (g)\pi (g)^*)
= (e_{\beta}(g)\delta_g)(e_{\beta}(g)\delta _g)^*=
e_{\beta}(g)\delta_e . $$ Moreover, for $x\in E^1$ and $b\in
e(x^*)\mathcal B$, we have
$$\varphi (\alpha _x(b))=\varphi (\pi(x)b\pi (x)^*)= \rho (x)
\varphi (b)\rho (x)^*= \beta_x( \varphi (b))\, ,$$ showing that the
restriction of $\varphi $ to $\mathcal B$ is an equivariant
homomorphism from $\mathcal B$ to $\mathcal C$. Observe that, being
$\pi $ a semi-saturated representation of $\mathbb F$ on $\mathcal
A$, we have that $\alpha ^{\pi}(g_1g_2\cdots g_n)=
\alpha^{\pi}(g_1)\cdot \alpha^{\pi}(g_2)\cdot \cdots \cdot
\alpha^{\pi} (g_n)$ in $\mathcal I ({\bf B})$, the inverse semigroup
associated to the Boolean algebra ${\bf B}$ of $\mathcal A$, where
$g_1g_2\cdots g_n$ is the reduced form of an element in $\mathbb F$.
Since the corresponding statement for the partial representation
$\rho$ on $\mathcal C\rtimes _{\beta} \mathbb F$ also holds, we see
that $\varphi (e(g))= e_{\beta}(g)$ for all $g\in \mathbb F$. Now
$\mathcal B$ is generated as an algebra by the projections $e(g)$,
with $g\in \mathbb F$, and so we deduce from the above that
$\varphi|_{\mathcal B}$ is the unique equivariant homomorphism from
$\mathcal B$ to $\mathcal C$. We have shown that $(\mathcal B,\alpha
)$ is a universal $(E,C)$-partial action.

Suppose conversely that $(\mathcal B, \alpha)$ is a universal
partial $(E,C)$-action on the unital commutative $*$-algebra
$\mathcal B$. Let $\alpha$ denote also the extension to a partial
action of $\mathbb F$ on $\mathcal B$ described in Lemma
\ref{lem:extendingECactions}, and let $\pi \colon \mathbb F\to
\mathcal B\rtimes _{\alpha } \mathbb F$ be the corresponding
semi-saturated partial representation (Lemma
\ref{lem:extendingECactions}). Obviously, the unitaries  $\pi(x)$,
$x\in E^1$ satisfy (PI1)-(PI4). Let $\rho $ be a semi-saturated
partial representation of $\mathbb F$ on a unital $*$-algebra
$\mathcal A '$ such that the partial isometries $\rho (x)$, $x\in
E^1$, satisfy (PI1)-(PI4). Let $\mathcal B '$ be the subalgebra of
$\mathcal A '$ generated by the projections $\rho (g)\rho (g)^*$,
for $g\in \mathbb F$ and let $\beta$ be the canonical partial action
of $\mathbb F$ on $\mathcal B '$. Observe that, in particular, the
partial action $\beta$ induces a partial $(E,C)$-action on $\mathcal
B '$. Therefore there is a unique equivariant homomorphism $\psi
\colon \mathcal B\to \mathcal B'$. Since both $\pi $ and $\rho $ are
semi-saturated we see that $\psi (e_{\alpha}(g)) = e_{\beta}(g)$ for
all $g\in \mathbb F$, and that $\psi (\alpha _g(b))=\beta_g( \psi
(b))$ for all $b\in e_{\alpha} (g^{-1}) \mathcal B$. We can thus
define a $*$-homomorphism $\varphi \colon \mathcal B \rtimes
_{\alpha} \mathbb F\to \mathcal B '\rtimes _{\beta} \mathbb F$ such
that $\varphi (b\delta_g)= \psi (b) \delta _g$ for all $b\in
e_{\alpha} (g) \mathcal B$. Composing the map $\varphi \colon
\mathcal B \rtimes _{\alpha} \mathbb F\to \mathcal B '\rtimes
_{\beta} \mathbb F$ with the natural $*$-homomorphism  $\mathcal
B'\rtimes _{\beta} \mathbb F \to \mathcal A'$ sending $b\delta _g$
to $b\rho (g)$ (\cite[Proposition 6.8]{DokExel}), we obtain a
$*$-homomorphism $\varphi ' \colon \mathcal B\rtimes _{\alpha}
\mathbb F\to \mathcal A'$ such that $\varphi ' (\pi (g))=\rho (g)$
for all $g\in \mathbb F$. If $\varphi ''$ is another such
$*$-homomorphism, then it must agree with $\varphi '$ on the
commutative subalgebra $\mathcal B$, and also $$\varphi
''(e_{\alpha}(g)\delta _g)= \varphi ''(\pi (g)) =\rho (g)=\varphi'
(e_{\alpha}(g)\delta _g) \, ,$$ so that $\varphi'=\varphi ''$.

We have shown that the partial representation $\pi$ of $\mathbb F$
on $\mathcal B\rtimes _{\alpha} \mathbb F$ is universal for partial
representations satisfying (PI1)-(PI4). This concludes the proof of
the theorem.
\end{proof}

Note that we have a canonical partial $(E,C)$ action on the unital
commutative $*$-algebra $B_{\infty}$, given by the projections
$e(x)=xx^*$, $e(x^*)= s(x)$, and by the isomorphisms $\alpha
_x\colon e(x^*)B_{\infty} \to e(x) B_{\infty}$ given by conjugation
by $\underline{x}$, for all $x\in E^1$.

The following result follows immediately from Proposition
\ref{prop:univ-partialreps} and Theorem
\ref{thm:equivalenceofuniversalprops}.

\begin{corollary}
\label{cor:unipropBinfty} The canonical partial $(E,C)$-action
$\alpha $ on the unital commutative $*$-algebra $B_{\infty}$ is
universal.
\end{corollary}

Similarly, we obtain the corresponding result for the completion
$\mathcal B_0$ of $B_{\infty}$:

\begin{corollary}
\label{cor:unipropmathcalB0} The canonical partial $(E,C)$-action
$\alpha $ on the unital commutative C*-algebra $\mathcal B _0$ is
universal for partial $(E,C)$-actions on commutative C*-algebras.
\end{corollary}

We can now use duality to obtain the universal property of the
topological space $\Omega (E,C)$. Recall that the symbol $\sqcup $
is used to indicate disjoint unions.

\begin{definition}
\label{def:univtopspaces} An {\it $(E,C)$-dynamical system} consists
of a compact Hausdorff space $\Omega$ with a family of clopen
subsets $\{ \Omega_v \}_{v\in E^{0}}$ such that
$$\Omega = \bigsqcup _{v\in E^{0}} \Omega_v , $$ and, for each $v\in E^{0,0}$,
a family of clopen subsets $\{ H_x \}_{x\in r^{-1}(v)}$ of
$\Omega_v$, such that
$$\Omega_v= \bigsqcup _{x\in X} H_x \qquad \text{ for all } X\in C_v , $$
together with a family of homeomorphisms
$$\theta_x\colon \Omega_{s(x)} \longrightarrow H_x $$
for all $x\in E^1$.

Given two $(E,C)$-dynamical systems $(\Omega, \theta)$, $(\Omega ',
\theta')$, there is an obvious definition of equivariant map
$f\colon (\Omega, \theta) \to (\Omega',\theta ')$, namely $f\colon
\Omega \to \Omega '$ is {\it equivariant} if $f(\Omega_w)\subseteq
\Omega'_w$ for all $w\in E^{0}$, $f(H_x)\subseteq H'_x$ for all
$x\in E^1$ and $f(\theta_x(y))= \theta'_x (f(y))$ for all $y\in
\Omega_{s(x)}$.

We say that an $(E,C)$-dynamical system $(\Omega ,\theta)$ is {\it
universal} in case there is a unique continuous equivariant map from
every $(E,C)$-dynamical system to $(\Omega ,\theta )$.
\end{definition}

For instance if $(E,C)=(E(m,n), C(m,n))$, then this is exactly the
definition of an $(m,n)$-dynamical system, given in \cite[Definition
3.1]{AEK}.

Using the (contravariant) equivalence between commutative unital
C*-algebras and compact Hausdorff topological spaces we get
immediately from Corollary \ref{cor:unipropmathcalB0}:

\begin{corollary}
\label{cor:univECdynsustem} The dynamical system
$(\Omega(E,C),\alpha ^*)$ is the universal $(E,C)$-dynamical system.
\end{corollary}

We also obtain the structure of $\Lab_K (E,C)$ and $\mathcal O
(E,C)$ as a crossed product:

\begin{corollary}
\label{cor:Lab+OECarecrossedprods}
\begin{enumerate}
\item For any field with involution $K$, we have
$$\Lab _K (E, C)\cong C_K(\Omega (E,C))\rtimes _\alpha
\mathbb F,$$ where $C_K(\Omega (E,C))$ is the algebra of $K$-valued
continuous functions on $\Omega (E,C)$, and $K$ is given the
discrete topology.
\item We have $$\mathcal O (E,C) \cong C(\Omega (E,C))\rtimes _\alpha \mathbb
F,$$ where $C(\Omega (E,C))$ is the C*-algebra of continuous
functions on $\Omega (E,C)$ (here $\C$ has the usual topology), and
$C(\Omega (E,C))\rtimes _\alpha \mathbb F$ denotes the full
C*-algebra crossed product.
\end{enumerate}
\end{corollary}

\section{The type semigroup}
\label{sect:typesemigroup}

Let $G$ be a group acting on a set $X$.  Tarski's Theorem can be
expressed as a result on the type semigroup $S(X,G)$, see
\cite[Theorem 9.1 and Corollary 9.2]{Wagon}. As Tarski shows in
\cite[Theorem 16.12]{Tarski}, his result holds also if we replace
total actions with partial actions. If $G$ acts by continuous
transformations on a topological space $X$, and $\mathbb D$ is a
$G$-invariant subalgebra of subsets of $X$, then we may consider the
corresponding type semigroup $S(X,G,\mathbb D) $ as in \cite{Kerr},
\cite{KN}, \cite{RS}. As we will notice in this section, we may as
well consider this semigroup for invariant subalgebras under a
partial $G$-action. An important case, considered in the above
references for global actions, is the case where $X$ is a
zero-dimensional metrizable compact space and $\mathbb D$ is the
subalgebra of $\mathcal P (X)$ consisting  of all the clopen subsets
of $X$. It is an open question in this case whether the analogue of
Tarski's Theorem holds. The proof of Tarski's Theorem is based on
two special properties of the type semigroup $S(X,G)$: the Schr\"
oder-Bernstein axiom (also named Cantor-Schr\" oder-Bernstein
property) $x\le y$ and $y\le x$ imply $x=y$, and the cancellation
property $nx=ny \implies x= y$. The Schr\" oder-Bernstein axiom is
known to hold for countably complete subalgebras. However the
cancellation law is known to fail even for complete subalgebras
\cite{Truss}.

Here we will show that for global actions of free groups on
zero-dimensional metrizable compact spaces, the type semigroup
$S(X,G,\mathbb K)$, where $\mathbb K$ denotes the subalgebra of all
clopen subsets of $X$, do not satisfy in general any cancellation
condition (see Theorem \ref{thm:main-typesemigroup2}). In particular
we show in Corollary \ref{cor:non-paradoxical2} the existence, for
each pair of positive numbers $(m,n)$ such that $1<m<n$, of a global
action of a finitely generated free group $\mathbb F$ on a
zero-dimensional metrizable compact space $\Omega$, such that
$\Omega$ contains a non-paradoxical but $(n,m)$-paradoxical clopen
subset. This is the first example of the failure of Tarski's Theorem
in this setting.

Our main construction gives in a natural way \emph{partial} actions
on zero-dimensional metrizable compact spaces. Using globalization
results from \cite{Abadie} we get corresponding global actions,
although on a locally compact, non-compact, Hausdorff space. Passing
finally to the one-point compactifications of these spaces, we
obtain the desired global actions on zero-dimensional metrizable
compact spaces.

\begin{definition}
\label{def:type-semigroup} Let $(\{ X_t\}_{t\in G}, \{ \theta _t
\}_{t\in G})$ be a partial action of a group $G$ on a set $X$. Let
$\mathbb D$ be a $G$-invariant subalgebra of $\mathcal P (X)$, with
$X_t\in \mathbb D$ for all $t\in G$. Let $S(X,G, \mathbb D)$ be the
set
$$ \Big\{ \bigcup _{i=1}^n A_i\times \{ i\} \mid A_i\in \mathbb D , n\in \N   \Big\}\Big/ \sim_{S} ,$$
where the equivalence relation $\sim _S$ is defined as follows: two
sets $A=\bigcup _{i=1}^n A_i\times \{i \}$ and $B= \bigcup_{j=1}^m
B_j\times \{j\}$ are equivalent, denoted by $A\sim _S B$, if there
exist $l\in \N$, $C_k\in \mathbb D$, $t_k\in G$, and natural numbers
$n_k,m_k$, $k=1,\dots , l$, such that $C_k\subseteq X_{t_k^{-1}}$, and
$$A= \bigsqcup _{k=1}^l C_k \times \{n_k \} ,\qquad B= \bigsqcup
_{k=1}^l \theta _{t_k} (C_k) \times \{ m_k \} .$$ The equivalence class
containing $A$ is denoted by $[A]$, and addition is defined by
$$\Big[ \bigcup _{i=1}^n A_i \times \{i \} \Big] + \Big[
\bigcup_{j=1 }^m B_j\times \{j \} \Big] = \Big[ \bigcup _{i=1}^n
A_i\times \{i \} \cup \bigcup_{j=1}^m B_j \times \{n+j\} \Big] .$$
Note that $S(X,G, \mathbb D )$ is a conical refinement monoid, with
$0=[\emptyset ]$. When $\mathbb D= \mathcal P (X)$,  $S(X,G, \mathbb
D)$ is denoted by $S(X,G)$.
\end{definition}

 A subset $E$ in $\mathbb D$ is said to be {\it $G$-paradoxical} (with
 respect to $\mathbb D$) in case the element $[E]$ of $S(X,G,
 \mathbb D )$ is properly infinite, that is, $2[E] \le [E]$.
 With these definitions we can state Tarski's Theorem  (\cite[Corollary
 9.2]{Wagon},
 \cite[Theorem 16.12]{Tarski}) as follows:

 \begin{theorem}
 \label{thm:TarskiThm}
Let $(\{ X_t\}_{t\in G}, \{ \theta _t \}_{t\in G})$ be a partial
action of a group $G$ on a set $X$, and let $E$ be a subset of $X$.
Then $E$ is $G$-non-paradoxical if and only if there exists a monoid
homomorphism $\mu \colon S(X,G)\to \mathbb R ^+ \cup \{\infty \}$
such that $\mu ([E])=1$.
\end{theorem}

The proof of Tarski's Theorem is based on the following properties
of $S(X,G)$ (see \cite[Theorem 3.5 and Theorem 8.7]{Wagon}):

 \smallskip

(a) {\it Schr\" oder-Bernstein}: Let $x,y\in S(X,G)$. If $x\le y$
and $y\le x$ then $x=y$.

\smallskip

(b) {\it $n$-cancellation}: Let $n\ge 1$ and $x,y\in S(X,G)$. If
$nx=ny$ then $x=y$.

\smallskip

For any semigroup satisfying these two properties, the two
conditions stated in Tarski's Theorem are equivalent, that is, if
$e$ is a nonzero element in an abelian monoid $S$ satisfying (a) and
(b) above, then $2e\nleq e$ in $S$ if and only if there is a monoid
homomorphism $\mu \colon S\to \mathbb R^+\cup \{\infty \}$ such that
$\mu (e) =1$. This is a consequence of a purely semigroup theoretic
result \cite[Theorem 9.1]{Wagon}. As observed in \cite[Theorem
5.4]{RS}, the two conditions stated in Tarski's Theorem are also
equivalent if the abelian semigroup $S$ satisfies the following
condition:

\smallskip

(c) {\it almost unperforation}: $(n+1)x\le ny \implies  x\le y$ for
every $n\ge 1$ and $x,y\in S$.

\smallskip

We are interested in the following concrete setting. Let $X$ be a
zero-dimensional metrizable compact Hausdorff space, and let
$\mathbb K=\mathbb K (X)$ be the subalgebra of $\mathcal P (X)$
consisting of all the clopen subsets of $X$. Let $(\{ X_t\}_{t\in
G}, \{ \theta _t \}_{t\in G})$ be a partial action of a group $G$ by
continuous transformations on $X$ such that $X_t\in \mathbb K$ for
all $t\in G$. Observe that $\mathbb K$ is then automatically
$G$-invariant. Then we want to know to what extent properties
(a),(b) or (c) hold in $S(X,G, \mathbb K)$. In particular we want to
know whether Tarski's Theorem holds in this context.

We have the following general result in this setting:

\begin{lemma}
\label{lem:mapfromStoV} Let $X$ be a totally disconnected compact
Hausdorff topological space, and let $\mathbb K$ be the subalgebra
of $\mathcal P (X)$ consisting of the clopen subsets of $X$. Let
$(\{ X_t\}_{t\in G}, \{ \theta _t \}_{t\in G})$ be a continuous
partial action of a group $G$ on $X$, such that $X_t\in \mathbb K$
for all $t\in G$. Let $K$ be a field, and let $C_K(X)$ be the
K-algebra of $K$-valued continuous functions on $X$, where $K$ is
given the discrete topology. Then there exists a canonical monoid
homomorphism
$$\Phi \colon S(X,G, \mathbb K)\longrightarrow \mon{C_K(X)\rtimes
_{\theta^*} G} $$ sending $\Big[\bigcup _{i=1}^n A_i\times \{ i \}
\Big]$ to $\sum _{i=1}^n [1_{A_i}\delta _e]$, for  $A_i\in \mathbb
K$, $i=1,\dots ,n$.
\end{lemma}

\begin{proof} Since additivity is clear,
we only have to prove
 that the map is compatible with the equivalence $\sim _S$ of
 Definition \ref{def:type-semigroup}. For this it is enough to show
 that if $C\in \mathbb K$ and $C\subseteq X_{t^{-1}}$ for some $t\in G$,
 then the projections
 $1_C\delta _e$ and $1_{\theta _t(C)}\delta _e$ are equivalent. Set
 $x=1_{\theta _t(C)} \delta _t$, and note that $x^*= 1_C \delta _{t^{-1}}$. Then
 one computes that $xx^* = 1_{\theta_t(C)}\delta _e$ and $x^*x= 1_C\delta _e$,
 showing the result.
\end{proof}

We do not know whether the map defined in the lemma is injective or
surjective in general, but we can show it is bijective for the
actions of $\mathbb F$ on the spaces $\Omega (E,C)$ introduced in
Section \ref{sect:dyn-systems}.

\begin{theorem}
\label{thm:typesemigroupforOmegaEC} Let $(E,C)$ be a finite
bipartite separated graph such that $s(E^1)=E^{0,1}$ and
$r(E^1)=E^{0,0}$, and let $(\Omega (E,C),\alpha^*)$ be the universal
$(E,C)$-dynamical system (see Corollary \ref{cor:univECdynsustem}).
Then the map
$$\Phi\colon S(\Omega (E,C), \mathbb F, \mathbb K)\longrightarrow \mon{C_K
(\Omega(E,C))\rtimes _\alpha \mathbb F } $$ is an isomorphism for
any field $K$.
\end{theorem}

\begin{proof}
By Corollary \ref{cor:Lab+OECarecrossedprods}(1) and Corollary
\ref{cor:VmonoidofLab}, we have
$$\mon{C_K
(\Omega(E,C))\rtimes _\alpha \mathbb F } \cong \mon{\Lab (E,C)}\cong
\varinjlim (M(E_n,C^n),\iota_n)\cong M(F_{\infty}, D^{\infty} ) .$$
Let $\Phi '$ be the composition of the map $\Phi$ with the
isomorphism $\mon{C_K (\Omega(E,C))\rtimes _\alpha \mathbb F } \cong
\varinjlim (M(E_n,C^n),\iota_n )$. It will suffice to show that
$\Phi '$ is an isomorphism.

Observe that there is a basis of clopen subsets of $\Omega (E,C)$
which is in bijective correspondence with the vertices in
$F_{\infty}$. Indeed, if $v\in D_n$ is a vertex in the $n$-th level,
then it represents a minimal projection in the commutative algebra
$B_n$. The image of this projection through the canonical limit map
$B_n\to B_{\infty}$ is a projection in $B_{\infty}$, and so
corresponds to a clopen subset of $\Omega (E,C)$. Each clopen subset
of $\Omega (E,C)$ is a finite (disjoint) union of clopen subsets
coming from the vertices of $E_n$ for some $n$. (Note that
$E_n^0=D_n\sqcup D_{n+1}$.)  The map $\Phi '$ sends the class
$[A_v]$ of the clopen set corresponding to a vertex $v\in D_n\sqcup
D_{n+1}$ to the class in the limit $\varinjlim (M(E_n,C^n),\iota_n)$
of the corresponding element $a_v$ of $M(E_n,C^n)$.

Put $S=S(\Omega (E,C),\mathbb F,\mathbb K )$. We are going to build
an inverse monoid homomorphism $\Psi \colon \varinjlim
(M(E_n,C^n),\iota_n) \to S$. To this end we need to define
compatible monoid homomorphisms $\Psi _n\colon M(E_n,C^n)\to S$. For
$v\in E_n^0$ define $\Psi _n(a_v) =[A_v]$, where, as before,  $A_v$
is the clopen subset of $\Omega (E,C)$ corresponding to $v$. If
$X\in C_v$, with $v\in E_n^{0,0}$, then we have to show that the
relation $a_v = \sum _{z\in X} a_{s(z)}$ is preserved by $\Psi _n$.
Now, it follows from either (\ref{eq:zinE2n}) or
(\ref{eq:zinE2nplus1}) that, for $z\in X$, there is an element
$u_z\in E^1\cup (E^1)^*$ such that the image of the projection
$zz^*$ in $\Lab (E,C)= A_{\infty}$ is a projection in $B_{\infty}$
which is of the form $\alpha _{u_z} (1_{A_{s(z)}})$, where the
clopen set $A_{s(z)}$ is contained in the domain of $\alpha
_{u_z}^*$. We will show this in the case where $n=2m$, and leave to the reader the verification in the case where $n$
is odd. If $n=0$, then $z=u_z$ and, using (\ref{eq:xixi*-formula}), the result is clear. Assume that 
$n=2m>0$. By (\ref{eq:zinE2n}) there is $u_z\in E^1$ and $w_1,\dots ,w_t\in U_{n}$ such that
$z= \Phi_{n}(u_ze(w_1)\cdots e(w_t))$.
As observed at the beginning of the proof of Theorem \ref{thm:commutatorsstepn}, the elements
$$\Phi_n(u_z^*u_z), \Phi_n(e(w_1)),\dots , \Phi_n (e(w_t))$$ 
are projections in $B_n$, so that in particular they are mutually
commuting projections in $A_n$. This implies that 
$$s(z)= z^*z =\Phi_n (e(u_z^*)e(w_1)\cdots e(w_t)),$$
and 
$$zz^*= \Phi_n(u_z)\Big( \Phi_n (e(u_z^*)e(w_1)\cdots e(w_t))\Big) \Phi_n(u_z)^*.$$
Let $\Phi_{i,\infty}\colon A_i\to A_{\infty}$ be the canonical map into the direct limit.
Note that, by (\ref{eq:xixi*-formula}), the element $\phi_n(zz^*)$ is a projection in $B_{n+1}$, so that
$\Phi_{n,\infty}(zz^*)=\Phi_{n+1,\infty}(\phi_n(zz^*))$ is a projection in $B_{\infty}$. 
Now, we have that  $\alpha_{u_z}$ is given by conjugation by $\Phi_{0,\infty}(u_z)=\Phi_{n,\infty}(\Phi_n(u_z))$
on  $\Phi_{n,\infty}(\Phi_n(e(u_z^*))) B_{\infty}$, which shows our claim.

\smallskip

Since $v=\sum _{z\in X} zz^*$ in $L(E_n,C^n)$, we
conclude that
\begin{equation}
\label{eq:equalityforAv} [A_v]=\sum _{z\in X}
[\alpha_{u_z}^*(A_{s(z)})]=\sum _{z\in X} [A_{s(z)}] ,
\end{equation}
showing that $\Psi _n$ preserves the relation $a_v=\sum _{z\in X}
a_{s(z)}$. To show compatibility with respect to $\iota _n\colon
M(E_n,C^n)\to M(E_{n+1}, C^{n+1})$, observe that, for $v\in
E_n^{0,1}$, we have
$$\Psi _{n+1} (\iota _n (a_v)) = \Psi _{n+1} (a_v)= [A_v]=\Psi _n(a_v).$$
If $v\in E_n^{0,0}$ then, by the proof of Lemma
\ref{lem:refinement}(a),(b), we have $\iota _n (a_v) =\sum _{z\in X}
a_{s(z)}$,
 where $X$ is any
element in $C_v$, so that using (\ref{eq:equalityforAv}), we get
$$\Psi _{n+1} (\iota _n (a_v))= \Psi _{n+1}\Big( \sum _{z\in X} a_{s(z)}\Big) =
\sum _{z\in X} [A_{s(z)}]= [A_v]= \Psi _n (a_v), $$ as desired.

The homomorphism $\Psi \colon \varinjlim (M(E_n,C^n), \iota_n)\to S$
just defined is clearly a two-sided inverse of $\Phi '$. This shows
the result.
\end{proof}

We can now show our first main result on the type semigroup.

\begin{theorem}
\label{thm:main-typesemigroup} Let $M$ be a finitely generated
abelian conical monoid. Then there exist a zero-dimensional
metrizable compact Hausdorff space $\Omega$ and a partial action of
a finitely generated free group $\mathbb F$ on $\Omega$ such that
there is a map $\iota \colon M \to S(\Omega,\mathbb F, \mathbb K)$
which is a refinement of $M$. In particular, the map $\iota \colon M
\to S(\Omega,\mathbb F, \mathbb K)$ is a unitary embedding.
\end{theorem}

\begin{proof} By Redei's Theorem (\cite{Freyd}), $M$ is finitely
presented. Let $\langle \mathcal X\mid \mathcal R \rangle $ be a
finite presentation of $M$, as in Definition
\ref{def:sep-graph-of-cam}, and let $(E,C)$ be the associated finite
bipartite separated graph, so that $M\cong M(E,C)$ (see Definition
\ref{def:sep-graph-of-cam}).

Let $\mathbb F$ be the free group on $E^1$, and let $(\Omega
(E,C),\alpha^*)$ be the universal $(E,C)$-dynamical system (see
Corollary \ref{cor:univECdynsustem}). By Theorem
\ref{thm:typesemigroupforOmegaEC}, we have that the canonical map
$$\Phi\colon S(\Omega (E,C), \mathbb F, \mathbb K)\longrightarrow \mon{C_K
(\Omega(E,C))\rtimes _\alpha \mathbb F } $$ is an isomorphism, and
by Corollary \ref{cor:Lab+OECarecrossedprods}(1) we have
$$\mon{C_K
(\Omega(E,C))\rtimes _\alpha \mathbb F } \cong \mon{\Lab (E,C)} .$$
Since the natural map $M(E,C)\to \mon{\Lab (E,C)}$ is a refinement
of $M(E,C)$ (Corollary \ref{cor:VmonoidofLab}), and $M\cong M(E,C)$,
we obtain the result.
\end{proof}

In particular, we obtain the following:

\begin{corollary}
\label{cor:non-paradoxical} There exist a partial action of a
finitely generated free group $\mathbb F$ on a zero-dimensional
metrizable compact space $\Omega$, and a non-$\mathbb F$-paradoxical
(with respect to $\mathbb K$) clopen subset  $A$ of $\Omega$  such
that $\mu (A)=\infty $ for every nonzero $\mathbb F$-invariant
finitely additive measure $\mu \colon \mathbb K\to [0, \infty]$.
\end{corollary}

\begin{proof}
Take positive integers $m,n$ such that $1<m<n$, and set
$(E,C)=(E(m,n), C(m,n))$. Set $\Omega := \Omega (E,C)$. Then there
is a canonical partial action of the free group $\mathbb F$ on $m+n$
generators on the zero-dimensional metrizable compact space
$\Omega$.  Moreover, it follows from Theorem
\ref{thm:main-typesemigroup} that there is a unitary embedding
$$\iota \colon \langle a \mid ma=na \rangle \longrightarrow S(\Omega,
\mathbb F, \mathbb K) .$$ Let $A=A_w$ be the clopen subset
corresponding to the vertex $w$ of the graph $E$ (see e.g. Example \ref{exam:m,ndyn-system}). Then $[A]$ is an
order-unit in $S(\Omega, \mathbb F, \mathbb K)$, and $m[A]=n[A]$.
Since $m<n$ this implies that $\mu (A)=\infty$ for every nonzero
finitely additive $\mathbb F$-invariant measure $\mu$ defined on
$\mathbb K$. Since $\iota (a) =[A]$ and $2a\nleq a$ in $\langle a
\mid ma=na \rangle $, we deduce that $2[A]\nleq [A]$ in $S(\Omega,
\mathbb F, \mathbb K)$ (use properties (3) and (1) in Definition
\ref{def:unitemb}). It follows that $A$ is not $\mathbb
F$-paradoxical (with respect to $\mathbb K$).
\end{proof}

We proceed now to extend the above results to the setting of global
actions. This is accomplished by the use of the globalization
techniques of \cite{Abadie}.

\medskip

Recall from \cite[Example 1.1]{Abadie}, that if
  $$
  \beta : G\times Y \to Y
  $$
  is an action of a given group $G$ on a set $Y$, and if $X\subseteq
Y$ is any subset, one may always restrict $\beta$ to $X$, regardless
of whether or not $X$ is $\beta$-invariant, obtaining a
\emph{partial action}, as follows: for each $g\in G$, set $X_g =
X\cap \beta_g(X)$, and observe that
  $$
  \beta_g(X_{g^{-1}}) =
  \beta_g\big(X\cap \beta_{g^{-1}}(X)\big) =
  \beta_g(X)\cap X = X_g.
  $$
  We may then define
  $$
  \alpha_g : X_{g^{-1}}\to X_g
  $$
  to be the restriction of $\beta_g$ to $X_{g^{-1}}$, and it may be
easily proven that $\alpha$ is a partial action of $G$ on $X$,
henceforth referred to as the \emph{restriction of $\beta$ to $X$}.

Conversely, if we start with a partial action $\alpha$ of $G$ on
$X$, a global (as opposed to partial) action $\beta$ on a set
$Y\supseteq X$ is called a \emph{globalization}, or an
\emph{enveloping action} for $\alpha$ \cite[Section 1]{Abadie}, if
$\alpha$ is the restriction of $\beta$ to $X$, and moreover
  $$
  Y = \bigcup_{g\in G} \beta_g(X).
  $$

Working in the category of topological spaces, as opposed to the
category of sets, we always assume that the transformations
$\alpha_g$ and $\beta_g$ are continuous and that $X$ is open in $Y$,
in which case we will clearly have that each $X_g$ is open in $X$.

It is our next goal to study the relationship between the type
semigroup of a partial action $\alpha$ and that of its globalization
or, more generally, any global action extending $\alpha$.

For this let us assume we are given a set $Y$ equipped with a fixed
nonempty collection ${\mathbb E}$ of subsets of $Y$, which is closed
under finite intersections, finite unions and relative complements.
In other words, ${\mathbb E}$ is a \emph{ring} of subsets of $Y$.
From now on we will refer to the members of ${\mathbb E}$ as
\emph{admissible} subsets.

Let us assume that we are also given a fixed admissible subset
$X\subseteq Y$, and let
  $$
  {\mathbb E}|_X = \{E\cap X: E\in{\mathbb E}\} = \{D\in {\mathbb E}:
D\subseteq X\}.
  $$
  It is readily verified that ${\mathbb E}|_X$ is a subalgebra of subsets of
$X$.

Let us now consider a global action $\beta$ of a given group $G$ on
$Y$, which we will suppose to be \emph{admissible} in the sense that
$\beta_g(E)$ is admissible for every admissible subset $E\subseteq
Y$. We may then consider the corresponding type semigroup
$S(Y,G,\mathbb E)$.

  Denoting by $\alpha$ the restriction of $\beta$ to $X$, it is clear
that each $X_g$ is admissible (relative to ${\mathbb E}|_X$) and
that $\alpha$ is an \emph{admissible partial action} in the sense
that, if $D$ is an admissible subset of some $X_{g^{-1}}$, then
$\alpha_g(D)$ is admissible.

  Given any $D\in{\mathbb E}|_X$, we shall denote the class of $D$ in
$S(X,G,{\mathbb E}|_X)$ by $[D]_X$.  Likewise, the class in
$S(Y,G,{\mathbb E})$ of any $E\in{\mathbb E}$ will be denoted by
$[E]_Y$.


\begin{proposition}
 \label{prop:InjectivityOfTypeMap{1.1}}
  The correspondence
  $$
  \phi: [D]_X \mapsto [D]_Y, \quad (D\in {\mathbb E}|_X),
  $$
  gives a well defined injective map from $S(X,G,{\mathbb E}|_X)$ to
$S(Y,G,{\mathbb E})$.
\end{proposition}

\begin{proof}
  We leave the easy verification of the well definedness of $\phi$ for
the reader.  In order to show that $\phi$ is injective, assume that
$D, D'\in {\mathbb E}|_X$ are such that
  $$
  [D]_Y = [D']_Y.
  $$
  Then there are sets $D_1,\ldots,D_n\in {\mathbb E}$, such that
$D=\bigsqcup_{i=1}^n D_i$ and, for suitable group elements
$g_1,\ldots, g_n$, one has that $D'=\bigsqcup_{i=1}^n
\beta_{g_i}(D_i)$.

Since the $D_i$ are subsets of $D$, and hence also of $X$, it is
evident that $D_i\in{\mathbb E}|_X$.  Moreover
  notice that
  $$
  \beta_{g_i}(D_i) \subseteq
  D' \cap \beta_{g_i}(X) \subseteq
  X \cap \beta_{g_i}(X) = X_{g_i},
  $$
  so $D_i\subseteq X_{g_i^{-1}}$, and it thus makes sense to speak of
$\alpha_{g_i}(D_i)$.  It is also clear that
  $$
  D'=\bigsqcup_{i=1}^n \alpha_{g_i}(D_i),
  $$
  so
  $
  [D]_X = [D']_X.
  $
  The general case, that is, when $D$ and $D'$ are of the form
  $\bigcup_{i=1}^n D_i \times \{i\}$, may be treated similarly.
  \end{proof}

A case of interest occurs when $Y$ is a Hausdorff topological space
and
  $$
  {\mathbb E} = {\mathbb K}(Y) := \{E\subseteq Y: E \hbox{ is compact
and open}\}.
  $$
  Given a compact-open subset $X\subseteq Y$, notice that
  $$
  {\mathbb K}(Y)|_X = {\mathbb K}(X).
  $$

  {\begin{proposition}
 \label{prop:BijectivityOfTypeMap{1.2}}
  Suppose we are given a continuous global action $\beta$ of a
discrete group $G$ on a Hausdorff space $Y$.  Also let $X\subseteq
Y$ be a compact-open subset such that
  $$
  Y = \bigcup_{g\in G} \beta_g(X).
  $$
  Letting $\alpha$ be the restriction of $\beta$ to $X$, one has that
the map
  $$
  \phi :S\big(X,G,{\mathbb K}(X)\big) \to S\big(Y,G,{\mathbb
K}(Y)\big)
  $$
  described above is bijective.
\end{proposition}

\begin{proof}
  Given $E\in{\mathbb K}(Y)$, we may view $\{\beta_g(X)\}_{g\in G}$ as
an open covering of $E$. Observing that $E$ is compact, there are
finitely many group elements $g_1,\ldots,g_n$, such that
  $
  E\subseteq \bigcup_{i=1}^n \beta_{g_i}(X).
  $
  Setting
  \def\cro{\hfill\cr}
  \def\croI{\vrule height 12pt width 0pt \cro}
  \begin{align*}
  E_1 & = E\cap \beta_{g_1}(X), \\ 
  E_2 & = \big(E\cap \beta_{g_2}(X)\big) \setminus E_1, \\ 
  E_3 & = \big(E\cap \beta_{g_3}(X)\big) \setminus (E_1 \cup E_2), \\ 
  & \vdots \vrule depth 5pt width 0pt  \\
  E_n & = \big(E\cap \beta_{g_n}(X)\big) \setminus (E_1 \cup \ldots \cup
E_{n-1}) 
  \end{align*}
  one has that $E= \bigsqcup_{i=1}^n E_i$.  Observing that each
$E_i\subseteq \beta_{g_i}(X)$, we have that
  $$
  D_i := \beta_{g_i^{-1}}(E_i) \subseteq X.
  $$
  Moreover,
  $
  \phi\big([D_i]_X\big) = [D_i]_Y = [E_i]_Y,
  $
  so
  $$
  \phi\Big(\sum_{i=1}^n [D_i]_X\Big) =
  \sum_{i=1}^n \phi\big([D_i]_X\big) =
  \sum_{i=1}^n [E_i]_Y = [E]_Y.
  $$
  Since $S\big(Y,G,{\mathbb K}(Y)\big)$ is generated by the collection
of all elements of the form $[E]_Y$, as above, we deduce that $\phi$
is surjective.  The fact that $\phi$ is injective follows from
  Proposition \ref{prop:InjectivityOfTypeMap{1.1}}.
  \end{proof}

\begin{corollary}
\label{cor:globalizedactions}
  Let $X$ be a compact Hausdorff topological space which is totally
disconnected in the sense that its topology admits a basis of
compact-open sets.  Also let $\alpha$ be a continuous partial action
of the group $G$ on $X$.
  Then $\alpha$ admits a unique globalization $\beta$ on a totally
disconnected, locally compact Hausdorff space $Y$, and moreover
  $S\big(X,G,{\mathbb K}(X)\big)$ is isomorphic to $S\big(Y,G,{\mathbb
K}(Y)\big)$.
\end{corollary}

 \begin{proof}
  By \cite[Theorem 1.1]{Abadie} we know that $\alpha$ admits a unique
globalization $\beta$ acting on a topological space $Y$.  Using
\cite[Proposition 1.2]{Abadie} we may easily show that $Y$ is
Hausdorff.  Since $X$ is open in $Y$, and $Y=\bigcup_{g\in G}
\beta_g(X)$, it is clear that $Y$ is  locally compact and totally
disconnected.

Finally, observing that $X$ is a compact-open subset of $Y$, we may
use Proposition \ref{prop:BijectivityOfTypeMap{1.2}}
  to conclude that
  $S\big(X,G,{\mathbb K}(X)\big)$ and $S\big(Y,G,{\mathbb K}(Y)\big)$
are isomorphic.
\end{proof}

We can now extend Theorem \ref{thm:typesemigroupforOmegaEC} to the
globalized actions. For a totally disconnected locally compact
Hausdorff space $Y$ and a field $K$, we denote by $C_{c,K}(Y)$ the
algebra of continuous functions with compact support from $Y$ to
$K$, where $K$ is given the discrete topology. If $\theta$ is a
global action of $G$ on $Y$ we get as in Lemma \ref{lem:mapfromStoV}
a monoid homomorphism $\Phi \colon S(Y,G,\mathbb K (Y))\to
\mon{C_{c,K}(Y)\rtimes _{\theta^*} G}$.

\begin{theorem}
\label{thm:typesemigroupforOmegaECglobalized} Let $(E,C)$ be a
finite bipartite separated graph such that $s(E^1)=E^{0,1}$ and
$r(E^1)=E^{0,0}$, and let $(\Omega (E,C),\alpha^*)$ be the universal
$(E,C)$-dynamical system. Let $\Omega^{\mathfrak G}(E,C)$ denote the
totally disconnected, locally compact Hausdorff space obtained from
the globalization of the action of $\mathbb F$ on $\Omega (E,C)$
(Corollary \ref{cor:globalizedactions}), and let $\beta^*$ be the
corresponding global action of $\mathbb F$ on it.

 Then the map
$$\Phi\colon S(\Omega^{\mathfrak G} (E,C), \mathbb F, \mathbb K (\Omega^{\mathfrak G} (E,C)))\longrightarrow
\mon{C_{c,K} (\Omega^{\mathfrak G}(E,C))\rtimes _\beta \mathbb F }
$$ is an isomorphism for any field $K$.
\end{theorem}

\begin{proof}
Set $X= \Omega (E,C)$ and $Y=\Omega ^{\mathfrak G}(E,C)$. Let
$e=1_X\in C_{c,K}(Y)$. Then it is easily seen that $e$ is a full
projection in $C_{c,K}(Y)\rtimes _{\beta} \mathbb F$ such that
$$e(C_{c,K}(Y)\rtimes _{\beta} \mathbb F)e\cong C_K(X)\rtimes _{\alpha} \mathbb F.$$ It follows
from the first paragraph of the proof of \cite[Lemma 7.3]{AF} that
the map
$$\mon{\iota}\colon \mon{C_K(X)\rtimes_{\alpha}
\mathbb F} \to \mon{C_{c,K}(Y)\rtimes _{\beta} \mathbb F}$$ induced
by the inclusion $\iota \colon C_K(X)\rtimes_{\alpha} \mathbb F \to
C_{c,K}(Y)\rtimes _{\beta} \mathbb F$ is an isomorphism. (See also
\cite[Corollary 5.6]{GoodLPA} for a more general statement.)

We have a commutative diagram:
$$\begin{CD} S(X,\mathbb F, \mathbb K (X))  @>{\Phi_X}>>
\mon{C_K(X)\rtimes
_{\alpha} \mathbb F}\\
@V{\cong}VV   @V{\mon{\iota}}V{\cong}V  \\
S(Y,\mathbb F, \mathbb K (Y))  @>{\Phi_Y}>> \mon{C_{c,K}(Y)\rtimes
_{\beta} \mathbb F}
\end{CD}$$
We observed before that $\mon{\iota}$ is an isomorphism. The map
$\Phi_X$ is an isomorphism by Theorem
\ref{thm:typesemigroupforOmegaEC}, and the left vertical map is an
isomorphism by Proposition \ref{prop:BijectivityOfTypeMap{1.2}}.
Therefore we conclude that $\Phi _Y$ is also an isomorphism.
\end{proof}

We finally can reach our goal of presenting global actions on
zero-dimensional metrizable compact spaces with arbitrary failure of
cancellation laws.

\begin{theorem}
\label{thm:main-typesemigroup2} Let $M$ be a finitely generated
abelian conical monoid. Then there exist a zero-dimensional
metrizable compact space $Z$ and a global action of a finitely
generated free group $\mathbb F$ on $Z$ such that there is an order
embedding $\iota \colon M \to S(Z,\mathbb F, \mathbb K (Z))$.
\end{theorem}

\begin{proof}
Let $(E,C)$ be the finite bipartite separated graph considered in
the proof of Theorem \ref{thm:main-typesemigroup}, and let $\mathbb
F$ be the free group on $E^1$. Let $\Omega^{\mathfrak G,
\star}(E,C)=\Omega ^{\mathfrak G}(E,C)\cup \{\star \}$ be the
one-point compactification of the locally compact Hausdorff space
$\Omega ^{\mathfrak G}(E,C)$, with the action  $\beta^*$ extended to
$\Omega^{\mathfrak G, \star}(E,C)$ by $\beta^*_t(\star) =\star$ for
all $t\in \mathbb F$. Set $X=\Omega(E,C)$, $Y=\Omega ^{\mathfrak
G}(E,C)$ and $Z=\Omega^{\mathfrak G, \star}(E,C)$.

By the proof of Theorem \ref{thm:main-typesemigroup}, there is an
order-embedding $M\hookrightarrow S(X, \mathbb F, \mathbb K (X))$,
and,  by Corollary \ref{cor:globalizedactions}, $Y$ is a totally
disconnected locally compact Hausdorff space such that the natural
map
$$S(X,\mathbb F, \mathbb K (X))\to S(Y,\mathbb F, \mathbb K (Y)) $$
is an isomorphism. It is easily seen that the natural map
$S(Y,\mathbb F, \mathbb K (Y))\to S(Z,\mathbb F,\mathbb K (Z))$ is
an order-embedding. Combining all these facts we get that $Z$ is a
zero-dimensional metrizable compact space and there is an
order-embedding $M\hookrightarrow S(Z,\mathbb F, \mathbb K (Z))$.
\end{proof}

\begin{corollary}
\label{cor:non-paradoxical2} There exist a global action of a
finitely generated free group $\mathbb F$ on a zero-dimensional
metrizable compact space $Z$, and a non-$\mathbb F$-paradoxical
(with respect to $\mathbb K$) clopen subset  $A$ of $Z$ such that
$\mu (A)=\infty $ for every finitely additive $\mathbb F$-invariant
measure $\mu \colon \mathbb K \to [0,\infty ]$ such that $\mu
(A)>0$.
\end{corollary}

\begin{proof}
Consider the space $X=\Omega (E,C)$  and the clopen subset $A$
appearing in the proof of Corollary \ref{cor:non-paradoxical}, and
set $Z=\Omega^{\mathfrak G, \star}(E,C)$ as in the proof of Theorem
\ref{thm:main-typesemigroup}. Then $A$ is a clopen subset of $Z$
satisfying the desired conditions.
\end{proof}

\section{Descriptions of the space $\Omega (E,C)$. }

We follow with the general hypothesis of Sections
\ref{sect:bipsepgraphs}}--\ref{sect:dyn-systems}, so that $(E,C)$ is
a finite bipartite separated graph with $s(E^1)= E^{0,1}$ and
$r(E^1)= E^{0,0}$. In this section we will provide descriptions of
the zero-dimensional metrizable compact space $\Omega (E,C)$ and the
partial action of the free group on $E^1$, $\mathbb F$, on the space
$\Omega (E,C)$. These descriptions generalize the descriptions
obtained in \cite{AEK} for the universal $(m,n)$-dynamical system.

\smallskip

We will use the notation established in Definition
\ref{def:univtopspaces}, so that $\Omega (E,C)=\bigsqcup _{v\in
E^0}\Omega(E,C)_v$, $\Omega (E,C)_v=\bigsqcup _{x\in X} H_x$ for all
$X\in C_v$ $\, \, $($v\in E^{0,0}$), and $\, \theta _x\colon
\Omega(E,C)_{s(x)}\to H_x \, $ are the structural clopen sets and
homeomorphisms of the universal $(E,C)$-dynamical system.

\smallskip

The main point to understand the structure of the space $\Omega
(E,C)$ is the following: Points in $\Omega (E,C)$ are completely
determined by their dynamics. Namely, suppose that we have a point
$\xi$ in $\Omega(E,C)_w$, where $w\in E^{0,1}$.  Then we can apply
the maps $\theta _x$ to $\xi$, for all $x\in E^1$ such that
$s(x)=w$. Given one such arrow $x$, the point $\theta_x (\xi )$
belongs to $H_x\subseteq \Omega(E,C)_{r(x)}$. For each $X\in
C_{r(x)}$ with $x\notin X$, we have a decomposition
$\Omega(E,C)_{r(x)} =\bigsqcup _{y\in X} H_y$, and the point
$\theta_x (\xi)$ belongs to exactly one of these sets $H_y$, say
$\theta _x(\xi)\in H_{y(\xi, x)}$ for a unique $y(\xi, x)\in X$.
Then we can apply the map $\theta _{y(\xi,x)}^{-1}$ to $\theta_x
(\xi)$ and obtain a new point in $\Omega(E,C)_{s(y(\xi ,x))}$, and
again we can re-start the process. The point $\xi$ is determined by
the information provided by this infinite process, and conversely
any coherent set of data determines a point in the space $\Omega
(E,C)$.

It is easy to see that Proposition \ref{prop:univ-partialreps}
generalizes to ${\mathcal O}(E,C)$, meaning that the latter is the
universal C*-algebra for semi-saturated partial representations of
$\mathbb F$ satisfying relations (PI1)--(PI4).

We may thus apply \cite[Theorem 4.4]{ELQ} to obtain a
characterization of $\Omega(E,C)$, the spectrum of relations
(PI1)--(PI4), as a subset of $2^{\mathbb F}$.  The first step is to
translate our relations in terms of functions on $2^{\mathbb F}$.
Relative to (PI1) we have the functions:
  $$
  f^1_{x, y}(\xi) = [x\in\xi][y\in\xi] - \delta_{x,y}[x\in\xi],
  \quad\forall \xi\in 2^{\mathbb F},
  \leqno{{\rm (F1)}}
  $$
  for $x,y\in X$, $X\in C$, where the brackets correspond to boolean
value.  Speaking of (PI2), we have functions of the form:
  $$
  f^2_{x, y}(\xi) = [x^{-1}\in\xi]-[y^{-1}\in\xi],
  \leqno{{\rm (F2)}}
  $$
  for all $x,y\in E^1$ such that $s(x)=s(y)$.
As for (PI3), the functions we consider are:
  $$
  f^3_{X, Y}(\xi) =   \sum_{x\in X} [x\in\xi] - \sum_{y\in Y} [y\in\xi],
  \leqno{{\rm (F3)}}
  $$
  for all $X, Y \in C_v$,  $v\in E^{0,0}$.
In order to account for (PI4), we must consider a single  function,
namely:
  $$
  f^4(\xi) =  -1 +\sum_{w\in E^{0,1}} [x_w^{-1}\in \xi] +
  \sum_{v\in E^{0,0}}\ \sum_{x\in X_v} [x\in \xi],
  \leqno{{\rm (F4)}}
  $$
  where we choose some $x_w$ in $s^{-1}(w)$, for each $w\in E^{0,1}$, and
some $X_v$ in $C_v$, for each $v\in E^{0,0}$.

Finally,  since ${\mathcal O}(E,C)$ is the universal C*-algebra for
a collection of  partial representations which are
\emph{semi-saturated}, we must include the functions
  $$
  f^5_{g, h}(\xi) =
  [h\in \xi] [g\in \xi] - [gh\in \xi],
  $$
  for all $g,h\in{\mathbb F}$, such that $|gh|=|g|+|h|$.  See
  \cite[Proposition 5.4]{ExelAmena}
  for more details.

Considering the set ${\mathcal R}$ formed by all of the above
functions, we have by \cite[Proposition 4.1]{ELQ}, that
  $$
  \Omega(E,C) = \Omega_{\mathcal R} =
  \{\xi\in 2^{\mathbb F}: 1\in\xi, \  f(g^{-1}\xi) = 0, \hbox{ for all } f\in {\mathcal R}, \hbox{ and all } g\in\xi\}.
  $$

  In order to give a more explicit description for this space, let us
first introduce some terminology.  Given $\xi$ in $2^{\mathbb F}$,
and given $g$ in $\xi$, we will let $\xi_g$ be the \emph{local
configuration of $\xi$ at $g$}, meaning the set of all elements
  $$
  x\in E^1 \cup (E^1)^{-1}
  $$
  (recall that ${\mathbb F}$ is the free group generated by $E^1$),
  such that $gx\in \xi$.   Thus the set $g\xi_g$ consists precisely of the
elements in $\xi$ whose distance to $g$ equals 1 relative to the
Cayley graph metric.

One may then check that $\Omega(E,C)$ consists precisely of all
$\xi\in 2^{\mathbb F}$ such that:
  \begin{enumerate}
  \item[(a)] $1\in\xi$,
 \item[(b)] $\xi$ is convex\footnote{A subset $\xi\subseteq\mathbb{F}$ is said to be convex if, whenever $g,h,k\in\mathbb{F}$ are such that $g,h\in\xi$ and
$|g^{-1} h| = |g^{-1} k| +  |k^{-1} h|$, then $k\in\xi$.},
  \cite[Definition 4.4]{ExelLaca} and  
  \cite[Proposition 4.5]{ExelLaca},   
\item[(c)] for every $g\in\xi$, there is some $v\in E^0$ such that:
    \begin{enumerate}
    \item[(c.1)] in case $v\in E^{0,1}$, the local
configuration $\xi_g$ consists of the elements of the form $e^{-1}$,
for all $e\in E^1$ such that $s(e) = v$ (and nothing more),
    \item[(c.2)] in case $v\in E^{0,0}$, the local
configuration $\xi_g$ consists of exactly one element of each group
$X\in C_v$ (and nothing more).
\end{enumerate}
\end{enumerate}

\medskip

Having completed the description of $\Omega (E,C)$, the partial
action of $\mathbb F$ is now easy to describe: for each $g\in
\mathbb F$ we put
  $$
  \Omega_g = \big\{\xi\in \Omega (E,C) : g\in\xi\big\},
  $$
  and we let
  $$
  \theta_g: \Omega _{g^{-1}} \to \Omega_g,
  $$
  be given by $\theta_g(\xi) = g\xi = \{gh: h\in\xi\}$.

\medskip

We now start the description of the space $\Omega (E,C)$ in terms of
certain ``choice functions".

\medskip

To avoid cumbersome considerations, we will from now on assume the
following:

\begin{hypothesis}
\label{hypo:groupsbiggerthantwo}
 $(E,C)$ is a finite bipartite
separated graph such that $|C_v|\ge 2$ for all $v\in E^{0,0}$ and
$s(E^1) =E^{0,1}$.
\end{hypothesis}

This is justified by Proposition \ref{prop:reductiontoHypo}.

\medskip

Given $e\in E^1$, there is a unique $X\in C$ such that $e\in X$.
This unique $X$ will be denoted by $X_e$.

\medskip

Let $v\in E^{0,0}$. For $X\in C_v$, define
$$Z_X= \prod _{Y\in C_{v},\,  Y\ne X} Y.$$
For each $Y_0\in C_v$, $Y_0\ne X$, we denote by $\pi _{Y_0}$ the
canonical projection map $Z_X\to Y_0$.

\medskip

We now provide a description of the points belonging to
$\Omega(E,C)_w$ for $w\in E^{0,1}$.

\begin{definition}
\label{def:Efunctions}  A {\it partial $E$-function} for $(E,C)$ is
a finite sequence
$$(\Omega_1,f_1), (\Omega _2, f_2),\dots , (\Omega _r, f_r),$$
where each $\Omega _i$ is a certain finite subset of $\mathbb F$,
and each $f_i$ is a function from $\Omega _i$ to $\bigsqcup _{X\in
C} Z_X$ satisfying the following rules:

\begin{enumerate}
\item $\Omega _1= s^{-1}(w)$ for some $w\in E^{0,1}$ and $f_1(e) \in Z_{X_e}$ for all $e\in
s^{-1}(w)$.
\item For each $j=2, \dots , r$,  $\Omega _j$ is the set of all the elements
of $\mathbb F$ of the form
$$ e_{2j-1} e_{2j-2}^{-1} e_{2j-3}\cdots e_4^{-1}e_3e_2^{-1}e_1 $$
where $$e_{2j-3}\cdots e_4^{-1}e_3e_2^{-1}e_1 \in \Omega _{j-1},
$$ $e_{2j-2}= (\pi _Y\circ f_{j-1})(e_{2j-3}\cdots e_4^{-1}e_3e_2^{-1}e_1)$
for some $Y\in C$ which is a factor of $Z_X$, where
$f_{j-1}(e_{2j-3}\cdots e_4^{-1}e_3e_2^{-1}e_1)\in Z_X$, and where
$e_{2j-1}\in E^1$ satisfies that $s(e_{2j-1})=s(e_{2j-2})$ and
$e_{2j-1}\ne e_{2j-2}$.
\item The function $f_j\colon \Omega _j\to \bigsqcup _{X\in
C} Z_X$ satisfies
$$f_j (e_{2j-1} e_{2j-2}^{-1} e_{2j-3}\cdots e_4^{-1}e_3e_2^{-1}e_1)
\in Z_{X_{e_{2j-1}}}$$ for all $e_{2j-1} e_{2j-2}^{-1}
e_{2j-3}\cdots e_4^{-1}e_3e_2^{-1}e_1\in \Omega _j$. Note that, by
Hypothesis \ref{hypo:groupsbiggerthantwo}, there is always at least
one such function $f_j$.
\end{enumerate}
An $E$-function is an infinite sequence $(\Omega _1,f_1), (\Omega
_2, f_2),\dots $ satisfying the above conditions for all $j=1,2,3,
\dots $.
\end{definition}

Note that it might happen that for some $E$-function  $(\Omega
_1,f_1), (\Omega _2, f_2),\dots $ we have $\Omega _r=\emptyset$ for
some $r\ge 2$. In this case the $E$-function will be of the form
$$(\Omega_1, f_1),(\Omega_2,f_2),\dots , (\Omega _{r-1}, f_{r-1}),
(\emptyset,\emptyset), (\emptyset,\emptyset), \dots .  $$ This case
occurs for instance for the separated graph $(E,C)$ with
$E^{0,0}=\{v\}$, $E^{0,1}=\{ w_1, w_2\}$, and $C=C_v=\{ X, Y \}$
with $X=\{\alpha\}$, $Y=\{ \beta \}$, $s(\alpha )=s(\beta )= v$ and
$r(\alpha )= w_1$, $r(\beta ) = w_2$.

\medskip

Note that every partial $E$-function $(\Omega _1,f_1),\dots ,(\Omega
_r,f_r)$ can be extended to an $E$-function $(\Omega _1,f_1),\dots ,
(\Omega_r , f_r), (\Omega _{r+1}, f_{r+1}),\dots $, although
possibly $\Omega _i =\emptyset $ for all sufficiently large $i$.
Another degenerate case is the case where there is only one
extension of $(\Omega _1,f_1),\dots , (\Omega _r,f_r)$ to an
$E$-function  $(\Omega _1,f_1),\dots , (\Omega_r , f_r), (\Omega
_{r+1}, f_{r+1}),\dots $, but $\Omega _i\ne \emptyset $ for all $i$.
Note that this happens for instance for the unique partial
$E$-function $(\Omega_1,f_1)$ of the separated graph
$(E(1,1),C(1,1))$.

\begin{theorem}
\label{thm:Efunctions} Let $\Omega (E,C)$ be the universal space
associated to the separated graph $(E,C)$, and let $w$ be a vertex
in $E^{0,1}$. Then the points of $\Omega(E,C)_w$ are in one-to-one
correspondence with the $E$-functions $(\Omega _1,f_1),(\Omega _2,
f_2),\dots $ such that $\Omega _1= s^{-1}(w)$. Moreover, given a
point $\xi$ in $\Omega(E,C)_w$, represented by the $E$-function
$(\Omega _1,f_1),(\Omega _2, f_2),\dots $, the partial $E$-functions
$(\Omega _1,f_1),(\Omega_2,f_2),\dots ,  (\Omega _r,f_r)$, $r\ge 1$
represent a fundamental system of clopen subsets of $\xi$ in $\Omega
(E,C)$. The clopen set represented by the partial $E$-function
$(\Omega _1,f_1),(\Omega_2,f_2),\dots , (\Omega _r,f_r)$ corresponds
to a vertex in $F_{\infty}$ lying in the $2r+1$ level $D_{2r+1}$ of
$F_{\infty}^0= \bigsqcup_{i=0}^{\infty} D_i$.
\end{theorem}

\begin{proof}
One can show, using arguments similar to those in \cite[proof of
Theorem 4.1]{AEK} that the set of $E$-functions
$(\Omega_1,f_1),(\Omega _2,f_2), \dots ,  $ such that $\Omega _1=
s^{-1} (w)$ for $w\in E^{0,1}$, is in bijective correspondence with
the set of configurations $\xi \in \Omega _{\mathcal R}$ such that
the local configuration $\xi_1$ of $\xi $ at $1$ consists of
elements of the form $e^{-1}$ for all $e\in s^{-1}(w)$ (form (c.1)).

\medskip

We now proceed to show the statement about the relation between
partial $E$-functions and vertices in the graph $F_{\infty}$.
First, we show that the partial $E$-functions $(\Omega_1, f_1)$ such
that $\Omega _1 = s^{-1}(w)$ (for $w\in D_1$), correspond to
vertices $v$ in $D_3$ such that $r_3(v)=w$, where $r_n\colon D_n\to
D_{n-2}$ is the map defined just before Lemma \ref{lem:firststep}.
Suppose that $\Omega _1=s^{-1}(w)= \{x_1,\dots ,x_r \}$. For
$i=1,\dots ,r$, consider
\begin{equation}
\label{eq:f1andOmega1}
y_i:= \alpha^{x_i} (f_1(x_i)) \in E^1_1 .
\end{equation}
This has sense because $f_1(x_i)\in \prod_{Y\in C_{r(x_i)}, \, Y\ne
X_{x_i} } Y$. Note that there are exactly $r$ distinct elements in
$C_w$, namely $X(x_1),\dots , X(x_r)$, and that $y_i\in X(x_i)$ for
all $i$. We can therefore build the elements
$$z_i = \alpha ^{y_i} (y_1,\dots \widehat{y}_i, \dots ,y_r)\in
E^1_2 $$ for $i=1,\dots ,r$, and the vertex $v:=v(y_1,\dots ,y_r)\in
D_3$. Note that $s^{-1}(v)=\{z_1,\dots , z_r \}$, so there is a
canonical bijection $x_i\leftrightarrow z_i$ between $s^{-1}(w)$ and
$s^{-1}(v)$. The vertex $v$ just constructed is the vertex
associated to the partial $E$-function $(\Omega _1,f_1)$.
Conversely, the vertex $v=v(y_1,\dots ,y_r)$ determines $y_1,\dots ,
y_r$, which in turn determine the function $f_1$ by using
(\ref{eq:f1andOmega1}).

We will show that a partial $E$-function $$ (\Omega _1,f_1),\dots ,
(\Omega _r,f_r),$$ with $r\ge 1$, for $(E,C)$, such that $\Omega
_1=s^{-1}(w)$, corresponds to a partial $E$-function
$$(\Omega _1',f_1'), \dots , (\Omega _{r-1}',f_{r-1}')$$
for the separated graph $(E_2,C^2)$, such that $\Omega
_1'=s^{-1}(v)$, where $v$ is the vertex in $D_3=E_2^{0,1}$
corresponding to $(\Omega _1,f_1)$ (under the correspondence defined
above). By induction hypothesis, this determines (and is determined
by) a vertex $v'$ in $(F'_{\infty})^{0, 2r-1}$, where $F'_{\infty}$
is the $F$-graph associated to $(E_2,C^2)$ (see Construction
\ref{cons:complete-multiresolution}). Since $F'_{\infty}$ is the
separated graph obtained from $F_{\infty}$ by eliminating its two
first levels, we see that the original partial $E$-function
corresponds to the vertex $v'$ in $F_{\infty}^{0,2r+1}=D_{2r+1}$.

So we only need to prove the desired correspondence between partial
$E$-functions for $(E,C)$ and $(E_2,C^2)$.  Observe that, given a
vertex $v'$ in $D_3$ such that $r_3(v')= v\in D_1$, there is a
canonical bijective map between $s^{-1}(v)$ and $s^{-1}(v')$. We
will denote this bijection by $e\leftrightarrow e'$, for $e\in
s^{-1}(v)$. Now take $e$ in $s^{-1}(v)$. Then the corresponding
arrow $e'$ has an end vertex $r(e')$ and we are going to check that
$r_2(r(e'))= r(e)$.

We have $s(e') =v(y_1,\dots ,y_l)$ and $e'= \alpha ^{y_t}(
y_1,\dots, \widehat{y}_t,\dots , y_l)$ for some $t$, where
$s^{-1}(s(e))= \{ x_1,\dots ,x_l\}$ and $x_t = e$. Now set
$$C_{r(x_t)}= C_{r(e)}= \{ X_1,\dots ,X_s \}$$
and write $x_t=\widetilde{x}_p\in X_p$ for a unique $p\in \{1,\dots
, s \}$. Then there are $\widetilde{x}_i\in X_i$ for $i\ne p$ such
that
$$y_t = \alpha ^{\widetilde{x}_p} (\widetilde{x}_1,\dots ,
\widehat{\widetilde{x}_p},\dots ,\widetilde{x}_s ).$$ Note that
$e\in X_p$, so that $X_e=X_p$. On the other hand
$$r(e')= r(\alpha ^{y_t}(
y_1,\dots, \widehat{y}_t,\dots ,
y_l))=s(y_t)=v(\widetilde{x}_1,\dots ,\widetilde{x}_s),$$ so that
$r_2(r(e'))=r(e)$. It follows that there is a canonical bijective
correspondence between $C_{r(e)}$ and $C_{r(e')}$ (see Lemma
\ref{lem:firststep} and Remark \ref{rem:enumerations}), and we will
denote this correspondence by $X\leftrightarrow X'$ for $X\in
C_{r(e)}$. Note that $X_{e'}=X_e'$.

\medskip

Let $(\Omega _1,f_1),\dots , (\Omega _r,f_r)$ be a partial
$E$-function for $(E,C)$, where $r> 1$. Let $w\in E^{0,1}$ such that
$\Omega _1= s^{-1}(w)$, and let $w'$ be the vertex in $D_3$
corresponding to $(\Omega_1, f_1)$. We first define $(\Omega_1',
f_1')$. Let $\Omega _1'= \{ e'\mid e\in \Omega _1 \}$. Then $f_1'$
is determined as follows. Given $e'\in \Omega_1'$, there is a unique
$e\in \Omega_1=s^{-1}(w)$ corresponding to it through the bijection
between $s^{-1}(w')$ and $s^{-1}(w)$. Now we are going to define
$f_1'(e')\in Z_{X{e'}}$. Adopt the notation of the above paragraph,
so that $e'= \alpha ^{y_t}( y_1,\dots, \widehat{y}_t,\dots , y_l)$
for some $t$, where $s^{-1}(s(e))= \{ x_1,\dots ,x_l\}$ and $x_t =
e$. Now set $C_{r(x_t)}= C_{r(e)}= \{ X_1,\dots ,X_s \}$ and write
$e=\widetilde{x}_p\in X_p$ for a unique $p\in \{1,\dots , s \}$.
Then there are $\widetilde{x}_i\in X_i$ for $i\ne p$ such that
$$y_t = \alpha ^{\widetilde{x}_p} (\widetilde{x}_1,\dots ,
\widehat{\widetilde{x}_p},\dots ,\widetilde{x}_s ).$$ For $i\ne p$,
we have $X_i'\in C_{r(e')}$, and $X_i'\ne X_p'=X_e'= X_{e'}$, and we
aim to define $\pi _{X_i'}(f_1'(e'))$. First set
$$\widetilde{y}_i= \alpha ^{\widetilde{x}_i}(\widetilde{x}_1,\dots ,
\widehat{\widetilde{x}_i},\dots ,\widetilde{x}_s ).$$ Observe that
$r(\widetilde{y}_i) =s(\widetilde{x}_i)$. Moreover, note that it
follows from (\ref{eq:f1andOmega1}) that $\widetilde{x}_i = \pi
_{X_i} (f_1(e))$. The elements in $C_{r(\widetilde{y}_i)}$ are in
bijective correspondence with the elements in
$s^{-1}(s(\widetilde{x}_i))$. Write $\mathcal S =
s^{-1}(s(\widetilde{x}_i))\setminus \{ \widetilde{x}_i \}$. Observe
that, since $\widetilde{x}_i =\pi _{X_i}(f_1(e))$, the element
$x\widetilde{x}_i^{-1} e$ belongs to $\Omega _2$ for all $x\in
\mathcal S$. Define $\pi _{X_i'} (f_1'(e'))= \alpha
^{\widetilde{y}_i}( (\widetilde{y}_{i, x })_{ x\in \mathcal S})$,
where
$$\widetilde{y}_{i, x} = \alpha ^{x} (f_2 (x\widetilde{x}_i^{-1}e )).$$
Note that, for $X_i$ as above, with $i\ne p$, we have
$$r_3(s(\pi
_{X_i'}(f_1'(e'))))=s(\widetilde{x}_i)=s(\pi_{X_i}(f_1(e))) ,$$ so
that given any element $e_3'(\pi_{X_i'}(f_1'(e')))^{-1} e'$ in
$\Omega _2'$ we can form a corresponding element $e_3
(\pi_{X_i}(f_1(e)))^{-1} e $ in $\Omega _2$.

\medskip

Assume that for some $r> j> 1$, we have defined a partial
$E$-function $(\Omega _1',f_1'),\dots , (\Omega _{j-1}', f_{j-1}')$
on $(E_2,C^2)$ satisfying the following conditions:
\begin{enumerate}
\item For all $1\le j_0<j$ and all $
e_{2j_0+1}'(e_{2j_0}')^{-1}\cdots e_3'(e_2')^{-1}e_1'$ in
$\Omega_{j_0+1}'$ we have $r_3(s(e_{2j_0}'))=s(e_{2j_0})$, where
$e_{2j_0}$ is inductively defined by $e_{2j_0}=\pi _Y(
f_{j_0}(e_{2j_0-1}e_{2j_0-2}^{-1}\cdots e_3e_2^{-1}e_1))$ if
$e'_{2j_0}=\pi _{Y'}( f'_{j_0}(e_{2j_0-1}'(e_{2j_0-2}')^{-1}\cdots
e_3'(e_2')^{-1}e_1'))$ for some $Y\in C_{r(e_{2j_0-1})}\setminus \{
X_{e_{2j_0-1}}\}$.

\smallskip

\noindent Therefore we have a well-defined element
$e_{2j_0+1}(e_{2j_0})^{-1}\cdots e_3e_2^{-1}e_1$ in $\Omega
_{j_0+1}$ for each element $
 e_{2j_0+1}'(e_{2j_0}')^{-1}\cdots e_3'(e_2')^{-1}e_1' $ in $\Omega _{j_0+1}'$.

\smallskip

\item For all $1\le j_0<j$, all
$e_{2j_0-1}'(e_{2j_0-2}')^{-1}\cdots e_3'(e_2')^{-1}e_1'\in \Omega
_{j_0}'$, all $X\in C_{r(e_{2j_0-1})}\setminus \{ X_{e_{2j_0-1}}
\}$, and all $x\in s^{-1}\Big( s\Big (\pi_X f_{j_0}
(e_{2j_0-1}e_{2j_0-2}^{-1}\cdots e_3e_2^{-1}e_1)\Big)\Big)\setminus
\, \Big{\{} \pi_X  f_{j_0}(e_{2j_0-1}e_{2j_0-2}^{-1}\cdots
e_3e_2^{-1}e_1) \Big{\}}$, we have
\begin{align*}\pi & _x\Big( \pi _{X'}f_{j_0}'(e_{2j_0-1}'(e_{2j_0-2}')^{-1}\cdots
e_3'(e_2')^{-1}e_1')\Big)\\
&  = \alpha ^x\Big(f_{j_0+1}\Big(x(\pi_{X}f_{j_0} (e_{2j_0-1}\cdots
e_2^{-1}e_1 ))^{-1} e_{2j_0-1}\cdots e_3e_2^{-1}e_1 \Big)\Big) ,
\end{align*}
where $\pi_x$ denotes the projection $\pi_{X(x)}$, being $X(x)$ the
element in $C_{s(x)}$ corresponding to $x$.
\end{enumerate}

Note that indeed condition (2) implies condition (1), but we found
useful to state condition (1). As we have seen before, both
conditions hold when $j_0=1$.

To complete the induction step, we have to define the map $f_j'$.
This is done essentially as in the case $j=1$, only the notation is
a little bit worse.

Let $e_{2j-1}'(e_{2j-2}')^{-1} \cdots (e_2')^{-1} e_1'$ be an
element in $\Omega _j'$, and let $e_{2j-1}e_{2j-2}^{-1}\cdots
e_2^{-1}e_1$ be the corresponding element in $\Omega _j$ (see condition
(1)). We want to define $f_j'(e_{2j-1}'(e_{2j-2}')^{-1} \cdots
(e_2')^{-1} e_1')\in Z_{X_{e_{2j-1}}'}$.

We have $s(e_{2j-1}') =v(y_1,\dots ,y_l)$ and $e_{2j-1}'= \alpha
^{y_t}( y_1,\dots, \widehat{y}_t,\dots , y_l)$ for some $t$, where
$s^{-1}(s(e_{2j-1}))= \{ x_1,\dots ,x_l\}$ and $x_t = e_{2j-1}$. Now
set
$$C_{r(x_t)}= C_{r(e_{2j-1})}= \{ X_1,\dots ,X_k \}$$
and write $x_t=\widetilde{x}_p\in X_p$ for a unique $p\in \{1,\dots
, k \}$. Then there are $\widetilde{x}_i\in X_i$ for $i\ne p$ such
that
$$y_t = \alpha ^{\widetilde{x}_p} (\widetilde{x}_1,\dots ,
\widehat{\widetilde{x}_p},\dots ,\widetilde{x}_k).$$ We have to
check that $\widetilde{x}_i=\pi_{X_i}(f_j ( e_{2j-1}\cdots
e_2^{-1}e_1))$ for $i\ne p$. For this we need condition (2). Indeed,
observe that there is $s\ne t$ such that
$$e_{2j-2}' = \alpha^{y_s}(y_1,\dots ,\widehat{y_s}, \dots , y_l)$$
Observe that
$$e'_{2j-2}= \pi _{X_{e'_{2j-2}}}( f_{j-1}'(e'_{2j-3}\cdots (e'_2)^{-1}
e_1')).$$ On the other hand
$$\pi _{x_t}(e_{2j-2}') = \pi_{x_t}(\alpha ^{y_s}(y_1,\dots
,\widehat{y_s},\dots , y_l))= y_t,$$ so that, putting everything
together and using (2), we  get
\begin{align*}
y_t & =\pi_{x_t}\pi_{X_{e'_{2j-2}}}(f_{j-1}'(e'_{2j-3}\cdots (e'_2)^{-1} e_1'))\\
& = \alpha^ {x_t} (f_j( e_{2j-1}e_{2j-2}^{-1}e_{2j-3} \cdots
e_2^{-1} e_1 ))
\end{align*}
because $e_{2j-2}=\pi _{X_{e_{2j-2}}}(f_{j-1}( e_{2j-3}\cdots
e_2^{-1}e_1 ))$ and $e_{2j-1}=x_t$. This shows that, for $i\ne p$,
we have $\widetilde{x}_i=\pi _{X_i}(f_j(e_{2j-1}\cdots
e_2^{-1}e_1))$, as desired.

For $i\ne p$, we have $X_i'\in C_{r(e_{2j-1}')}$, and $X_i'\ne X_p'=
X_{e_{2j-1}'}$, and we aim to define $\pi
_{X_i'}(f_j'(e'_{2j-1}\cdots  (e_2')^{-1} e_1'))$. First set
$$\widetilde{y}_i= \alpha ^{\widetilde{x}_i}(\widetilde{x}_1,\dots ,
\widehat{\widetilde{x}_i},\dots ,\widetilde{x}_k).$$ Observe that
$r(\widetilde{y}_i) =s(\widetilde{x}_i)$.

The elements in $C_{r(\widetilde{y}_i)}$ are in bijective
correspondence with the elements in $s^{-1}(s(\widetilde{x}_i))$.
Write $\mathcal S = s^{-1}(s(\widetilde{x}_i))\setminus \{
\widetilde{x}_i \}$. Observe that, since
$\widetilde{x}_i=\pi_{X_i}(f_j (e_{2j-1}\cdots e_2^{-1} e_1))$, the
element $x\widetilde{x}_i^{-1} e_{2j-1}\cdots e_2^{-1} e_1$ belongs
to $\Omega _{j+1}$ for all $x\in \mathcal S$. Define
$$\pi _{X_i'}
(f_j'( e_{2j-1}'\cdots (e_2')^{-1}e_1'))= \alpha ^{\widetilde{y}_i}(
(\widetilde{y}_{i, x })_{ x\in \mathcal S}),$$ where
$$\widetilde{y}_{i, x} = \alpha ^{x} (f_{j+1} (x \widetilde{x}_i^{-1}e_{2j-1}
\cdots e_2^{-1} e_1 )).$$ Note that, for $X_i$ as above, with $i\ne
p$, we have
$$r_3(s(\pi
_{X_i'}(f_j'(e_{2j-1}' \cdots
(e_2')^{-1}e_1')))=s(\widetilde{x}_i)=s(\pi_{X_i}(f_j(e_{2j-1}\cdots
e_2^{-1}e_1))) ,$$ so that given any element $
e_{2j+1}'(e_{2j}')^{-1}\cdots  (e_2')^{-1}e_1'$ in $\Omega _{j+1}'$
we can form a corresponding element $e_{2j+1}e_{2j}^{-1}\cdots
e_2^{-1}e_1$ in $\Omega _{j+1}$. This shows (1) for $j$, while (2)
for $j$ follows from the construction.

\medskip

It follows from (2) that the partial $E$-function
$(\Omega_1,f_1),\dots ,(\Omega _r,f_r)$ is completely determined by
$(\Omega _1',f_1'),\dots , (\Omega _{r-1}',f_{r-1}')$. (Notice that
knowledge of the latter provides knowledge of the initial vertex
$w'$ in $D_3= E_2^{0,1}$.)

\medskip

Now observe that the points in $F_{\infty}$ are in bijective
correspondence with a basis of clopen subsets for the topology of
$\Omega (E,C)$. Given an $E$-function $(\Omega_1,f_1),
(\Omega_2,f_2),\dots $ we have a descending chain of clopen subsets
$U_1\supseteq U_2 \supseteq \cdots $, where $U_r$ is the clopen set
corresponding to $(\Omega_1,f_1),\dots , (\Omega _r,f_r)$. The
intersection $\bigcap _{i=1}^{\infty} U_i$ will be non-empty, and
using that $\Omega (E,C)$ is Hausdorff, it is easy to check that it
consists of a single point $\xi $ in $\Omega (E,C)$. This completes
the proof of the theorem.
\end{proof}

\begin{remarks}
 \label{rem:kinfofstems}
(1) In certain applications of [26, Theorem 4.4], notably to the case of the free group, the elements $\xi$ in the space $\Omega_R$
may be more or less completely described by means of their intersections with the positive cone in the free group.  Moreover, this
intersection often consists of the prefixes of a certain infinite word (in the positive generators), called the ``stem''
of $\xi$, which therefore completely determines the configuration $\xi$.  However, due to the zig-zag nature of the configurations
present in $\Omega(E,C)$, no such easy description is available.  Our use of partial $E$-functions is therefore an attempt at
giving a concrete description of these highly intricated configurations. 

\smallskip

\noindent (2) We now give the intuitive idea behind the concept of an $E$-function. Let $\xi\in \Omega(E,C)_w$, where $w\in E^{0,1}$.
 Consider the picture of $\xi$ as a subset of vertices in the Cayley graph $\Gamma$ of the free group $\mathbb F$
 generated by $E^1$. The intersection $\xi \cap B_1$ of $\xi$ with the ball of radius $1$ (centered at $1$) does not involve any choice, since 
 it is given by $1$ and all vertices $e^{-1}$, where $e\in s^{-1}(w)$ (because the local configuration $\xi _1$ of $\xi$ at $1$ 
 is of type (c.1)). 
 The intersection $\xi\cap B_2$ with the ball of radius $2$
 involves the choices which are codified by the function $f_1\colon s^{-1}(w) \to \bigsqcup_{X\in C} Z_X$, because, for $e\in s^{-1}(w)$, 
 the local configuration $\xi_{e^{-1}}$ of 
$\xi$ at $e^{-1}$ must be of type (c.2).
 In particular the vertices at distance $2$ from the origin are the elements of the form $e^{-1}\pi_X(f_1(e))$, for all
 $e\in s^{-1}(w)$ and all $X\in C_{r(e)}$ with $X\ne X_e$.
 In general, the vertices which are at distance $2r-1$ of the origin are completely determined by the vertices which are
 at distance $2r-2$, and the choice function $f_r$ determines the vertices which are at distance $2r$ from the vertices at distance $2r-1$.
 In conclusion knowing the partial $E$-function $(\Omega_1,f_1), \dots , (\Omega_r,f_r)$ of $\xi$ is equivalent to knowing the intersection of $\xi$ with the ball
 $B_{2r}$. Of course, this can be identified with the set of paths in the Cayley graph connecting $1$ to a vertex in $\xi$ at distance $2r$ from $1$. 
 Obviously, the element $\xi$ is determined by the intersections $\xi \cap B_{2r}$ for all $r\ge 1$. 
 
 \smallskip
 
 \noindent (3)
 The hardest part of the proof of Theorem \ref{thm:Efunctions} consists in showing the deep fact that
 knowledge of the intersection $\xi \cap B_{2r}$ is equivalent to knowledge of a vertex in the $2r+1$ layer $D_{2r+1}$ of the
 graph $F_{\infty}$ associated to $(E,C)$ in Construction \ref{cons:complete-multiresolution}.
 \qed
 \end{remarks}

The action of $\mathbb  F$ on (partial) $E$-functions is described  in our next result.

\begin{lemma}
 \label{lem:actionForFuncts}
  Let $$g= z_r^{-1}x_rz_{r-1}^{-1} x_{r-1}\cdots z_1^{-1}x_1 $$ be a reduced
word in $\mathbb F$, where $x_1,\dots ,x_r,z_1,\dots ,z_r\in E^1$.
Then $\text{Dom}(\theta _g)=\emptyset $ unless $z_i\in Y_i$ for some
$Y_i \in C_{r(x_i)}$ with $Y_i\ne X_{x_i}$ for all $i=1,\dots ,r$,
and $s(x_{i+1})=s(z_i)$ for $i=1,\dots ,r-1$. Assume that the latter
conditions hold. Then the domain of $\theta_g$ is precisely the set
of all $E$-functions $\mathfrak f=((\Omega_1,f_1),(\Omega_2,f_2),
\dots )$ such that $\Omega _1 = s^{-1}(s(x_1))$ and
$$\pi _{Y_i}f_i(x_i \cdots z_2^{-1}x_2z_1^{-1}x_1)=z_i $$
for all $ i=1,\dots ,r$, and the range of $\theta_g$ is the set of those
$E$-functions $\mathfrak f'=(f_1',f_2',\dots , )$ such that
$\pi_{X_{x_{r-i+1}}}f_i'(z_{r-i+1} \cdots z_{r-1}x_r^{-1}z_{r})=
x_{r-i+1}$ for all $i=1,\dots ,r$. Moreover for $\mathfrak f \in
\text{Dom}(\theta_g)$ let $\mathfrak f'= {^{g}\mathfrak f}$ denote
the image of $\mathfrak f$ under the action of $g$. Then $\mathfrak
f'=((\Omega _1',f_1'),(\Omega _2', f_2'),\dots )$ with
  $$
  f_{r+t}'(y_{2t-1} \cdots y_2^{-1}y_1 x_1^{-1}z_1 \cdots z_{r-1}x_r^{-1}z_{r})=
f_t(y_{2t-1} \cdots y_2^{-1}y_1) $$ if $y_{2t-1} \cdots y_2^{-1}y_1
\in \Omega _{t}$ and $y_1\ne x_1$.

  Moreover, for $i=2,3,\dots , r$ and $y_{2t-1} \cdots y_2^{-1}
  y_1z_{i-1}^{-1}x_{i-1}\cdots z_1^{-1}x_1
\in \Omega _{t+i-1}$ with $y_1\ne x_i$,
  $$
  f_{r-i+1+t}'(y_{2t-1} \cdots y_2^{-1}
  y_1x_i^{-1}z_i \cdots z_{r-1}x_r^{-1}z_r)= f_{i-1+t}(y_{2t-1} \cdots y_2^{-1}
  y_1z_{i-1}^{-1}x_{i-1}\cdots z_1^{-1}x_1),$$
and for $y_{2t-1}\cdots y_2^{-1}y_1z_r^{-1}x_r\cdots
z_2^{-1}x_2z_1^{-1}x_1\in \Omega _{t+r}$ we have
  $$
  f_t'(y_{2t-1}\cdots y_2^{-1} y_1)= f_{r+t}(y_{2t-1} \cdots
  y_2^{-1}y_1z_r^{-1}x_r\cdots z_2^{-1}x_2z_1^{-1}x_1)\, .$$
\end{lemma}

\begin{proof}
The result is clear once we realize that the action of $g$ on a
point $\xi$ in $\Omega(E,C)_{s(x_1)}$ consists in moving $\xi$ using
the composition $\theta_{z_r}^{-1}\theta_{x_r}\cdots
\theta_{z_1}^{-1}\theta_{x_1}$, in case this is possible.
\end{proof}

Notice that, for $g$ as in Lemma \ref{lem:actionForFuncts}, the condition that $\text{Dom}(\theta_g)\ne \emptyset$
is precisely the condition that the path $z_r^*x_r\cdots z_1^*x_1$ is an admissible path in the double $\hat{E}$
of $E$ (see Section \ref{sect:topfree} below for the definition of admissible path). For such an element 
$g\in \mathbb F$ and $\xi \in \text{Dom}(\theta_g)$, we have that the configuration 
$\theta_g(\xi)$ is just a translation of the configuration $\xi$.

\section{Examples}
\label{sect:examples}

Our first example is of a generic type. At first sight, it looks
restrictive to concentrate attention to bipartite separated graphs.
However, the following result shows that (modulo Morita equivalence)
{\it all} graph C*-algebras (or Leavitt path algebras) can be
studied that way. For a general separated graph $(E,C)$ we define
$\Lab (E,C)$ (respectively $\mathcal O (E,C)$) as the quotient of
$L(E,C)$ (respectively $C^*(E,C)$) by the ideal (resp. closed ideal)
generated by the commutators $[e(u), e(u')]$, for all $u,u'\in U$,
where $U$ is the multiplicative semigroup generated by $E^1\sqcup
(E^1)^*$,  and $e(u)=uu^*$.

\begin{proposition}
\label{prop:reductiontobipartitegraphs} Let $(E,C)$ be a separated
graph. Then there exists a bipartite separated graph
$(\widetilde{E},\widetilde{C})$ such that
$$L_K(\widetilde{E},\widetilde{C}) \cong  M_2(L_K(E,C)),\qquad
C^*(\widetilde{E},\widetilde{C}) \cong M_2(C^*(E,C)).$$  Moreover we
have
$$\Lab _K(\widetilde{E}, \widetilde{C}) \cong M_2(\Lab _K
(E,C)),\qquad \mathcal O (\widetilde{E}, \widetilde{C})\cong
M_2(\mathcal O (E,C)) .$$
\end{proposition}

\begin{proof}
Let $V_0$ and $V_1$ be two disjoint copies of $E^0$, and denote the
canonical maps $E^0\to V_i$ by $v\mapsto v_i$ for $i=0,1$.  Write
$\widetilde{E}^{0,0}= V_0$ and $\widetilde{E}^{0,1}=  V_1$. Now
$\widetilde{E}^1$ will be the disjoint union of a copy of $E^0$ and
a copy of $E^1$: $$\widetilde{E}^1= \{h_v\mid v\in E^0 \} \bigsqcup
\{e_0\mid e\in E^1 \},$$ with
$$
\tilde{r}(h_v)=v_0, \quad \tilde{s}(h_v)=v_1, \quad \tilde{r}(e_0)=
r(e)_0, \quad \tilde{s}(e_0)= s(e)_1, \qquad (v\in E^0, e\in E^1 )
.$$ Denote by $e_{ij}$, $0\le i,j\le 1$ the standard matrix units of
$M_2(K)$.  We define maps $\varphi \colon L(\widetilde{E},
\widetilde{C})\to M_2(L(E,C))$ and $\psi \colon M_2(L(E,C))\to
L(\widetilde{E},\widetilde{C})$ by the rules
$$\varphi (v_i)= v\otimes e_{ii},\quad \varphi (h_v)= v\otimes
e_{01},\quad \varphi (e_0)= e\otimes e_{01}, \qquad (v\in E^0, e\in
E^1, i=0,1) ,$$ and
$$\psi (v\otimes e_{ii}) =v_i,\qquad \psi (v\otimes e_{01})=h_v,
\qquad (v\in E^0, i=0,1) ,$$
\begin{align*}\psi (e\otimes e_{00}) & =
e_0h_{s(e)}^*,\qquad \, \psi (e\otimes e_{11})  = h ^*_{r(e)}e_0\\
 \psi
(e\otimes e_{01}) & =e_0,\qquad \qquad \psi (e\otimes e_{10})  =
h^*_{r(e)} e_0 h^*_{s(e)}, \qquad (e\in E^1) .
\end{align*}
It is straightforward (though somewhat tedious) to check that
$\varphi$ and $\psi$ give well-defined $*$-homomorphisms which are
mutually inverse. Let $U$ and $\widetilde{U}$ denote the
multiplicative semigroups of the algebras $L(E,C)$ and
$L(\widetilde{E}, \widetilde{C})$ generated by $E^1\sqcup (E^1)^*$
and $\widetilde{E}^1\sqcup (\widetilde{E}^1)^*$ respectively. Let
$J$ denote the ideal of $L(E,C)$ generated by all elements of the
form $[e(u), e(u')]$, with $u,u'\in U$, and let $\widetilde{J}$
denote the corresponding ideal of $L(\widetilde{E},\widetilde{C})$.
 Then
it is easy to check that, for $\tilde{u}\in \widetilde{U}$, $\varphi
(e(\tilde{u}))$ is a matrix consisting of an element of the form
$e(u)$, for some $u\in U$, in one of the two main diagonal positions
and $0$s in the rest of positions. This clearly implies that
$\varphi (\widetilde{J})\subseteq M_2(J)$. Since the image by $\psi
$ of an element of the form $u\otimes e_{00}$, for  $u\in U$,
belongs clearly to $\widetilde{U}$, we see that the reverse
containment $M_2(J)\subseteq \varphi (\widetilde{J})$ also holds.
Thus $\varphi (\widetilde{J})= M_2(J)$ and we conclude that $\varphi
$ defines an isomorphism from  $\Lab (\widetilde{E},\widetilde{C})$
onto $M_2(\Lab (E,C))$, as desired.

The same proof applies (thanks to the universal properties of the
involved C*-algebras) to show the corresponding result for the graph
C*-algebras.
\end{proof}

In particular, since any classical graph C*-algebra $C^*(E)$
satisfies that the final projections $[e(u),e(u')]=0$ for all
$u,u'\in U$, we see that $\mathcal O (E)= C^*(E)$ can always be
interpreted (through a very concrete Morita equivalence) as a graph
C*-algebra of a bipartite separated graph. A similar statement
applies to Leavitt path algebras. The description of a Leavitt path
algebra in terms of a crossed product by a partial action has been
recently achieved by Gon\c calves and Royer in \cite{GR}.

We are now going to show that we can always reduce our study of the
graph algebras of finite bipartite separated graphs, up to Morita
equivalence and finite-dimensional summands, to the study of finite
bipartite separated graphs satisfying Hypothesis
\ref{hypo:groupsbiggerthantwo}.

\begin{proposition}
\label{prop:reductiontoHypo} Let $(E,C)$ be a finite bipartite
separated graph. Then there exists a finite bipartite separated
graph $(E',C')$ satisfying Hypothesis \ref{hypo:groupsbiggerthantwo}
such that $L_K(E,C)$ is Morita equivalent to $L_K(E',C')\oplus K^n$
for some $n\ge 0$. Similar statements hold for $C^*(E,C)$, $\Lab
(E,C)$ and $\mathcal O (E,C)$.
\end{proposition}

\begin{proof}
Suppose that there is $v\in E^{0,0}$ such that $|C_v|=1$. Let
$(F,D)$ be the bipartite separated graph such that
$F^{0,0}=E^{0,0}\setminus \{v\}$, $F^{0,1}= E^{0,1}$, and
$F^1=E^1\setminus r^{-1}(v)$. Set $p=\sum _{w\in E^0\setminus
\{v\}}w$. Then $p$ is a full projection in $L_K(E,C)$ such that
$$L_K(F,D)\cong pL_K(E,C)p.$$
Therefore $L_K(E,C)$ is Morita equivalent to $L_K(F,D)$. We can
apply this process to suppress all the vertices $v'$ in $E^{0,0}$
such that $|C_{v'}|=1$, and thus we arrive at a bipartite separated
graph $(\widetilde{E}, \widetilde{C})$ such that $L_K(E,C)$ is
Morita equivalent to $L_K(\widetilde{E},\widetilde{C})$ and
$|C_{w}|\ne 1$ for all $w\in E^{0,0}$.

Now let $(E',C')$ be the bipartite separated graph obtained from
$(\widetilde{E}, \widetilde{C})$ by deleting all the vertices in
$\widetilde{E}^{0,0}\setminus r(\widetilde{E}^1)$ and all the
vertices in $\widetilde{E}^{0,1}\setminus s(\widetilde{E}^1)$. Then
clearly $(E',C')$ satisfies Hypothesis
\ref{hypo:groupsbiggerthantwo}, and
$L_K(\widetilde{E},\widetilde{C})\cong L_K(E',C')\oplus K^n$, where
$n=|E^0|- |r(\widetilde{E}^1)|- |s(\widetilde{E}^1)|$. This shows
the result.
\end{proof}

We start our description of concrete examples with a motivational
example for the entire theory of separated graphs, see \cite{AG},
\cite{AG2}, \cite{Aone-rel}, \cite{AEK}.

\begin{example}
\label{exam:m,ndyn-system}

For integers $1\le m\le n$, define the separated graph
$(E(m,n),C(m,n))$, where
\begin{enumerate}
\item $E(m,n)^0 := \{v,w\}$ (with $v\ne w$).
\item $E(m,n)^1 :=\{\alpha_1,\dots , \alpha_n,\beta_1,\dots ,\beta_m\}$ (with $n+m$ distinct edges).
\item $s(\alpha_i)=s(\beta _j) =v$ and $r(\alpha _i)=r(\beta _j)=w$
for all $i$, $j$.
\item $C(m,n)= C(m,n)_v := \{X,Y\}$, where $X:= \{\alpha_1,\dots ,\alpha_n\}$ and $Y:=  \{\beta _1,\dots, \beta _m \}$.
\end{enumerate}
See Figure \ref{fig:m,nsepargraph} for a picture in the case $m=2$,
$n=3$. By \cite[Proposition 2.12]{AG},
$$L(E(m,n),C(m,n))\cong
M_{n+1}(L(m,n)) \cong M_{m+1}(L(m,n)),$$ where $L(m,n)$ is the
classical Leavitt algebra of type $(m,n)$. The same argument (by way
of universal properties) shows that
\begin{equation*}  \label{isomatUnc}
C^*(E(m,n),C(m,n)) \cong M_{n+1}(U^{\text{nc}}_{m,n})\cong
M_{m+1}(U^{\text{nc}}_{m,n}) \,,
\end{equation*}
where $U^{\text{nc}}_{m,n}$ denotes the C*-algebra generated by the
entries of a universal unitary $m\times n$ matrix, as studied by
Brown and McClanahan in \cite{Brown, McCla1, McCla2, McCla3}.

The C*-algebras $\mathcal O _{m,n}$ and $\mathcal O _{m,n}^r$
studied in \cite{AEK} are precisely the C*-algebras $\mathcal O
(E(m,n), C(m,n))$ and $\mathcal O ^r( E(m,n), C(m,n))$ in the
notation of the present paper. (See Definition
\ref{def:reducedO(E,C)} for the definition of the reduced C*-algebra
$\mathcal O ^r(E,C)$.) It was shown in \cite{AEK} that these two
C*-algebras are not isomorphic.

The algebras $\Lab _{m,n}(K):= \Lab _K(E(m,n), C(m,n))$ have not
appeared before in the literature. They provide natural examples
(indeed universal examples) of actions on a compact Hausdorff space
supporting $(m,n)$-paradoxical decompositions (Corollary
\ref{cor:non-paradoxical}).
\end{example}

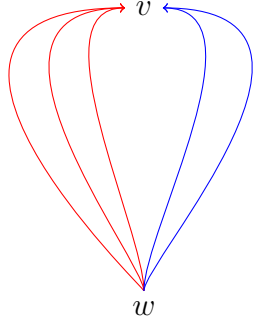
\begin{figure}[htb]
\begin{tikzpicture}[scale=4]
 \node (v) at (0,1)  {$v$};
 \node (w) at (0,0) {$w$};
 \draw[<-,red] (v.west) .. controls+(left:9mm) and +(up:.1mm) ..
 (w.north);
 \draw[<-,red] (v.west) .. controls+(left:6mm) and +(up:1mm) ..
 (w.north);
\draw[<-,red] (v.west) .. controls+(left:3mm) and +(up:2mm) ..
 (w.north);
\draw[<-,blue] (v.east) .. controls+(right:7mm) and +(up:1mm) ..
 (w.north);
\draw[<-,blue] (v.east) .. controls+(right:3.5mm) and +(up:2mm) ..
 (w.north);
\end{tikzpicture}
\caption{The separated graph $(E(2,3),C(2,3))$}
\label{fig:m,nsepargraph}
\end{figure}

We now proceed to describe several concrete examples of separated
graphs and their associated algebras. Curiously enough, these
algebras (or the dynamical systems underlying them) have appeared
before in the literature in different contexts. In many cases, one
has to look at an appropriate full corner of the algebras to find
the significant construction.

\begin{example}
\label{exam:purerefine} Our first example is the ``pure refinement"
example, given in Figure \ref{fig:pureref}. Let $(E,C)$ be the
separated graph described by that picture, with $C_v=\{ X, Y\}$ and
$X=\{ \alpha _1,\alpha_2\}$, $Y=\{ \beta _1,\beta _2 \}$. The corner
$vL_K(E,C)v$ is isomorphic to a free product $K^2*_KK^2$. The corner
of the abelianized Leavitt algebra $\Lab _K(E,C)$ is isomorphic to
$K^4$. We thus see a drastic reduction of the complexity in the
transition from $L(E,C)$ to $\Lab (E,C)$. A similar statement holds
for the C*-algebras $C^*(E,C)$ and $\mathcal O (E,C)$.

\begin{center}{
\begin{figure}[htb]
\begin{tikzpicture}[scale=2]
\node (v) at (1.5,1)  {$v$};
 \node (w_1) at (0,0) {$w_1$};
 \node (w_2) at (1,0) {$w_2$};
 \node (w_3) at (2,0) {$w_3$};
 \node (w_4) at (3,0) {$w_4$};
 \draw[<-,red]  (v.west) ..  node[above]{$\alpha_2$} controls+(left:3mm) and +(up:3mm) ..
 (w_1.north) ;
 \draw[<-,red] (v.south) .. node[below, left]{$\alpha_1$}  controls+(left:2mm) and +(up:2mm) ..
 (w_2.north);
\draw[<-,blue] (v.south) .. node[below, right]{$\beta_1$}
controls+(right:2mm) and +(up:2mm) ..
 (w_3.north);
\draw[<-,blue] (v.east) .. node[above]{$\beta_2$}
controls+(right:3mm) and +(up:3mm) ..
 (w_4.north);
\end{tikzpicture}
\caption{The separated graph of pure refinement}
\label{fig:pureref}
\end{figure}
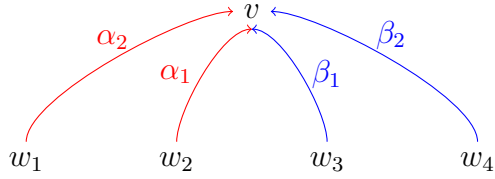}
\end{center}

\end{example}

\begin{example}
\label{exam:trussexam} Let $(E,C)$ be the separated graph described
in Figure \ref{fig:trussexample}, with $C_v=\{ X, Y\}$ and $X=\{
\alpha _0,\alpha_1\}$ and $Y=\{ \beta _0,\beta _1 \}$.

\begin{center}{
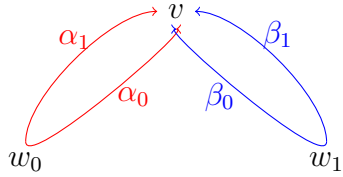
\begin{figure}[htb]
\begin{tikzpicture}[scale=2]
\node (v) at (1,1)  {$v$}; \node (w_0) at (0,0) {$w_0$};
 \node (w_1) at (2,0) {$w_1$};
 \draw[<-,red]  (v.west) ..  node[above]{$\alpha_1$} controls+(left:3mm) and +(up:3mm) ..
 (w_0.north) ;
 \draw[<-,red] (v.south) .. node[below, right]{$\alpha_0$}  controls+(right:1mm) and +(down:2mm) ..
 (w_0.north);
\draw[<-,blue] (v.south) .. node[below, left]{$\beta_0$}
controls+(left:1mm) and +(down:2mm) ..
 (w_1.north);
\draw[<-,blue] (v.east) .. node[above]{$\beta_1$}
controls+(right:3mm) and +(up:3mm) ..
 (w_1.north);
\end{tikzpicture}
\caption{The separated graph corresponding to Truss example}
\label{fig:trussexample}
\end{figure}}
\end{center}

The Leavitt path algebra $L_K(E,C)$ is Morita equivalent to the
corner $vL(E,C)v$, which is isomorphic to $M_2(K)*_K M_2(K)$.

By Corollary \ref{cor:Lab+OECarecrossedprods}, $\Lab _K(E,C)\cong
C_K(\Omega (E,C))\rtimes \mathbb F_4$, where $\Omega :=\Omega (E,C)$
is a 0-dimensional compact space admitting a decomposition $\Omega =
X\sqcup Y\sqcup Z$ into clopen subsets, such that $Z$ decomposes in
two different ways as a disjoint union of clopen subsets
$Z=H_0\sqcup H_1 =V_0\sqcup V_1$, together with homeomorphisms
$\mathfrak h_i\colon X\to H_i$ for $i=0,1$, and $\mathfrak v_j\colon
Y\to V_j$ for $j=0,1$. This construction is closely related to Truss
example in \cite{Truss}, which we briefly describe now. Let
$\mathcal E$ be the non-directed graph with vertices $\{ v_i \mid
i\in \Z \}$ and with an edge connecting a vertex $v_{i}$ with
$v_{i+1}$, for all $i\in \Z$. Let $T$ be the set of all the
``labelings" of the edges and vertices of $\mathcal E$ satisfying
the following conditions:
\begin{enumerate}
\item[(i)] the vertices are labelled $0$ or $1$, and adjacent
vertices have different labels,
\item[(ii)] the edges are labelled by pairs $(i,j)$ where $i,j\in
\{0,1\}$,
\item[(iii)] if $(i,j), (k,l)$ are the labels on the edges incident
with a vertex labelled $0$ then $i\ne k$.
\item[(iv)] if $(i,j),(k,l)$ are the labels on the edges incident
with a vertex labelled $1$ then $j\ne l$.
\end{enumerate}
Let
$$X=\{ t\in T \mid  v_0 \text{ is labelled } 0 \text{ and } (v_0,v_1) \text{ is labelled } (0,*) \}. $$
$$Y=\{ t\in T \mid  v_0 \text{ is labelled } 1 \text{ and } (v_0,v_1) \text{ is labelled } (*,1) \}. $$

Observe that $(v_{-1},v_0)$ is labelled $(1,*)$ (respectively
$(*,0)$) for every $t$ in $X$ (respectively $Y$).

Let $T'$ be a disjoint copy of $T$, with vertices $v_i'$ and with
the same conditions on the labels of vertices and edges. Let $\iota
\colon T \to T'$ be the canonical map. Set
$$Z= \{t'\in T'\mid v_0' \text{ is labelled } 0 \}.$$

Consider the shift map $S\colon T\to T$ defined by the following
conditions: $S(t)$ is the labelled graph such that the label of
$S(t)$ at $v_i$ is the label of $t$ at $v_{i-1}$, and the label of
$S(t)$ at $(v_{i},v_{i+1})$ is the same as the label of $t$ at
$(v_{i-1},v_{i})$. Let $S'= \iota \circ S \circ \iota ^{-1}$ be the
corresponding bijection of $T'$.

We also need the involution $\phi \colon T\to T$, which sends any
$t\in T$ to the labelled graph $\phi (t)$ defined by the condition
that the bijection $v_i\to v_{-i}$ on the set of vertices $\{ v_i
\mid i\in \Z \}$ is an isomorphism of labelled graphs between $t$
and $\phi (t)$.

We first indicate the connection between the definition using
$E$-functions (see Theorem \ref{thm:Efunctions}) and the description
above. Consider an $E$-function $(\Omega _1,f_1), (\Omega
_2,f_2),\dots $, where $\Omega _1= s^{-1}(w_0)=\{ \alpha _0,\alpha
_1 \}$. This $E$-function will correspond, upon identification of
the labels $(i,j)$ with the labels $(\alpha _i, \beta_j)$, to an
element $t$ of $X$, as follows. The function $f_1$ takes values on
$\{ \beta_0,\beta_1 \}$ so we define the label at the edge
$(v_0,v_1)$ as $(\alpha_0,f_1(\alpha _0))$ and the label at the edge
$(v_{-1},v_0)$ as $(\alpha_1, f_1(\alpha_1))$. Now the label at the
edge $(v_1,v_2)$ is defined by $(f_2( \ol{f_1(\alpha _0)}f_1(\alpha
_0)^{-1} \alpha _0), \ol{f_1(\alpha _0)})$, where we set
$\ol{\alpha_i}=\alpha _{1-i}$ and $\ol{\beta _j}=\beta _{1-j}$.
Similarly, the label at the edge $(v_{-2},v_{-1})$ is defined by
$(f_2(\ol{f_1(\alpha_1)}f_1(\alpha_1)^{-1} \alpha_1), \ol{f_1(\alpha
_1)} ) $. Continuing in this way, we define the labels at all the
edges $(v_i,v_{i+1})$. The labels at vertices are defined so as to
make $t$ an element of $X$. It is easy to see that this procedure
establish a bijection between the set of $E$-functions $(\Omega
_1,f_1), (\Omega _2,f_2),\dots $ with $\Omega _1= s^{-1}(w_0)$ and
$X$. Similarly there is a bijection between the set of $E$-functions
$(\Omega _1,f_1), (\Omega _2,f_2),\dots $ with $\Omega _1=
s^{-1}(w_1)$ and $Y$. For each $r\ge 1$, the intersection of an element $\xi \in X$ with the ball $B_{2r}$
in the Cayley graph of $\mathbb F_4$ has exactly two vertices: the one obtained by going in the ``positive direction'',
and the one obtained by going in the ``negative direction''. 

We are now in a position to describe the maps $\mathfrak
h_i,\mathfrak v_j$ in our picture of $\Omega$. Set
$$H_0=\{ t'\in  Z \mid  (v_0',v_1') \text{ is labelled } (0,*)
\} ,\quad H_1=\{ t'\in  Z \mid   (v_0',v_1') \text{ is labelled }
(1,*) \} ,$$
$$V_0=\{ t'\in  Z \mid  (v_0',v_1') \text{ is labelled } (*,0)
\} ,\quad V_1=\{ t'\in  Z \mid   (v_0',v_1') \text{ is labelled }
(*,1) \} ,$$

Then $Z=H_0\sqcup H_1= V_0\sqcup V_1$. Define homeomorphisms
$\mathfrak h_0\colon X\to H_0$ and $\mathfrak h_1\colon X\to H_1$ by
$$\mathfrak h_0 = \iota |_{X},\qquad  \text{ and } \qquad \mathfrak h_1= (\iota \circ
\phi)|_X.$$ Define homeomorphisms $\mathfrak v_0\colon Y\to V_0$ and
$\mathfrak v_1\colon Y\to V_1$ by
$$\mathfrak v_0 = (S'\circ \iota) |_Y,\qquad  \text{ and } \qquad \mathfrak v_1= (S'\circ \iota
\circ \phi)|_Y.$$
Observe that 
$$2[X]= [\mathfrak h _0(X)]+[\mathfrak h _1(X)] = [Z] = [\mathfrak v_0(Y)]+[\mathfrak v_1(Y)]= 2[Y]$$
in $S(\Omega , \mathbb F_4, \mathbb E)$.
Truss showed in \cite{Truss} that $[X]\ne [Y]$ even in the semigroup
$S(\Omega , \mathbb F_4, \mathbb B)$, where $\mathbb B$ denotes the
subalgebra of $\mathcal P (\Omega )$ consisting of the Borel subsets
of $\Omega $. Since $M(E,C)=\langle a,b\mid 2a=2b \rangle $, the
proof of Theorem \ref{thm:main-typesemigroup} gives that $[X]\ne
[Y]$ in $S(\Omega , \mathbb F_4, \mathbb E)$.
\end{example}

\begin{example}
\label{exam:alggenepi} Let $(E,C)$ be the separated graph described
in Figure \ref{fig:partiisometry}, with $C_v=\{ X, Y\}$ and $X=\{
\alpha _1,\alpha_2\}$ and $Y=\{ \beta _1,\beta _2 \}$. By
\cite[Lemma 5.5(a)]{Aone-rel}, the corner $vC^*(E,C)v$ is the unital
universal C*-algebra generated by a partial isometry. See \cite{BN}
for a recent study of the {\it non-unital} universal C*-algebra
generated by a partial isometry (of which $vC^*(E,C)v$ is the
unitization).

\begin{center}{
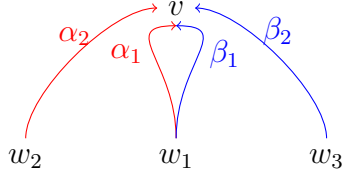
\begin{figure}[htb]
\begin{tikzpicture}[scale=2]
\node (v) at (1,1)  {$v$};
 \node (w_1) at (1,0) {$w_1$};
 \node (w_2) at (0,0) {$w_2$};
 \node (w_3) at (2,0) {$w_3$};
 \draw[<-,red]  (v.west) ..  node[above]{$\alpha_2$} controls+(left:3mm) and +(up:3mm) ..
 (w_2.north) ;
 \draw[<-,red] (v.south) .. node[below, left]{$\alpha_1$}  controls+(left:4mm) and +(up:5mm) ..
 (w_1.north);
\draw[<-,blue] (v.south) .. node[below, right]{$\beta_1$}
controls+(right:4mm) and +(up:5mm) ..
 (w_1.north);
\draw[<-,blue] (v.east) .. node[above]{$\beta_2$}
controls+(right:3mm) and +(up:3mm) ..
 (w_3.north);
\end{tikzpicture}
\caption{The separated graph of a partial isometry}
\label{fig:partiisometry}
\end{figure}}
\end{center}

The C*-algebra $v\mathcal O (E,C) v$ is the C*-algebra of the {\it
free monogenic inverse semigroup}, which was studied, among other
places, in \cite{HR}. The $0$-dimensional compact space $\Omega
(E,C)$ decomposes as $\Omega (E,C)= X_v\sqcup \bigsqcup_{i=1}^3
Y_{w_i}$, and we have two decompositions $X_v= H_{\beta _1}\sqcup
H_{\beta _2}= H_{\alpha_1}\sqcup H_{\alpha_2}$ in clopen sets, with
a universal homeomorphism $\alpha:=  \theta _{\alpha_1}\circ
\theta_{\beta_1}^{-1}\colon H_{\beta_1}\to H_{\alpha _1}$, in the
sense that given any other homeomorphism $\beta \colon X'_1\to
Z_1'$, where $X_1'$ and $Z_1'$ are clopen subsets of a compact
Hausdorff space $X'$, there exists a unique equivariant continuous
map from $X'$ to $X_v$ (cf. Corollary \ref{cor:univECdynsustem}).
The picture of the space $X_v$, together with the action of $\alpha
$ on it is shown in Figure \ref{Fig:spaceXv}. Observe that the sets
$H_{\beta _2}$ and $H_{\alpha _2}$ are precisely the sets of points
at the bottom row and right-most column, respectively. The unique
fixed point of $\alpha$ is sited at the top left corner.

\begin{figure}[htb]
 \[
 {
 \xymatrix{
 & \bullet \ar@{.>}[d] &  \ar@{.>}[l]  & \cdots \cdots   &   \bullet\ar[l] & \bullet \ar[l] & \bullet \ar[l]  \\
 &  \ar@{--}[d]    &  \cdots  \cdots &  \cdots &  \cdots  &  \cdots &  \\
& \bullet \ar[d] & \cdots   \cdots  & \cdots    \cdots  &  \cdots \cdots & \bullet \ar[ld] & \bullet \ar[ld] \\
& \bullet \ar[d] &  \cdots   \cdots &  \cdots \cdots  &  \bullet  \ar[ld]  & \bullet \ar[ld] & \bullet \ar[ld] \\
 & \bullet & \cdots \cdots &  \bullet  & \bullet & \bullet & \bullet
   }
 }
 \]
\caption{The compact space $X_v$.} \label{Fig:spaceXv}
\end{figure}
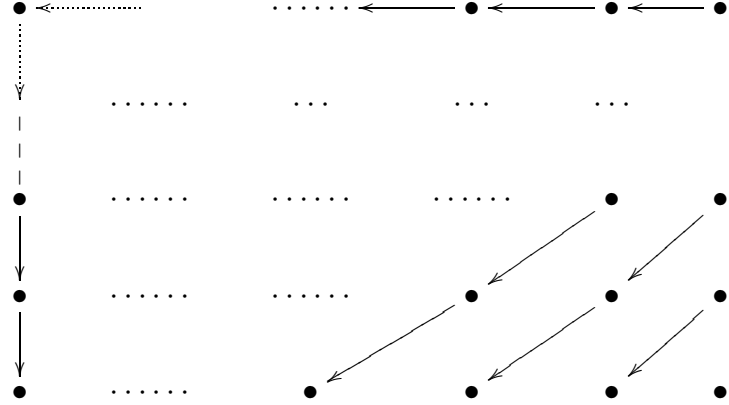

It is not hard to show, using the universal properties of
$C(X_v)\rtimes _{\alpha^*}\Z$ and $\mathcal O (E,C)$, that there
is a canonical isomorphism 
$$C(X_v)\rtimes _{\alpha^*}\Z \overset{\cong}{\longrightarrow}   v\mathcal O (E,C)v $$
which is the identity on $C(X_v)$ and sends the canonical partial isometry $u_1$
to the partial isometry $u_{\alpha_1\beta_1^{-1}}$ in $v\mathcal O (E,C)v $.
It can be shown that this partial crossed product description of the C*-algebra of the free monogenic inverse semigroup 
coincides with the one obtained by applying \cite[Theorem 6.5]{ExelPAMS98} to the group 
$G=\mathbb Z$.

\end{example}

\begin{example}
\label{exam:lamplighter} Let $(F,D)$ be the separated graph
described in Figure \ref{fig:lampgroup}, with $C_v=\{ X, Y\}$ and
$X=\{ \alpha _1,\alpha_2\}$ and $Y=\{ \beta _1,\beta _2 \}$. By
\cite[Lemma 5.5(b)]{Aone-rel}, we have
$$vC^*(E,C)v\cong C^*( (\bgast _{\Z}\Z_2)\rtimes \Z) \,$$
where $\Z$ acts on $\bgast_{\Z} \Z_2$ by shifting the factors of the
free product. Similar computations give that $vL_K(E,C)v$ is
isomorphic to the group algebra $K[(\bgast _{\Z}\Z_2)\rtimes \Z]$.

\begin{center}{
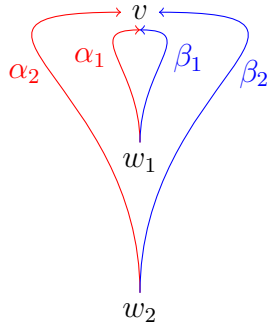
\begin{figure}[htb]
\begin{tikzpicture}[scale=2]
\node (v) at (0,1)  {$v$};
 \node (w_1) at (0,0) {$w_1$};
 \node (w_2) at (0,-1) {$w_2$};
 \draw[<-,red]  (v.south) ..  node[below, left]{$\alpha_1$} controls+(left:4mm) and +(up:5mm) ..
 (w_1.north) ;
 \draw[<-,red] (v.west) .. node[below, left]{$\alpha_2$}  controls+(left:14mm) and +(up:14mm) ..
 (w_2.north);
\draw[<-,blue] (v.east) .. node[below, right]{$\beta_2$}
controls+(right:14mm) and +(up:14mm) ..
 (w_2.north);
\draw[<-,blue] (v.south) .. node[below, right]{$\beta_1$}
controls+(right:4mm) and +(up:5mm) ..
 (w_1.north);
\end{tikzpicture}
\caption{The separated graph underlying the lamplighter group}
\label{fig:lampgroup}
\end{figure}}
\end{center}

It is easy to show that
$$v\mathcal O (F, D)v\cong C^*(\Z_2\wr \Z), \qquad v\Lab _K(F,D)v\cong K[\Z_2\wr \Z
] , $$ where $\Z_2 \wr Z$ is the wreath product $(\oplus _{\Z}
\Z_2)\rtimes \Z$. This is a well-known group, called the lamplighter
group. This group provided the first counter-example to the Strong
Atiyah's Conjecture, see \cite{DS}, \cite{GLSZ}. Observe that
$\Omega (E,C)=X_v\sqcup Y_{w_1}\sqcup Y_{w_2}$, with $X_v\cong \prod
_{\Z} \Z_2$ being the {\it Pontrjagin dual} of $\oplus _{\Z} \Z_2$.
Here we have two decompositions $X_v= H_{\beta _1}\sqcup H_{\beta
_2}= H_{\alpha_1}\sqcup H_{\alpha_2}$ in clopen sets, and a
universal homeomorphism $\alpha $ of $X_v$ sending $H_{\beta_i}$ to
$H_{\alpha_i}$, for $i=1,2$. The action of $\alpha $ on $X_v$ is
essentially the dual action of the action of $\Z$ on $\oplus _\Z
\Z_2$. We obtain the well-known representation of the group algebra
of the lamplighter group:
$$K[\Z_2\wr \Z]\cong v\Lab _K(F,D) v\cong C_K\Big( \prod _\Z \Z_2 \Big) \rtimes
\Z.$$ The group algebras $K[\Z_p\wr \Z]$, $p\ge 3$, can be similarly
represented, using a bipartite separated graph with $p+1$ vertices
$v,w_1,\dots ,w_p$, with $E^{0,0}=\{ v\}$, $E^{0,1}=\{w_1,\dots ,w_p
\}$, and $2p$ arrows $\{\alpha _i,\beta _i\mid i=1,\dots ,p \}$,
with $s(\alpha _i)=s(\beta _i)=w_i$ and $C_v=\{ X, Y\}$, for
$X=\{\alpha _1,\dots ,\alpha _p \}$, $Y=\{ \beta _1,\dots ,\beta
_p\}$.
\end{example}

\section{Topologically free actions}
\label{sect:topfree}

In this final section we characterize the finite bipartite separated
graphs $(E,C)$ such that the canonical action of the free group
$\mathbb F$ on the universal space $\Omega (E,C)$ is topologically
free, and we obtain consequences for the algebraic structure of the
reduced C*-algebra $\mathcal O^r (E,C)$ (see Definition
\ref{def:reducedO(E,C)}) and of the abelianized Leavitt algebra
$\Lab (E,C)$.

Recall the following definition from \cite{ELQ}.

\begin{definition}
 \label{def:TopFree}
  Let $\theta $ be a partial action of a group $G$ on a compact Hausdorff space
$X$. The partial action $\theta$ is {\it topologically free} if for
every $t\in G\setminus \{1\}$, the set $F_t:=\{ x\in U_{t^{-1}}\mid
\theta _t(x)=x\}$ has empty interior.
\end{definition}

We now give our particular version of condition (L).

We will use paths in the \emph{double} of $E$, that is, the graph
$\Ehat$ obtained from $E$ by adjoining, for each $e\in E^1$, an edge
$e^*$ going in the reverse direction of $e$, that is $s(e^*)=r(e)$
and $r(e^*)=s(e)$. Given a path
$$\gamma = e_{2r}^* e_{2r-1} \cdots e_4^*e_3e_2^*e_1 ,$$
in $\Ehat$, we say that $\gamma $ is an {\it admissible path} in
case $X_{e_{2i-1}}\ne X_{e_{2i}}$ for $i=1,2\dots ,r$, that is, the
edges $e_{2i-1}$ and $e_{2i}$ belong to different elements of
$C_{r(e_{2i-1})}=C_{r(e_{2i})}$ for all $i=1,2\dots ,r$, and
$e_{2i+1}\ne e_{2i}$ for $i=1,\dots ,r-1$. The same definition
applies to paths of the forms $ e_{2r-1} \cdots e_4^*e_3e_2^*e_1$,
or $e_{2r}^* e_{2r-1} \cdots e_4^*e_3e_2^* $, or $e_{2r-1} \cdots
e_4^*e_3e_2^* $.

\begin{definition}
\label{def:cycles} Let $(E,C)$ be a finite bipartite separated
graph, with $s(E^1)=E^{0,1}$ and $r(E^1)=E^{0,0}$. A {\it closed
path} in $(E,C)$ is a non-trivial admissible path $\gamma =
e_{2r}^*e_{2r-1}\cdots e_2^*e_1$ such that $s(e_1)=s(e_{2r})$.

A {\it cycle} in $(E,C)$ is a closed path $\gamma =
e_{2r}^*e_{2r-1}\cdots e_2^*e_1$ in $(E,C)$ such that $e_1\ne
e_{2r}$ and $s(e_{2i-1})\ne s(e_{2j})$ for all $1\le i\le j< r$.

Let $\gamma $ be a closed path in $(E,C)$. An {\it entry} of $\gamma
$ is a non-trivial admissible path of the form
$x_{2t-1}x_{2t-2}^*\cdots x_4^*x_3x_2^*x_1$ such that
$s(x_1)=s(e_1)$, and $|X|\ge 2$ for some $X\in C_{r(x_{2t-1})}$ such
that $X\ne X_{x_{2t-1}}$.

We say that $(E,C)$ {\it satisfies condition (L)} if every cycle has
an entry.
\end{definition}

Note that condition (L) is very common. Indeed, for the separated
graphs $(E,C)$ associated to one-relator monoids
$$\langle a_1,a_2,\dots ,a_n \mid \sum_{i=1}^n r_ia_i = \sum
_{i=1}^n s_ia_i \rangle $$ with $r_i,s_i$ non-negative integers such
that $r_i+s_i>0$ for all $i$, $\sum _i r_i >0$,  and $\sum _i s_i
>0$ (see \cite{Aone-rel}), the only one that does not satisfy
condition (L) is the one corresponding to the presentation $\langle
a \mid a=a \rangle$. Its separated graph is the $(E(1,1), C(1,1))$.
Obviously, the cycle $\beta^*\alpha$ does not have entries.

We need a couple of preparatory lemmas.

\begin{lemma}
\label{lem:well-definedexit} Let $\gamma = e_{2r}^*e_{2r-1}\cdots
e_2^*e_1$ be a closed path in $(E,C)$, and assume that, for some
$1\le i< r$  there exists an admissible path $\nu
=x_{2t-1}x_{2t}^*\cdots x_4^*x_3x_2^*x_1 $ such that
$s(x_1)=s(e_{2i})$ and $|X|\ge 2$ for some $X\in C_{r(x_{2t-1})}$
such that $X\ne X_{x_{2t-1}}$. Then there exists an entry for
$\gamma $.
\end{lemma}

\begin{proof}
If $x_1\ne e_{2i}$ then $\nu e_{2i}^* e_{2i-1} \cdots e_2^*e_1$ is
an entry for $\gamma $.

If $t\ge i+1$ and  $x_1=e_{2i}, x_2=e_{2i-1},\dots  , x_{2i}=e_{1}$,
then $x_{2t-1}x_{2t-2}^*\cdots x_{2i+2}^*x_{2i+1}$ is an entry for
$\gamma$.

 Suppose that $x_1=e_{2i}, x_2=e_{2i-1},\dots  ,
x_{2k}=e_{2i-2k+1}$, but $x_{2k+1}\ne e_{2i-2k}$ for some $k< i$.
Then
$$x_{2t-1} x_{2t-2}^* \cdots x_{2k+1} e_{2i-2k}^* e_{2i-2k-1} \cdots
e_2^*e_1 $$ is an entry for $\gamma $.

Suppose that
 $x_1=e_{2i}, x_2=e_{2i-1},\dots  , x_{2k}=e_{2i-2k+1}$,
$x_{2k+1}= e_{2i-2k}$, for some $k< i$, and that either $2t-1=2k+1$
or $x_{2k+2}\ne e_{2i-2k-1}$.

If $2k+1=2t-1$ then $$\nu = x_{2k+1}x_{2k}^* \cdots x_2^*x_1=
e_{2i-2k}e_{2i-2k+1}^*\cdots e_{2i-1}^*e_{2i}.$$ There exists $X\in
C_{r(e_{2i-2k-1})}$ such that $|X|\ge 2$ and $e_{2i-2k}\notin X$. We
distinguish two cases:

(a) If $e_{2i-2k-1}\notin X$, then $e_{2i-2k-1}e_{2i-2k-2}^* \cdots
e_2^*e_1$ is an entry for $\gamma $.

(b) If $e_{2i-2k-1}\in X$, then $e_{2i-2k} e_{2i-2k+1}^*\cdots
e_{2r-1}^*e_{2r}$ is an entry for $\gamma $.

Assume now that $x_{2k+2}\ne e_{2i-2k-1}$, and write
$$\nu= \nu' x_{2k+1} x_{2k}^* \cdots x_2^*x_1 $$
with $\nu '$ of positive length. Again we distinguish two cases:

(a) $x_{2k+2}$ and $e_{2i-2k-1}$ belong to different sets of
$C_{r(e_{2i-2k-1})}$. In this case the admissible path
$$\nu ' x_{2k+2}^* e_{2i-2k-1} e_{2i-2k-2}^* \cdots e_2^*e_1$$
is an entry for $\gamma $.

(b) $x_{2k+2}$ and $e_{2i-2k-1}$ belong to $X\in
C_{r(e_{2i-2k-1})}$. In particular, we have $|X|\ge 2$, because we
are assuming that $x_{2k+2}\ne e_{2i-2k-1}$. It follows that
$$e_{2i-2k} e_{2i-2k+1}^* \cdots e_{2r-1}^* e_{2r}$$
is an entry for $\gamma $.

\end{proof}

\begin{lemma}
\label{lem:exitsforCPs} Let $\gamma = e_{2r}^*e_{2r-1}\cdots
e_2^*e_1$ be a closed path in $(E,C)$, and suppose that $(E,C)$
satisfies condition (L). Then the path $\gamma $ has an entry.
\end{lemma}

\begin{proof}
Suppose that there are $1\le i \le j <r$ such that
$s(e_{2i-1})=s(e_{2j})$. Taking such a pair $(i,j)$ with $j-i$
minimal, we have that $s(e_{2k-1})\ne s(e_{2l})$ for all $k,l$ such
that $i\le k\le l< j$. Using Lemma \ref{lem:well-definedexit}, and
changing notation, we can thus assume that $s(e_{2i-1})\ne
s(e_{2j})$ for $1\le i\le j <r$.

If in addition $e_1\ne e_{2r}$, then $\gamma $ is a cycle and so
there is an entry for $\gamma $ by condition (L).

\smallskip

Suppose that $e_1=e_{2r}$. Note that this implies $r>1$ because
$\gamma $ is an admissible path. We claim that $e_{2r-1}\ne e_2$. If
$r=2$ then $e_2\ne e_3=e_{2r-1}$. If $r\ge 3$, then
$$s(e_2) =s(e_3) \ne s(e_{2(r-1)})=s(e_{2r-1})$$
and consequently $e_2\ne e_{2r-1}$. Therefore $e_2\ne e_{2r-1}$ in
any case. Write $X=X_{e_{2r-1}}$, $Y=X_{e_1}=X_{e_{2r}}$,
$Z=X_{e_2}$. Then $X\ne Y$ and $Y\ne Z$, and $X,Y,Z\in C_{r(e_1)}$.
There are two cases: $X= Z$ and $X\ne Z$. If $X= Z$, then since
$e_2,e_{2r-1}\in X$ and $e_2\ne e_{2r-1}$ we have $|X| \ge 2$, and
thus $e_1$ is an entry for $\gamma $. If $X\ne Z$, then
$$     e_2^* e_{2r-1} e_{2r-2}^*e_{2r-3} \cdots   e_6^*e_5e_4^*e_3
$$
is a cycle, and consequently it has an entry $\eta$, because $(E,C)$
has condition (L). It follows from Lemma \ref{lem:well-definedexit}
that $\gamma $ has an entry.
\end{proof}

\begin{theorem}
\label{thm:top-free} Let $(E,C)$ be a finite bipartite separated
graph, with $s(E^1)=E^{0,1}$ and $r(E^1)=E^{0,0}$. Then the partial
action of $\mathbb F$ on $\Omega (E,C)$ is topologically free if and
only if $(E,C)$ satisfies condition (L).
\end{theorem}

\begin{proof}
Suppose first that $(E,C)$ does not satisfy condition (L), and let
$$\gamma =e_{2r}^*e_{2r-1}\cdots e_2^*e_1 $$
by a cycle without entries in $(E,C)$. Write
$$g= e_{2r}^{-1} e_{2r-1} \cdots e_2^{-1}e_1\in \mathbb F.$$
Note that $g\ne 1$, because the defining expression of $g$ is the
reduced form of $g$ (with respect to the canonical generating set
$E^1$).  Observe that it follows from the condition that the cycle
$\gamma $ does not have entries that there is a unique $E$-function
$(\Omega _1,f_1), (\Omega _2,f_2),\dots  $ such that $\Omega _1=
s(e_1)$. Let $\xi$ be the point of $\Omega (E,C)$ represented by
this $E$-function (see Theorem \ref{thm:Efunctions}). Then we have
$\Omega (E,C)_{s(e_1)}=\{\xi \}$.

Since $\gamma $ is a cycle, it follows that $\theta _g(\xi)= \xi$.
Therefore $\Omega (E,C)_{s(e_1)}= \{\xi \}$ is a clopen subset of
$\Omega (E,C)$, fixed by the non-trivial element $g$ of $\mathbb F$.

Therefore $\Omega (E,C)_{s(e_1)}= \{\xi \}$ is a clopen subset of
$\Omega (E,C)$, fixed by the non-trivial element $g$ of $\mathbb F$.

Conversely, suppose that $(E,C)$ satisfies condition (L). Assume
first that
$$g= z_r^{-1}x_rz_{r-1}^{-1} x_{r-1}\cdots z_1^{-1}x_1 $$
is a reduced word in $\mathbb F\setminus \{1\}$. We may assume that
$\text{Dom}(\theta _g)\ne \emptyset$, and this implies by Lemma
\ref{lem:actionForFuncts} that
$$\gamma =  z_r^*x_rz_{r-1}^* x_{r-1}\cdots z_1^*x_1 $$
is an admissible path in $(E,C)$. We may also assume that $\theta_g$
has some fixed point. So let $\xi$ be a fixed point of $\theta_g$,
with associated $E$-function $\mathfrak f =(\Omega _1,f_1),(\Omega
_2,f_2) ,\dots $. By Lemma \ref{lem:actionForFuncts}, we have
$\Omega _1= s^{-1}(s (x_1))$. Let $\mathfrak f '= ((\Omega
_1',f_1'),(\Omega _2',f_2'),\dots )$ be the $E$-function
corresponding to $\theta_g(\xi)$. Since $\theta _g(\xi)= \xi$, we
must have $\Omega_i'=\Omega _i$ and $f_i'=f_i$ for all $i$. In
particular, since $\Omega _1'=s^{-1}(s(z_r))$, we must have
$s(x_1)=s(z_r)$. It follows that $\gamma $ is a closed path in
$(E,C)$.

Suppose first that $x_1\ne z_r$. Note that then for every $i\ge 1$, 
$$\gamma ^i =z_r^*x_r\cdots z_1^*x_1z_r^*x_r\cdots z_2^*x_2z_1^*x_1
$$
is a closed path in $(E,C)$. Let $W$ be an open neighborhood of
$\xi$. By Theorem \ref{thm:Efunctions}, there exists a positive
integer $k\ge r$ such that $W$ contains the clopen neighborhood $U$
of $\xi$ defined by the partial $E$-function $(\Omega _1,f_1),\dots,
(\Omega _k,f_k)$. We will show that $U$ contains a point $\xi '$
such that $\theta _g(\xi ')\ne \xi '$. By Lemma
\ref{lem:exitsforCPs}, there exists an entry
$$\eta = y_{2t-1}y_{2t-2}^*y_{2t-1}\cdots  y_2^*y_1$$
for $\gamma $. By definition, this means that $\eta $ is an
admissible path such that $s(y_1)=s(x_1)$ and such that there is
$X\in C_{r(y_{2t-1})}$ such that $|X|\ge 2$ and $y_{2t-1}\notin X$.
We may assume that $t$ is minimal with this property, so that for
every $1\le t'<  t$, we have $|X'| =1$ for all $X'\in
C_{r(y_{2t'-1})}$ such that $y_{2t'-1}\notin X'$. Observe that this
implies that $\eta \in \Omega _{t}$, because for every $X'\in
C_{r(y_{2t'-1})}$ such that $X_{y_{2t'-1}}\ne X'$ we have that
$\pi_{X'}f_{t'}(y_{2t'-1}y_{2t'-2}^{-1}\cdots y_2^{-1}y_1)$ is the
unique element of $X'$, and so in particular we have that
$$y_{2t'}=\pi _{X_{y_{2t'}}} f_{t'}(y_{2t'-1}y_{2t'-2}^{-1}\cdots
y_2^{-1}y_1).$$ Write $g_1= y_{2t-1}y_{2t-2}^{-1}\cdots
y_2^{-1}y_1\in \mathbb F$.

 Now let $i$ be a positive integer such that
$ir+t> k$. If $x_1\ne y_1$, then
$$\eta (\gamma ^*)^{i}= y_{2t-1}y_{2t-2}^*\cdots y_2^*y_1x_1^*z_1 \cdots
x_r^*z_rx_1^*z_1 \cdots \,\, \, \cdots x_r^*z_r $$ is an admissible
path. Now observe that, by Lemma \ref{lem:actionForFuncts}, we have
$g_1g^{-i}\in \Omega _{ir+t}'=\Omega _{ir+t}$ and
$$f_{ir+t}(g_1g^{-i}) = f'_{ir+t}(g_1g^{-i})= f_{t}(g_1).$$
Now let $\xi'$ be the $E$-function $(\Omega _1'',f_1''), (\Omega
_2'',f_2''),\dots $ defined as follows. For $1\le j\le ir+t-1$, set
$\Omega _j''=\Omega _j$ and $f_j''=f_j$. Now let $X\in
C_{r(y_{2t-1})}$ such that $|X|\ge 2$ and $y_{2t-1}\notin X$. This
means that there is $x\in X$ such that
$$x\ne f_{t}(y_{2t-1}y_{2t-2}^{-1} \cdots y_2^{-1} y_1)= f_{t}(g_1).$$
We define
$$\pi_X f_{ir+t}''(g_1g^{-i}) = x ,$$
and we complete the definition of $f_{ir+t}''$ arbitrarily, so that
we get a partial $E$-function $(\Omega _1'', f_1''),\dots (\Omega''
_{ir+t}, f''_{ir+t})$. Extend this partial $E$-function to an
$E$-function $(\Omega _1'', f_1''),(\Omega _2'', f_2''),\dots $ and
let $\xi '$ be the point in $\Omega (E,C)$ determined by this
$E$-function. Then $\xi ' \in U\subseteq W$, because $(\Omega
_j'',f_j'')=(\Omega _j, f_j)$ for $1\le j\le ir+t-1$ and $ir+t-1 \ge
k$. On the other hand, if $\theta _g(\xi ')=\xi '$, then it follows
from Lemma \ref{lem:actionForFuncts} that
$$x= f_{ir+t}''(g_1g^{-i}) = f_{t}''(g_1)=f_{t}(g_1)$$
which is a contradiction. Therefore $\theta _g(\xi ')\ne \xi'$.

Now assume that $x_1=y_1$. Then, since we are assuming that $x_1\ne
z_r$, we have that $y_1\ne z_r$ and consequently
$$\eta \gamma ^{i}= y_{2t-1}y_{2t-2}^*\cdots y_2^*y_1z_r^*x_r \cdots
z_1^*x_1z_r^*x_r \cdots \,\, \, \cdots z_1^*x_1 $$ is an admissible
path. By Lemma \ref{lem:actionForFuncts}, we have $f_{ir+t}(g_1g^i)=
f'_t(g_1)= f_{t}(g_1)$. So, proceeding as above, but now using the
element $g_1g^i$ instead of $g_1g^{-i}$, we obtain an element $\xi'$
in $U$ such that $\theta _g (\xi ')\ne \xi '$.

We now consider the case where $x_1=z_r$. In this case we have,
since $\theta _g(\xi ) =\xi$,
$$z_1= f_1(x_1) = f_1(z_r)=f_1'(z_r)= x_r.$$
Analogously, if in addition $x_2=z_{r-1}$, then $z_2=x_{r-1}$.
Proceeding in this way, we see that there is an admissible path $\nu
=z_i^*x_i\cdots z_1^*x_1$ such that
$$\gamma = \nu ^* \gamma ' \nu$$
and $\gamma '= z_{r-i}^*x_{r-i}\cdots z_{i+1}^*x_{i+1}$ is a closed
path in $(E,C)$ such that $x_{i+1}\ne z_{r-i}$. Write $g'=
z_{r-i}^{-1}x_{r-i}\cdots z_{i+1}^{-1}x_{i+1}$ and
$h=z_i^{-1}x_i\cdots z_1^{-1}x_1$, and observe that $\theta _h(\xi)$
is a fixed point of $\theta _{g'}$.

Let $W$ be an open neighborhood of $\xi$, which we may assume to be
contained in the domain of $\theta_h$. By the above, there is a
point $\xi '$ in the open neighborhood $\theta_h(W)$ of $\theta
_h(\xi)$ such that $\theta _{g'}(\xi') \ne \xi'$. Then
$\theta_{h^{-1}}(\xi ')\in W$ and $\theta _g(\theta
_{h^{-1}}(\xi'))\ne \theta _{h^{-1}}(\xi ')$.

This completes the proof of the result in the case where $g$ is a
reduced word of the form $z_r^{-1}x_rz_{r-1}^{-1} x_{r-1}\cdots
z_1^{-1}x_1 $. Now assume that
$$g=x_rz_{r-1}^{-1}x_{r-1}\cdots z_1^{-1}x_1 z_0^{-1}$$
is a reduced word in $\mathbb F$, with $x_1,\dots ,x_r,z_0,\dots
,z_{r-1}\in E^1$ and $r\ge 1$.
  Assume that $\theta _g (\xi) =\xi$ for all $\xi\in V$, where $V$ is a non-empty
  open subset of $\Omega (E,C)$ such that $V\subseteq \text{Dom}(\theta _g)$.
In particular this implies that $\text{Dom}(\theta _g)\ne
\emptyset$, so that $s(z_0)=s(x_1)$, $z_{r-1}^*x_{r-1}\cdots
z_1^*x_1$ is an admissible path, and $s(x_r)=s(z_{r-1})$. Since
$\theta _g$ has fixed points we have $r(z_0)=r(x_r)$. By the same
reason, either $x_r=z_0$ or $x_r$ and $z_0$ belong to different
elements of $C_{r(x_r)}$. Indeed, if $z_r\ne z_0$, and $x_r$ and
$z_0$ belong to $X\in C_{r(z_0)}$, then
$$V\subseteq \theta _{x_r}(\Omega(E,C)_{s(x_r)})\cap \theta
_{z_0}(\Omega (E,C)_{s(z_0)})=H_{x_r}\cap H_{z_0}=\emptyset ,$$
which is a contradiction.

Assume first that $x_r=z_0$. Then for $g'= z_{r-1}^*x_{r-1}\cdots
z_1^*x_1$, we have that $g'\ne 1$ and that the set of fixed points
of $\theta _{g'}$ contains the non-empty open set $\theta
_{z_0^{-1}}(V)$. So we arrive to a contradiction by the first part
of the proof.

If $z_0$ and $x_r$ belong to different elements of $C_{r(z_0)}$,
then set
  $$
  g'':=z_0^{-1}x_rz_{r-1}^{-1}x_{r-1}\cdots z_1^{-1}x_1.
  $$
Then the set of fixed points of $\theta_{g''}$ contains the
non-empty open set $\theta _{z_0^{-1}}(V)$, and we arrive again to a
contradiction. This concludes the proof of the theorem.
\end{proof}

\begin{remark}
\label{rem:cantorset}
Let $w\in E^{0,1}$. We may define an entry based on $w$ as a
non-trivial admissible path of the form $x_{2t-1}x_{2t-2}^*\cdots
x_4^*x_3x_2^*x_1$ such that $s(x_1)=w$, and $|X|\ge 2$ for some
$X\in C_{r(x_{2t-1})}$ such that $X\ne X_{x_{2t-1}}$. Every vertex
$w\in E^{0,1}$ such that there is no entry based on $w$ determines
an isolated point in $\Omega (E,C)$. As another possible source of
isolated points in $\Omega (E,C)$ we mention vertices $w\in E^{0,1}$
such that $|s^{-1}(w)|=1$. For instance, in Example
\ref{exam:alggenepi} there is a dense set of isolated points in the
space $X_v$, namely all the points (in the representation provided
by Figure \ref{Fig:spaceXv}) out of the union of the top row and the
left-most column. Observe that however all three vertices in
$E^{0,1}$ have entries based on them in that example. In many cases,
$\Omega (E,C)$ is a Cantor set, that is, a metrizable,
zero-dimensional compact set without isolated points.
\end{remark}

We can now prove an analogue to the uniqueness theorem for graph
C*-algebras with property (L). It applies to a {\it reduced} version
of $\mathcal O (E,C)$, defined as follows:

\begin{definition}
\label{def:reducedO(E,C)} Let $(E,C)$ be a finite bipartite
separated graph. Then we denote by $\mathcal O ^r (E,C)$ the
C*-algebra defined by the {\it reduced crossed product} of $C(\Omega
(E,C))$ by $\mathbb F$:
$$\mathcal O ^r(E,C)= (C(\Omega (E,C))\rtimes ^r_{\alpha} \mathbb F .$$
There exists a canonical {\it faithful} conditional expectation
$E\colon \mathcal O ^r(E,C)\to C(\Omega (E,C))$, see \cite[Section
2]{ExelAmena}.
\end{definition}

Recall the following standard definition.

\begin{definition}
\label{def:(SP)}
  A C*-algebra satisfies property (SP) (for {\it
small projections}) in case every nonzero hereditary C*-subalgebra
contains a nonzero projection. Equivalently, for every nonzero
positive element $a$ in $A$ there is $x\in A$ such that $x^*ax$ is a
nonzero projection.
\end{definition}

\begin{theorem}
\label{thm:SPproperty} Let $(E,C)$ be a finite bipartite separated
graph satisfying condition (L). Then the C*-algebra $\mathcal O^r
(E,C)$ satisfies property ${\rm (SP)}$. More precisely, given a
nonzero positive element $c$ in $\mathcal O^r(E,C)$, there is an
element $x\in \mathcal O^r(E,C)$ such that $x^*cx$ is a nonzero
projection in $C(\Omega (E,C))$. In particular every nonzero ideal
of $\mathcal O^r (E,C)$ contains a nonzero projection of $C(\Omega
(E,C))$.
\end{theorem}

\begin{proof}
Since $\Omega (E,C)$ is topologically free when $(E,C)$ satisfies
condition (L) (Theorem \ref{thm:top-free}), the same proof used in
\cite[Theorem 4.9]{AEK} (which is based on \cite[Proposition
2.4]{ELQ}) applies here.
\end{proof}

We can also derive a purely algebraic consequence. Note that the
consideration of crossed products enables us to give a basically
unified treatment of the purely algebraic case and the C*-case.

Recall that for a given field with involution $K$ and a compact
zero-dimensional topological space $X$, we denote by $C_K(X)$ the
*-algebra of continuous functions $f\colon X\to K$, where $K$ is given
the discrete topology. Note that $C_{\C}(X)$ is a dense *-subalgebra
of $C(X)$.

\begin{theorem}
\label{thm:AlgebraicSP} Let $(E,C)$ be a finite bipartite separated
graph satisfying condition (L). Then the *-algebra $\Lab (E,C)=
C_K(\Omega (E,C)) \rtimes _{\alpha} \mathbb F$ satisfies property
${\rm (SP)}$. More precisely, given a nonzero element $a$ in $\Lab
(E,C)$, there is a projection $h$ in $C_K(\Omega (E,C))$, a nonzero
element $\lambda \in K$  and an element $g\in \mathbb F$ such that
$$ha(e(g)\delta _g)h=\lambda h\delta _e.$$
In particular every nonzero ideal of $\Lab (E,C)$ contains a nonzero
projection of $C_K(\Omega (E,C))$.
\end{theorem}

\begin{proof}
If is sufficient to show the result for any crossed product
$\mathcal A= C_K(X)\rtimes _{\theta^*} G$, where $G$ is a discrete
group, $X$ is a zero-dimensional compact space, and $\theta $ is a
topologically free partial action of $G$ on $X$ such that $U_t$ is a
clopen subset of $X$ for all $t\in G$. For $g\in G$, we denote by
$e(g)$ the characteristic function of $U_g$, so that the ideals
$D_g$ associated to the induced action $\theta^*$ of $G$ on $C_K(X)$
are given by $D_g= e(g)C_K(X)$. Recall that, setting $\alpha
=\theta^*$, we have $\alpha_g (e(g^{-1})f)(\xi) =
(e(g^{-1})f)(\theta _{g^{-1}}(\xi))= f(\theta _{g^{-1}}(\xi ))$ for
$\xi \in U_t$ and $f\in C_K(X)$.

We will show the statement for $\mathcal A=C_K(X)\rtimes _{\theta^*}
G$ as above by using a slight  modification of the arguments
presented in \cite[Proposition 2.4]{ELQ}. Let $a$ be a nonzero
element of $\mathcal A$. Multiplying $a$ on the right by a suitable
element of the form $e(g)\delta _g$, for $g\in G$, we obtain an
element $c:=a(e(g)\delta _g)$ with $c_e\ne 0$, where $c=\sum _{t\in
T} c_t\delta _t$ for a finite subset $T$ of $G$. Note that $c_e=\sum
_{i=1}^n \lambda _i \chi _{W_i}$, for some non-zero scalars $\lambda
_i\in K$ and pairwise disjoint non-empty clopen subsets $W_i$ of
$X$. Setting $\lambda := \lambda _1$, and using the fact that $X$ is
zero-dimensional, we obtain as in the proof of \cite[Proposition
2.4]{ELQ} a clopen subset $V$ contained in $W_1$ such that, putting
$h:=\chi _{V}$, we have
$$h(c_t\delta _t)h =0, \qquad \text{for all } t\in T\setminus \{ e
\}.$$ Since $V\subseteq W_1$ we thus get $hch= hc_eh= \lambda
h\delta _e$, so that we obtain
$$ha (e(g) \delta _g) h= hch =\lambda h\delta _e\, ,$$
as desired.
\end{proof}

\begin{remark}
\label{rem:one-sidedideals} Observe that, in the conditions of
Theorem \ref{thm:AlgebraicSP}, we also have that every nonzero
one-sided ideal of $\Lab (E,C)$ contains a nonzero idempotent.
Indeed if $a$ is a nonzero element of $\Lab (E,C)$, and
$ha(e(g)\delta _g)h=\lambda h \delta_e$ for some nonzero projection
$h$ in $C_K(X)$, some $g\in G$ and some $\lambda \in K\setminus
\{0\}$, then $ a[\lambda ^{-1}(e(g)\delta_g)h]$ is a nonzero
idempotent belonging to the right ideal generated by $a$, and
analogously $[\lambda^{-1}(e(g)\delta _g)h]a$ is a nonzero
idempotent in the left ideal generated by $a$.
\end{remark}

\medskip

\section*{Acknowledgements}
It is a pleasure to thank Ken Goodearl and Takeshi Katsura for their
useful comments. We also thank the anonymous referee for carefully reading 
the paper and for providing many helpful suggestions.

\end{document}